\documentclass[11pt]{article}

\usepackage[T1]{fontenc}
\usepackage{lmodern}
\usepackage{microtype}
\usepackage[a4paper,margin=29mm]{geometry}
\usepackage{amsmath,amssymb,amsthm,mathtools,mathrsfs}
\usepackage{booktabs,tabularx,longtable,array}
\usepackage{subcaption}
\usepackage[dvipsnames]{xcolor}
\usepackage{enumitem}
\usepackage{graphicx}
\usepackage{tikz-cd}
\usepackage{fancyhdr}
\usepackage{lastpage}
\usepackage{float}
\usepackage{url}
\usepackage{xurl}
\usepackage[hidelinks]{hyperref}
\usepackage{cleveref}
\pdfmapfile{+euler.map}

\definecolor{statusblue}{HTML}{17365D}
\definecolor{statusgreen}{HTML}{1D6F42}
\definecolor{statusamber}{HTML}{8A5B00}
\definecolor{statusred}{HTML}{982B2B}
\definecolor{boxblue}{HTML}{EDF3F8}
\definecolor{boxred}{HTML}{F9EEEE}

\newcommand{\Z}{\mathbf Z}
\newcommand{\Q}{\mathbf Q}
\newcommand{\F}{\mathbf F}
\newcommand{\C}{\mathbf C}

\newcommand{\cO}{\mathcal O}

\newcommand{\Hom}{\operatorname{Hom}}
\newcommand{\End}{\operatorname{End}}
\newcommand{\im}{\operatorname{im}}

\theoremstyle{plain}
\newtheorem{theorem}{Theorem}[section]
\newtheorem{proposition}[theorem]{Proposition}
\newtheorem{lemma}[theorem]{Lemma}
\newtheorem{corollary}[theorem]{Corollary}
\theoremstyle{definition}
\newtheorem{definition}[theorem]{Definition}

\newtheorem{conjecture}[theorem]{Conjecture}

\theoremstyle{remark}
\newtheorem{remark}[theorem]{Remark}

\numberwithin{equation}{section}

\setlist[itemize]{leftmargin=6mm}
\setlist[enumerate]{leftmargin=7mm}
\allowdisplaybreaks

\title{\bfseries Prime-power congruences for level-one eigenforms}
\author{Nadim Rustom\thanks{%
  \begin{tabular}[t]{@{}l@{}}
    \href{mailto:rustom.nadim@proton.me}{\texttt{rustom.nadim@proton.me}}\\
    Independent Researcher
  \end{tabular}}}

\begin{document}
\maketitle

\begin{abstract} We study prime-power congruences for the prime-to-\(p\) Hecke eigensystems of normalized cuspidal level-one eigenforms as the weight varies, and we address the problem of classifying such systems for small prime powers. We construct Hecke-equivariant transfer maps between level-one modular-symbol modules via multiplication by suitable lifts of Dickson invariants; these maps reduce the problem to finite computation and propagate Hecke operator relations through all weights. For \(p=2,3,5,7\), we obtain all-weight classifications modulo \(2^8\), \(3^4\), \(5^3\), and \(7^2\), respectively.\end{abstract}

\tableofcontents

\part{General theory}

\section{Introduction and main results}

\subsection{The problem}

In this work, we study congruences modulo prime powers between systems of Hecke eigenvalues arising from normalized cuspidal eigenforms of level \(1\). We will use the adjective \emph{strong}, following \cite{CKW2013,KRW2016}, to refer to either eigenforms in characteristic $0$, or Hecke eigensystems arising from such eigenforms. More precisely, for a fixed prime power \(p^m\), we ask which strong prime-to-\(p\) Hecke eigensystems can occur in level $1$ as the weight varies, and we seek an explicit classification of these systems.

Questions of this kind belong to the general study of modular forms and Galois representations modulo prime powers. For fixed level \(N\), prime \(p\nmid N\), and integer \(m\geq1\), the author, in joint work with Kiming and Wiese, formulated the following finiteness conjecture \cite[Conjecture~1]{KRW2016}.

\begin{conjecture} \label{conj:krw-finiteness} There are only finitely many modulo \(p^m\) congruence classes of strong Hecke eigenforms on \(\Gamma_1(N)\). \end{conjecture}

Equivalently, there should exist a constant \(B(N,p,m) \), called a \emph{strong weight
bound}, such that every normalized characteristic-zero eigenform \(f\) on \(\Gamma_1(N)\) is congruent modulo \(p^m\) to a normalized eigenform \(g\) on \(\Gamma_1(N)\) of weight at most \(B(N,p,m)\). A corresponding finiteness statement in characteristic \(p\) is known: for fixed level, only finitely many systems of Hecke eigenvalues with values in \(\overline{\mathbf F}_p\) occur as the weight varies \cite{Jochnowitz1982}. 

The case of level \(1\) already exhibits much of the subtlety of the general problem. Classical results of Hatada \cite{Hatada1977a,Hatada1977b,Hatada1979,Hatada1981} established a series of striking congruences for Hecke eigenvalues. However, these results generally concern individual Hecke eigenvalues and do not determine whole eigensystems. We record a few representative examples. For every normalized level-one eigenform \(f\), Hatada proved, among other congruences, \[ a_\ell(f)\equiv 1+\ell \pmod{2^4\overline{\mathbf Z}_2} \qquad \bigl(\ell\equiv\pm1\pmod{2^3}\bigr), \] and \[ a_\ell(f)\equiv 1+\ell \pmod{2^5\overline{\mathbf Z}_2} \qquad \bigl(\ell\equiv\pm1\pmod{2^4}\bigr) \] \cite{Hatada1979,Hatada1981}. 

Coleman and Stein \cite[Conjecture~3.4]{ColemanStein2004} conjectured that exactly five systems of Hecke eigenvalues modulo \(3^2\) occur in levels \(1\), \(3\), and \(9\), all first occurring at level \(1\), and gave computational evidence for this conjecture. 

Calegari and Emerton subsequently introduced a cohomological approach which allowed them to reprove and extend several of Hatada's congruences \cite[Theorem~2.2 and Lemmas~3.1--3.3]{CalegariEmerton2004}. 

For \(p=7\), Hatada had recorded a conjecture of Serre predicting \begin{equation}\label{conj:serre-conjecture} a_\ell(f)\equiv1+\ell \pmod{7\overline{\mathbf Z}_7} \qquad \bigl(\ell\equiv\pm1\pmod7\bigr). \end{equation} Calegari and Emerton proved the weaker congruence \[ a_\ell(f)\equiv1+\ell \pmod{7^{1/2}\overline{\mathbf Z}_7} \qquad \bigl(\ell\equiv\pm1\pmod7\bigr) \] \cite[Lemma~3.2]{CalegariEmerton2004}. They also proved, for \(p=3\), \[ a_\ell(f)\equiv1+\ell \pmod{3^2\overline{\mathbf Z}_3} \qquad \bigl(\ell\equiv\pm1\pmod{3^2}\bigr), \] for every normalized level-one eigenform \(f\) \cite[Lemma~3.3]{CalegariEmerton2004}. 

In \cite{Rustom2019}, the author obtained a stronger result at level \(1\). For \(p=2\), all systems of Hecke eigenvalues modulo \(2^7\) were determined; in particular, for every normalized level-one eigenform \(f\) of weight \(k\) and every odd prime \(\ell\), \[ a_\ell(f)\equiv 1+\ell^{k-1} \pmod{2^7\overline{\mathbf Z}_2}. \] This gave further evidence for Conjecture~\ref{conj:krw-finiteness}. 

There is also a methodological continuity between these works. Calegari and Emerton point out that modular symbols are particularly well adapted to proving congruences that are independent of the weight, and interpret Hatada's calculations in this framework \cite{CalegariEmerton2004}. Their approach gives a cohomological formulation of the same general idea. The method of \cite{Rustom2019} was also in this spirit, using Merel's explicit theory of modular symbols to obtain finite computations from which congruences could be deduced. Related use of Dickson invariants in the study of torsion in the cohomology of \(\operatorname{SL}_2(\mathbf Z)\) appears in work of Deng \cite{Deng2019}, who also derives congruences between individual level-one cuspforms and Eisenstein series. His general congruence result is modulo a prime, with some higher prime-power congruences treated in specific examples.

The present work continues this approach and develops it further. In particular, we are able to prove the level \(1\) cases of Serre's conjecture (\ref{conj:serre-conjecture}) and Coleman-Stein's conjecture \cite[Conjecture~3.4]{ColemanStein2004}. The main new ingredient is a family of Hecke-equivariant transfer maps entirely within level \(1\), obtained by multiplication by suitable lifts of Dickson invariants. These maps propagate relations between level-one modular-symbol modules of different degrees and thereby reduce all-weight congruence questions to finite computation. In this way, we first obtain congruences for selected Hecke eigenvalues. Then we show that the Hecke operators to which these eigenvalues correspond generate the completed $p$-adic Hecke algebra, so that these congruences are enough to control every other Hecke eigenvalue. We do this by lifting Medvedovsky's \cite{Medvedovsky2015} explicitly determined generators of the corresponding mod $p$ Hecke algebras.

In this sense, the methods used here may be viewed as a more systematic development of the ideas introduced in \cite{Rustom2019}. We do not claim that the results obtained here are the limits of what one could prove using this method: the necessary computations were run on the author's modest laptop, and we believe that more can be achieved on a more powerful machine. 

In the GitHub repository \cite{RustomHeckeCongruences}, we provide Nim code interfacing with FLINT and PARI/GP to perform and verify the computations underlying the results of this paper. The code constructs the required modular-symbol modules, computes the Hecke actions, and verifies the prescribed Hecke-operator relations. We chose Nim for its speed and readability, particularly for readers familiar with Python. For each main theorem, we provide a Jupyter notebook presenting the verification, which can be run with a SageMath kernel.

In the next few introductory subsections, we fix notation and state the main results of the paper.

\paragraph{Acknowledgements and AI use disclosure.} The core of this work was developed by the author while employed as Assistant Professor at Ko\c{c} University, Istanbul, during 2018--2020. In particular, the level $1$ Coleman-Stein conjecture, and the level $1$ Serre conjecture modulo $7^{3/4}\bar \Z_7$, were obtained at that time.

AI tools, in particular Codex/GPT-5.6 Sol, were used to extend these initial results. AI contributions consisted mainly in the identification of the conjectural classifications and weight bounds in each case. AI also contributed significantly to the search for the modular-symbol Hecke operator identities used in Part~II to obtain congruences for selected Hecke eigenvalues. As the working modulus increases, these identities become increasingly complicated, and AI was useful for rapidly testing, rejecting, and identifying plausible identities. 

All AI-generated code was reviewed, and all computational results obtained were subsequently verified independently by the author. AI was also used for proofreading the manuscript, correcting typographical errors, and identifying notational inconsistencies. AI was not involved the development of the theoretical ideas in this work. The author takes full responsibility for all mathematical statements, proofs, computations, and references in the final manuscript.

\subsection{Congruence conventions}

Let \(p\) be a prime number, and fix once and for all algebraic closures \(\overline{\mathbf Q}\), with ring of integers \(\overline{\mathbf Z}\), and \(\overline{\mathbf Q}_p\), with ring of integers \(\overline{\mathbf Z}_p\), together with an embedding \[ \overline{\mathbf Q}\hookrightarrow\overline{\mathbf Q}_p. \] Let \(v_p\) denote the normalized valuation on \(\overline{\mathbf Q}_p\), so that \(v_p(p)=1\). 

We first clarify what we mean by congruence modulo \(p^m\), since two different notions arise when the coefficient field is ramified. 

\begin{definition} \label{def:conventions} Let \(\cO\) be the ring of integers of a finite extension of \(\Q_p\), let \(\pi\) be a uniformizer, and let \(e_{\cO}\) be the ramification index, so that \( p\cO=\pi^{e_{\cO}}\cO. \) For \(m\geq1\) and \(x,y\in\cO\), define \[ \begin{aligned} x\equiv_{\mathrm{val}} y\pmod{p^m} &\quad\Longleftrightarrow\quad v_p(x-y)>m-1,\\ x\equiv_{\mathrm{lit}} y\pmod{p^m} &\quad\Longleftrightarrow\quad x-y\in p^m\cO. \end{aligned} \] We call the first a \emph{valuative congruence modulo \(p^m\)} and the second a \emph{literal congruence modulo \(p^m\)}. The valuative convention is the one used in \cite[\S1.1, pp.~479--480]{KRW2016}. 
\end{definition} 

Equivalently, valuative congruence modulo \(p^m\) is equality in \(\cO/\pi^{e_{\cO}(m-1)+1}\cO, \) whereas literal congruence modulo \(p^m\) is equality in \( \cO/p^m\cO = \cO/\pi^{me_{\cO}}\cO. \) If the coefficient field is unramified, the two notions agree. If \(e_{\cO}>1\), literal congruence is stronger. More generally, literal congruence modulo \(p^m\) implies valuative congruence modulo \(p^m\), while valuative congruence modulo \(p^{m+1}\) implies literal congruence modulo \(p^m\). Indeed, \[ \pi^{e_{\cO}m+1}\cO \subseteq \pi^{me_{\cO}}\cO \subseteq \pi^{e_{\cO}(m-1)+1}\cO. \]

For modular forms, congruence is understood coefficientwise. Thus, if \[ f=\sum_{n\geq0}a_n(f)q^n, \qquad g=\sum_{n\geq0}a_n(g)q^n \] have coefficients in \(\cO\), we write \[ f\equiv g\pmod{p^m} \] valuatively, respectively literally, if \[ a_n(f)\equiv a_n(g)\pmod{p^m} \qquad(n\geq0) \] in the valuative, respectively literal, sense. The Fourier coefficients of a normalized Hecke eigenform \(f\) (normalized meaning that \(a_1(f)=1\)) are algebraic integers; hence, under our fixed embedding \[ \overline{\mathbf Q}\hookrightarrow\overline{\mathbf Q}_p, \] they belong to \(\overline{\mathbf Z}_p\). If \(f\) and \(g\) are normalized Hecke eigenforms, we say that their \emph{prime-to-\(p\) Hecke eigensystems} are congruent modulo \(p^m\), valuatively or literally, if \[ a_n(f)\equiv a_n(g)\pmod{p^m} \qquad((n,p)=1) \] in the corresponding sense. Such congruences may be checked in the ring of integers of any finite extension of \(\mathbf Q_p\) containing the coefficients of both forms.

\subsection{Cyclotomic Hecke eigensystems}

Let \(\varphi\) denote the Euler totient function, and put \(\theta:=q\frac{d}{dq}\). For even \(\kappa\geq4\), let \[ G_\kappa:=-\frac{B_\kappa}{2\kappa} +\sum_{n\geq1}\sigma_{\kappa-1}(n)q^n, \qquad \sigma_{\kappa-1}(n):=\sum_{d\mid n}d^{\kappa-1}, \] be the Eisenstein series normalized at the first degree term. For \(i\geq1\), the constant term is killed by \(\theta^i\), and \[ \theta^iG_\kappa =\sum_{n\geq1}n^i\sigma_{\kappa-1}(n)q^n \] has integral coefficients.

\begin{definition}[Cyclotomic Hecke eigensystem] \label{def:theta-eis} Let \(f\) be a normalized Hecke eigenform. Its prime-to-\(p\) Hecke eigensystem is called \emph{cyclotomic modulo \(p^m\)} if, for some integer \(i\geq m\) and some even integer \(\kappa\geq4\), \[ a_n(f)\equiv n^i\sigma_{\kappa-1}(n)\pmod{p^m} \qquad ((n,p)=1). \] The congruence may be understood literally or valuatively, as in Definition~\ref{def:conventions}.
\end{definition}

If \(e\geq m\) is divisible by the exponent of \((\mathbf Z/p^m\mathbf Z)^\times\), this condition is equivalent to \[ \theta^e f\equiv\theta^iG_\kappa\pmod{p^m} \] in the same sense. At primes \(\ell\ne p\), it gives \[ a_\ell(f)\equiv\ell^i+\ell^{i+\kappa-1}\pmod{p^m}. \] By symmetry of the divisor sum, the prime-to-\(p\) coefficient system depends only on the unordered pair \(\{i,i+\kappa-1\}\), with exponents taken modulo the exponent of \((\mathbf Z/p^m\mathbf Z)^\times\). We also refer to this system as a cyclotomic packet. 

\subsection{Main results}

We now state the main results of the paper. Throughout this subsection, let \(f\) be a normalized cuspidal level-one eigenform of even weight \(k\), and let \(\mathcal O_f\) be the ring of integers of a finite extension of \(\mathbf Q_p\) containing the coefficients of \(f\). 

\begin{theorem}[Prime-to-\(2\) classification modulo \(256\)] \label{thm:intro-p2} If \[ k-2\equiv2,4,10,12\pmod{16}, \] then \[ \theta^{64}f \equiv_{\mathrm{val}} \theta^{64}G_k\pmod{256}. \] If \[ k-2\equiv0,6,8,14\pmod{16}, \] then exactly one of \[ \theta^{64}f \equiv_{\mathrm{val}} \theta^{64}G_k\pmod{256} \] and \[ \theta^{64}f \equiv_{\mathrm{val}} \theta^{16}G_{k+32}\pmod{256} \] holds.

Writing the selected Eisenstein twist as \(\theta^iG_\kappa\), the unordered pair \[ \{a,b\}:=\{i,i+\kappa-1\}\subset\mathbf Z/64\mathbf Z \] is uniquely determined by \(f\). Its possibilities are:
\[
\begin{array}{c|l}
k\bmod16 & \{a,b\}\\ \hline
0,2,8,10 & \{0,k-1\},\ \{16,k-17\}\\
4,6,12,14 & \{0,k-1\}.
\end{array}
\]

Conversely, for every even residue class \(k_0\bmod64\), each permitted prime-to-\(2\) Hecke eigensystem is realized by a normalized cuspidal level-one eigenform of some weight \(K\equiv k_0\pmod{64}\). Thus there are exactly \(48\) strong prime-to-\(2\) Hecke eigensystems modulo \(256\), in the valuative sense.
\end{theorem}

\begin{theorem}[Prime-to-\(3\) classification modulo \(81\)] \label{thm:intro-p3} There exist integers \(i\geq4\) and even \(\kappa\geq4\), with \(\kappa+2i\equiv k\pmod{54}\), and a modular form \(\delta_f\) over \(\mathbf F_9\), such that \[ \theta^{54}f \equiv_{\mathrm{val}} \theta^iG_\kappa+27\delta_f\pmod{81}. \] In particular, \[ \theta^{54}f\equiv_{\mathrm{lit}}\theta^iG_\kappa\pmod{27}. \] The unordered pair \[ \{a,b\}:=\{i,i+\kappa-1\}\subset\mathbf Z/18\mathbf Z \] is uniquely determined by \(f\) and belongs to the following list:
\[
\begin{array}{c|l}
k\bmod18 & \{a,b\}\\ \hline
0  & \{0,17\},\ \{8,9\}\\
2  & \{0,1\},\ \{9,10\}\\
4  & \{0,3\},\ \{6,15\},\ \{9,12\}\\
6  & \{0,5\},\ \{9,14\}\\
8  & \{0,7\},\ \{9,16\}\\
10 & \{0,9\},\ \{3,6\},\ \{12,15\}\\
12 & \{0,11\},\ \{2,9\}\\
14 & \{0,13\},\ \{4,9\}\\
16 & \{0,15\},\ \{3,12\},\ \{6,9\}.
\end{array}
\]

There exists a unique value \(\alpha \in \Z_3[\sqrt 2]\) listed for \(k\bmod54\) in Table~\ref{tab:p3-signatures} of Appendix~\ref{app:p3-signatures}, such that \[ a_2(f)\equiv_{\mathrm{val}}\alpha\pmod{81}, \qquad a_7(f)\equiv_{\mathrm{val}}1+7^{k-1}\pmod{81}. \] Once \(i,\kappa\) are fixed, the signature \((k\bmod54,\alpha)\) determines the \(q\)-expansion of \(\delta_f\). The correction \(\delta_f\) is defined over \(\mathbf F_3\) when \(\alpha\) is rational, and over \(\mathbf F_9\), but not over \(\mathbf F_3\), otherwise.

Two normalized cuspidal level-one eigenforms have congruent prime-to-\(3\) Hecke eigensystems modulo \(81\), in the valuative sense, if and only if their weights are congruent modulo \(54\) and their corresponding values \(\alpha\) agree. 

Every signature in Table~\ref{tab:p3-signatures} is realized by a normalized cuspidal level-one eigenform of weight at most \(214\) in the corresponding residue class modulo \(54\). Thus there are exactly \(159\) such eigensystems modulo \(81\). Their reductions realize all \(21\) pairs listed above, giving exactly \(21\) strong prime-to-\(3\) Hecke eigensystems modulo \(27\), in the literal sense.
\end{theorem}

Theorem~\ref{thm:intro-p3} implies the Coleman-Stein conjecture at level $1$. 

\begin{theorem}[Prime-to-\(5\) classification modulo \(125\)] \label{thm:intro-p5} There exist integers \(i\geq3\) and even \(\kappa\geq4\), with \(\kappa+2i\equiv k\pmod{100}\), and a modular form \(\delta_f\) over \(\mathbf F_5\), such that \[ \theta^{100}f \equiv_{\mathrm{val}} \theta^iG_\kappa+25\delta_f\pmod{125}. \] In particular, \[ \theta^{100}f\equiv_{\mathrm{lit}}\theta^iG_\kappa\pmod{25}. \] The unordered pair \[ \{a,b\}:=\{i,i+\kappa-1\}\subset\mathbf Z/20\mathbf Z \] is uniquely determined by \(f\) and belongs to the following list:
\[
\begin{array}{c|l}
k\bmod20 & \{a,b\}\\ \hline
0  & \{0,19\},\ \{3,16\},\ \{4,15\},\
     \{5,14\},\ \{6,13\},\ \{9,10\}\\
2  & \{0,1\},\ \{4,17\},\ \{5,16\},\
     \{6,15\},\ \{7,14\},\ \{10,11\}\\
4  & \{0,3\},\ \{4,19\},\ \{5,18\},\
     \{8,15\},\ \{9,14\},\ \{10,13\}\\
6  & \{0,5\},\ \{2,3\},\ \{6,19\},\
     \{9,16\},\ \{10,15\},\ \{12,13\}\\
8  & \{0,7\},\ \{1,6\},\ \{2,5\},\
     \{10,17\},\ \{11,16\},\ \{12,15\}\\
10 & \{0,9\},\ \{1,8\},\ \{4,5\},\
     \{10,19\},\ \{11,18\},\ \{14,15\}\\
12 & \{0,11\},\ \{1,10\},\ \{2,9\},\
     \{5,6\},\ \{12,19\},\ \{15,16\}\\
14 & \{0,13\},\ \{3,10\},\ \{4,9\},\
     \{5,8\},\ \{14,19\},\ \{15,18\}\\
16 & \{0,15\},\ \{1,14\},\ \{4,11\},\
     \{5,10\},\ \{7,8\},\ \{17,18\}\\
18 & \{0,17\},\ \{1,16\},\ \{2,15\},\
     \{5,12\},\ \{6,11\},\ \{7,10\}.
\end{array}
\]

Put \(r\equiv k-2\pmod{100}\), with \(0\leq r<100\), and use the constants in Table~\ref{tab:p5-mod125-constants}. There exists a unique pair \[ (a,t)\in\mathbf Z/125\mathbf Z\times\mathbf Z/25\mathbf Z \] satisfying \[ t=v_r\qquad\text{or}\qquad t\equiv v_r\pm2\pmod5, \] and \[ a^2\equiv h_{r,0}+5h_{r,1}t+25h_{r,2}t^2\pmod{125}, \] such that \[ a_2(f)\equiv_{\mathrm{val}}a\pmod{125}, \qquad a_{19}(f)\equiv_{\mathrm{val}}5t\pmod{125}. \] Once \(i,\kappa\) are fixed, the signature \((r,a,5t)\) determines the \(q\)-expansion of \(\delta_f\).

Two normalized cuspidal level-one eigenforms have congruent prime-to-\(5\) Hecke eigensystems modulo \(125\), in the valuative sense, if and only if their weights are congruent modulo \(100\) and their corresponding pairs \((a,t)\) agree. 

Every permitted signature \((r,a,5t)\) is realized by a normalized cuspidal level-one eigenform of weight at most \(598\) in the corresponding residue class modulo \(100\). Thus there are exactly \(1100\) such eigensystems modulo \(125\). Their reductions realize all \(60\) pairs listed above, giving exactly \(60\) strong prime-to-\(5\) Hecke eigensystems modulo \(25\), in the literal sense.
\end{theorem}

\begin{theorem}[Prime-to-\(7\) classification modulo \(49\)] \label{thm:intro-p7} There exist integers \(i\geq2\) and even \(\kappa\geq4\), with \(\kappa+2i\equiv k\pmod{42}\), and a modular form \(\delta_f\) over \(\mathbf F_7\), such that \[ \theta^{42}f \equiv_{\mathrm{val}} \theta^iG_\kappa+7\delta_f\pmod{49}. \] In particular, \[ \theta^{42}f\equiv_{\mathrm{lit}}\theta^iG_\kappa\pmod7. \] The unordered pair \[ \{a,b\}:=\{i,i+\kappa-1\}\subset\mathbf Z/6\mathbf Z \] is uniquely determined by \(f\) and belongs to the following list: 
\[
\begin{array}{c|l}
k\bmod6 & \{a,b\}\\ \hline
0 & \{0,5\},\ \{1,4\},\ \{2,3\}\\
2 & \{0,1\},\ \{2,5\},\ \{3,4\}\\
4 & \{0,3\},\ \{1,2\},\ \{4,5\}.
\end{array}
\]

There exists a unique triple \[ (c,x,y)\in\{0,\ldots,6\}^3 \] listed for \(k\bmod42\) in Table~\ref{tab:p7-signatures} of Appendix~\ref{app:p7-signatures}, such that \[ a_3(f)\equiv_{\mathrm{val}}c+7x\pmod{49}, \qquad a_{29}(f)\equiv_{\mathrm{val}}2+7y\pmod{49}. \] Once \(i,\kappa\) are fixed, the signature \((k\bmod42,c,x,y)\) determines the \(q\)-expansion of \(\delta_f\).

Two normalized cuspidal level-one eigenforms have congruent prime-to-\(7\) Hecke eigensystems modulo \(49\), in the valuative sense, if and only if their weights are congruent modulo \(42\) and their corresponding triples \((c,x,y)\) agree. 

Every signature in Table~\ref{tab:p7-signatures} is realized by a normalized cuspidal level-one eigenform of weight at most \(380\) in the corresponding residue class modulo \(42\). Thus there are exactly \(315\) such eigensystems modulo \(49\). Their reductions realize all nine pairs listed above, giving exactly \(9\) strong prime-to-\(7\) Hecke eigensystems modulo \(7\), in the literal sense.
\end{theorem}

Theorem~\ref{thm:intro-p7} implies Serre's conjecture at level $1$. 

\section{Modular forms and Hecke algebras}

\subsection{Hecke eigensystems, coefficient quotients, and completed Hecke algebras}
\label{sec:basic}

Denote by \(S_k(\Z) \subset \Z[[q]]\) the lattice generated by the $q$-expansions of level $1$ cuspforms of weight $k$ with coefficients in $\Z$.  Let $K$ be a finite extension of $\Q_p$ with ring of integers $\cO$. We will denote by $S_k(\cO)$ the $\cO$-submodule of $\cO[[q]]$ spanned by the image of $S_k(\Z)$. 

For \(h\geq0\), put \[ S_{\leq h}(\cO):=\sum_{0\leq k\leq h}S_k(\cO),  \qquad  D_{\leq h}(\cO):= \bigl( S_{\leq h}(\cO)\otimes_{\cO} K\bigr)    \cap\cO[[q]]. \] Equivalently, \(D_{\leq h}(\cO)\) is the \(p\)-power saturation of \(S_{\leq h}(\cO)\) inside \(\mathbf \cO[[q]]\). For $B$ a $p$-adically complete and separated local $\cO$-algebra \(B\), put \[ D_{\leq h}(B):=D_{\leq h}(\cO)\otimes_{\mathbf \cO}B \] and denote by $S_k(B)$ the image of $S_k(\cO)$ in $D_{\leq h}(B)$. These spaces inherit the action of the Hecke operators from the spaces of modular forms. 

We recall the definition of strong, weak, and dc-weak eigenform, first introduced in \cite{CKW2013}.

\begin{definition}[Strong, weak, and dc-weak eigenforms]\label{def:strong-weak} Let \(B\) be a finite local \(\mathbf Z_p\)-algebra and let \(f = q + \sum_{n\geq 2}a_n(f)q^n\in B[[q]]\) be a normalized \(q\)-expansion with coefficients in \(B\).
\begin{enumerate}
\item The form \(f\) is \emph{strong} if it is obtained by reducing a
normalized characteristic-zero eigenform through a coefficient-ring map
to \(B\).
\item The form \(f\) is \emph{weak} if \(f\in S_k(B)\) for some \(k\)
and is a simultaneous Hecke eigenform there.
\item The form \(f\) is \emph{dc-weak} if \(f\in D_{\leq h}(B)\) for
some \(h\) and is a simultaneous Hecke eigenform there.
\end{enumerate}
\end{definition}

\begin{definition} \label{def:anemic} 
For each \(h\), define
\begin{equation}
\begin{aligned}
 \mathbb T_{\leq h}^{(p)}(R)
  &:=\operatorname{im}\!\left(
       R[\mathsf T_n:(n,p)=1]\to
      \operatorname{End}_{R}(D_{\leq h}(R))\right).
\end{aligned}
 \label{eq:bounded-anemic}
\end{equation}
If \(h\leq h'\), restriction from \(D_{\leq h'}(R)\) to \(D_{\leq h}(R)\) allows us to define the completed algebra \begin{equation}  \mathbb T_{\mathrm{dc}}^{(p)}(R)    :=\varprojlim_h\mathbb T_{\leq h}^{(p)}(R).  \label{eq:complete-anemic} \end{equation}
\end{definition}

We will usually write $\mathbb T_{\mathrm{dc}}^{(p)}:=\mathbb T_{\mathrm{dc}}^{(p)}(\Z_p)$.

For \(u\in\mathbf Z_p^\times\), Katz's weight action \([u]\) on divided congruences acts on weight \(k\) by \(u^k\); see \cite[(2.4)]{Katz1975a} and \cite[\S X]{Katz1975b}.  The extension to the low-level cases \(N<3\) used here is explained in \cite[Theorem~2.2 and the following paragraph]{Rustom2019}. For odd \(p\), put \(t_\Lambda:=[1+p]\), and for \(p=2\), put \(t_\Lambda:=[5]\).  It can be shown that $t_\Lambda \in \mathbb T_{\mathrm{dc}}^{(p)}(R)$ \cite[Lemma 9]{BellaicheKhare2015}.

Hecke eigensystems away-from-$p$ arising from strong, weak, or dc-weak eigenforms with coefficients in $R$ correspond to continuous $\Z_p$-algebra homomorphisms $\mathbb T_{\mathrm{dc}}^{(p)} \rightarrow R$.

\begin{definition}[Shallow Hecke algebra]
\label{def:shallow}
Let \(M_{\leq h}(\F_p)\) be the space of classical level-one modular
forms modulo \(p\) of weights at most \(h\).  Following the terminology
of \cite[\S1.2]{BellaicheKhare2015}, define the \emph{shallow Hecke
algebra} by
\[
\mathbb T^{sh,(p)}
:=\varprojlim_h\operatorname{im}\!\left(
\F_p[\mathsf T_n:(n,p)=1]
\longrightarrow\End_{\F_p}(M_{\leq h}(\F_p))\right).
\]
For a
residual prime-to-\(p\) eigensystem \(\tau\), its completed local factor is denoted \(\mathbb T^{sh,(p)}_\tau\).
\end{definition}

\subsection{Lifting shallow Hecke generators}

Let \(\mathbf F/\mathbf F_p\) be a finite extension, and put \[\mathbb T_{\mathbf F}^{\mathrm{sh},(p)}:=\mathbb T^{\mathrm{sh},(p)}\otimes_{\mathbf F_p}\mathbf F.\] For an \(\mathbf F\)-valued residual eigensystem
\[\tau:\mathbb T_{\mathbf F}^{\mathrm{sh},(p)}\longrightarrow\mathbf F,\] let \[\mathfrak m_\tau^{\mathrm{sh}}:=\ker(\tau),\] and denote by \(\mathbb T_\tau^{\mathrm{sh},(p)}\) the corresponding completed local factor.

Let \(\mathcal O\) be the ring of integers of a finite extension of \(\mathbf Q_p\) with residue field \(\mathbf F\), and assume that all residual prime-to-\(p\) eigensystems are \(\mathbf F\)-valued. The Hecke algebra \(\mathbb T_{\mathrm{dc}}^{(p)}(\mathcal O)\) is a complete Noetherian semilocal ring, with decomposition \[\mathbb T_{\mathrm{dc}}^{(p)}(\mathcal O)=\prod_\tau\mathbb T_{\mathrm{dc},\tau}^{(p)}(\mathcal O),\] where \(\tau\) runs over the residual prime-to-\(p\) eigensystems. 

\begin{lemma}[Lifting shallow generators]
\label{lem:shallow-lifting}
Assume that \(p\geq5\). Let \[\mathcal H=\{h_1,\ldots,h_s\}\subseteq\mathbb T_{\mathrm{dc}}^{(p)}(\mathcal O),\] and write \(\bar h_i\) for the image of \(h_i\) in \(\mathbb T_{\mathbf F}^{\mathrm{sh},(p)}\). Suppose that:
\begin{enumerate}
\item for every residual eigensystem \(\tau\), the classes of \[
\bar h_i-\tau(\bar h_i),\qquad 1\leq i\leq s,\] span \[\mathfrak m_\tau^{\mathrm{sh}}/(\mathfrak m_\tau^{\mathrm{sh}})^2;\]
\item the map \[\tau\longmapsto\bigl(\tau(\bar h_1),\ldots,\tau(\bar h_s)\bigr)\in\mathbf F^s\] is injective.
\end{enumerate}

Then \[\mathbb T_{\mathrm{dc}}^{(p)}(\mathcal O)=\overline{\mathcal O[t_\Lambda,h_1,\ldots,h_s]}.\]
\end{lemma}

\begin{proof}
The passage from generators of \[\mathfrak m_\tau^{\mathrm{sh}}/(\mathfrak m_\tau^{\mathrm{sh}})^2\] to topological generators of the completed local factor \(\mathbb T_{\mathrm{dc},\tau}^{(p)}(\mathcal O)\), using Nakayama's lemma, is the argument used in \cite[\S7]{Rustom2019}. We recall it briefly. By \cite[Lemma~9]{BellaicheKhare2015}, \(t_\Lambda=[1+p]\) belongs to \(\mathbb T_{\mathrm{dc}}^{(p)}(\mathcal O)\), and the argument in the proof of \cite[Proposition~16]{BellaicheKhare2015} gives, after extension of scalars and passage to the \(\tau\)-local factor, \[\mathbb T_{\mathrm{dc},\tau}^{(p)}(\mathcal O)/(\pi,t_\Lambda-1)\mathbb T_{\mathrm{dc},\tau}^{(p)}(\mathcal O)\simeq\mathbb T_\tau^{\mathrm{sh},(p)}.\] Hypothesis~(1), together with Nakayama's lemma, therefore shows that \(t_\Lambda,h_1,\ldots,h_s\) topologically generate every completed local factor as an \(\mathcal O\)-algebra.

Put \(A=\overline{\mathcal O[t_\Lambda,h_1,\ldots,h_s]}\subset \mathbb T_{\mathrm{dc}}^{(p)}(\mathcal O)\), and let \(\mathfrak m_\tau\) be the maximal ideal of \(\mathbb T_{\mathrm{dc},\tau}^{(p)}(\mathcal O)\). The preceding paragraph shows that \(A\) surjects onto each \[ \mathbb T_{\mathrm{dc},\tau}^{(p)}(\mathcal O)/ \mathfrak m_\tau^n. \] The maximal ideals \(\mathfrak n_\tau:=\ker(\tau|_A)\) are distinct by hypothesis~(2). The kernels of these maps are therefore pairwise comaximal, so the Chinese remainder theorem shows that \(A\) surjects onto their product for every \(n\). Thus \(A\) is dense in \(\mathbb T_{\mathrm{dc}}^{(p)}(\mathcal O)\). Since \(A\) is closed by definition, the two algebras are equal. 
\end{proof}

\subsubsection{\texorpdfstring{$p = 2,3$}{p=2,3}}\label{sec:big-hecke-p2p3}

By \cite[Theorem~7.2(i)--(ii) and Proposition~3.8(i)]{Rustom2019}, the completed prime-to-\(2\) Hecke algebra is \begin{equation}\mathbb T_{\mathrm{dc}}^{(2)}=\overline{\mathbf Z_2[T_3,T_5,t_\Lambda]},\qquad t_\Lambda=[5], \label{eq:p2-big-generation} \end{equation} and the completed prime-to-\(3\) Hecke algebra is \begin{equation} \mathbb T_{\mathrm{dc}}^{(3)} = \overline{\mathbf Z_3[T_2,T_7,t_\Lambda]}, \qquad t_\Lambda=[4]. \label{eq:p3-big-generation} \end{equation} Consequently, two continuous characters of \(\mathbb T_{\mathrm{dc}}^{(2)}\) with values in a finite \(\mathbf Z_2\)-algebra are equal if they agree on \(T_3\), \(T_5\), and \(t_\Lambda\), while two continuous characters of \(\mathbb T_{\mathrm{dc}}^{(3)}\) with values in a finite \(\mathbf Z_3\)-algebra are equal if they agree on \(T_2\), \(T_7\), and \(t_\Lambda\). \subsubsection{\texorpdfstring{$p = 5$}{p=5}}
\label{sec:p5-big-hecke}

Medvedovsky shows that the shallow Hecke algebra at \(p=5\) has four
twist-equivalent residual local factors. One of them is \[\tau_0:=\omega^{-1}+1\] where $\omega$ is the mod-$5$ cyclotomic character. Then \[\mathbb T_{\tau_0}^{\mathrm{sh},(5)}\simeq\mathbf F_5[[x,y]],\qquad x:=T_{11}-2,\qquad y:=T_{19};\] see \cite[Theorem~8.1 and \S8.5]{Medvedovsky2015}. We would like to lift this result to obtain generators of \(\mathbb T_{\mathrm{dc}}^{(5)}\). The difficulty is that the residual
values of \(T_{11}\) do not separate the four local factors. We will show that \(T_{11}\) may be replaced by \(T_2\).

Let \(\Delta\) denote the modular discriminant of weight $12$, and let \(E_4\) denote the normalized Eisenstein series of weight \(4\). Consider the following forms modulo \(5\):
\[
\begin{aligned}
f_0&=2\Delta+\Delta^2,\\
f_1&=3\Delta+4\Delta^3+\Delta^4,\\
f_2&=3\Delta+\Delta^3+\Delta^5+3\Delta^6+2\Delta^7
      +4\Delta^8+\Delta^9\\
   &\qquad
      +2\Delta^{11}+2\Delta^{12}
      +4\Delta^{13}+\Delta^{14}.
\end{aligned}
\]
These are in fact mod $5$ reductions of classical modular forms, as can be seen using the identity \(E_4 \equiv 1 \pmod{5}\). 

The residual \(T_2\)-values on the four local factors are
\[
4,\qquad3,\qquad1,\qquad2,
\]
with value \(4\) on \(\tau_0\). Put \(s:=T_2-4.\) Then \(s\) belongs to \(\mathfrak m_{\tau_0}^{\mathrm{sh}}\) and is a unit on each of the other three local factors. A direct \(q\)-expansion calculation gives
\begin{equation}
\begin{array}{c|ccc}
 &s&x&y\\ \hline
f_0&0&0&0\\
f_1&3f_0&0&f_0\\
f_2&3f_0&4f_0&4f_0
\end{array}
\label{eq:p5-cotangent-table}
\end{equation}
In particular, \(s^2f_i=0\). Since \(s\) is a unit on every residual local factor other than \(\tau_0\), it follows that \(f_0,f_1,f_2\) belong to the \(\tau_0\)-primary summand.

We also get from \eqref{eq:p5-cotangent-table} that \[(\mathfrak m_{\tau_0}^{\mathrm{sh}})^2f_i=0 \qquad (i=0,1,2).\]

Since \(\mathbb T_{\tau_0}^{\mathrm{sh},(5)}\simeq\mathbf F_5[[x,y]],\) we have \[\dim_{\F_5}\mathfrak m_{\tau_0}^{\mathrm{sh}}/(\mathfrak m_{\tau_0}^{\mathrm{sh}})^2 = 2.\]

Suppose \[\beta s+\gamma y=0,\qquad\beta,\gamma\in\mathbf F_5,\] is a relation in \(\mathfrak m_{\tau_0}^{\mathrm{sh}}/(\mathfrak m_{\tau_0}^{\mathrm{sh}})^2.\) Then \(\beta s+\gamma y\) annihilates \(f_1\) and \(f_2\). By \eqref{eq:p5-cotangent-table}, this forces \(\beta=\gamma=0.\) Thus the classes of \(s\) and \(y\) are linearly independent, and therefore form a basis of \(\mathfrak m_{\tau_0}^{\mathrm{sh}}/(\mathfrak m_{\tau_0}^{\mathrm{sh}})^2.\) Equivalently, the classes of \(T_2-4\) and \(T_{19}\) form a basis of this cotangent space.

By \cite[Proposition~2.41]{Medvedovsky2015}, twisting by the \(\theta\) operator permutes the four residual local factors and their cotangent spaces. Under these permutations, the two classes above are carried, up to nonzero scalars, to the classes of
\[
T_2-c_\tau,\qquad T_{19},
\]
where \(c_\tau=\tau(T_2)\). Hence these classes span \(\mathfrak m_\tau^{\mathrm{sh}}/(\mathfrak m_\tau^{\mathrm{sh}})^2\) for every residual eigensystem \(\tau\). Moreover, the residual values of \(T_2\) distinguish the four residual eigensystems. Lemma \ref{lem:shallow-lifting} therefore gives \begin{equation}\mathbb T_{\mathrm{dc}}^{(5)}=\overline{\mathbf Z_5[T_2,T_{19},t_\Lambda]},\qquad t_\Lambda=[6].\label{eq:p5-big-generation}\end{equation}

\subsubsection{\texorpdfstring{$p=7$}{p=7}}
\label{sec:big-hecke-p7}
Let \(\omega\) be the mod-\(7\) cyclotomic character. Medvedovsky shows that the shallow Hecke algebra has nine residual eigensystems \[\tau_{a,b}=\omega^a+\omega^b,\qquad a\in\{0,2,4\},\quad b\in\{1,3,5\};\] see \cite[Theorem~8.1 and \S8.7]{Medvedovsky2015}.

\begin{lemma}
\label{lem:p7-shallow-parameters}
For \(\tau=\tau_{a,b}\), put \[c_\tau:=\tau(T_3)=3^a+3^b\in\mathbf F_7.\] Then the classes of \[T_3-c_\tau,\qquad T_{29}-2\] form a basis of \(\mathfrak m_\tau^{\mathrm{sh}}/(\mathfrak m_\tau^{\mathrm{sh}})^2.\)
\end{lemma}

\begin{proof} Put \(d:=b-a\in\{1,3,5\}\pmod6.\) By \cite[Corollary~7.2]{Medvedovsky2015}, we need a prime $\ell$ such that \[\ell^d\not\equiv1\pmod7,\qquad\ell^6\not\equiv1\pmod{49}.\] These conditions are satisfied for \(\ell=3\). Thus \(T_3-c_\tau\) will be one of our generators.

For the other generator, we use Medvedovsky's criterion \cite[Corollary~7.8]{Medvedovsky2015}. Put \[ E:=\mathcal O_{\mathbf Q(\zeta_7)}[1/7]^\times, \] the group of \(7\)-units of \(\mathbf Q(\zeta_7)\). We need classes \(\alpha,\beta\in E/E^7\) lying in the \(\omega^{1-d}\)- and \(\omega^{1+d}\)-eigenspaces, respectively. Furthermore, we need a prime $\ell \equiv 1 \pmod{7}$ such that neither $\alpha$ nor $\beta$ are $7$th powers in $\F_\ell$. 

For \(j\in\{0,2,4\}\), put \[u_j:=\prod_{r=1}^{6}(1-\zeta_7^r)^{r^{-j}}\in E/E^7.\] If \(\sigma_c(\zeta_7)=\zeta_7^c\), then \[\sigma_c(u_j)=u_j^{c^j},\] so \(u_j\) lies in the \(\omega^j\)-eigenspace of \(E/E^7\). For \(d=1,3,5\), the required pairs of eigenspace exponents are \[(1-d,1+d)=(0,2),\ (4,4),\ (2,0)\pmod6.\]

For any prime \(\mathfrak l\) of \(\mathbf Q(\zeta_7)\) above \(29\),
the reductions of \(u_0,u_2,u_4\) in \[\mathcal O_{\mathbf Q(\zeta_7)}/\mathfrak l\simeq \mathbf F_{29}\] are not seventh powers. Thus we can take $T_{29}$ as our second generator.
\end{proof}

The level-one Hecke recursion gives \begin{equation}[3]=3(T_3^2-T_9),\label{eq:p7-weight-in-away}\end{equation} so \[[3]\in\mathbb T_{\mathrm{dc}}^{(7)}.\] We apply Lemma~\ref{lem:shallow-lifting} with \[\mathcal H=\{T_3,T_{29},[3]\}.\] Lemma~\ref{lem:p7-shallow-parameters} verifies hypothesis~(1) of Lemma~\ref{lem:shallow-lifting}. 

It remains to verify that the residual values of the elements of \(\mathcal H\) distinguish the nine residual eigensystems. For \(\tau=\tau_{a,b}\), we have \[\tau(T_3)=3^a+3^b,\qquad\tau([3])=3^{a+b+1}.\] The corresponding pairs are
\[
\begin{array}{c|ccc}
 & b=1 & b=3 & b=5\\ \hline
a=0 & (4,2) & (0,4) & (6,1)\\
a=2 & (5,4) & (1,1) & (0,2)\\
a=4 & (0,1) & (3,2) & (2,4).
\end{array}
\]
These nine pairs are distinct. Hence \(T_3\) and \([3]\) jointly distinguish the nine residual eigensystems, so hypothesis~(2) of Lemma~\ref{lem:shallow-lifting} is satisfied.

Lemma~\ref{lem:shallow-lifting} therefore gives
\[
\mathbb T_{\mathrm{dc}}^{(7)}
=
\overline{\mathbf Z_7[t_\Lambda,T_3,T_{29},[3]]},
\qquad
t_\Lambda=[8].
\]

Finally, \(t_\Lambda\) is redundant. The element \(3\) is a topological generator of \(\mathbf Z_7^\times\). By continuity of the weight action, \([8]\) belongs to the closed
\(\mathbf Z_7\)-algebra generated by \([3]\). Thus \[\mathbb T_{\mathrm{dc}}^{(7)}=\overline{\mathbf Z_7[T_3,T_{29},[3]]}.\]

\subsection{Weak realization of cyclotomic packets}

To compare a strong eigensystem with a cyclotomic packet through the completed Hecke algebra, we need to know that the packet itself is realized by a genuine weak Hecke eigensystem modulo \(p^m\).

\begin{lemma}[Weak realization of Eisenstein twists] \label{lem:weak-cyclotomic} Let \(p\) be a prime, let \(m\geq1\), let \(i\geq m\), and let \(\kappa\geq4\) be even. The reduction of \(\theta^iG_\kappa\) modulo \(p^m\) is a normalized weak cuspidal eigenform of some weight \(K\) satisfying \[ K\equiv\kappa+2i\pmod{\varphi(p^m)}.
\]
\end{lemma}

\begin{proof} Put \[ G_\kappa^*:=G_\kappa-p^{\kappa-1}G_\kappa\mid V_p. \] By \cite[\S1.6 and Théorème~5]{Serre1973}, \(\theta^iG_\kappa^*\) is an integral level-one \(p\)-adic modular form of weight \(\kappa+2i\). Since \(i\geq m\), \[ \theta^iG_\kappa^* \equiv\theta^iG_\kappa\pmod{p^m}. \] Thus there exist a sufficiently large even weight \(K\)
and a classical modular form \(F\in M_K(\mathbf Z_p)\) such that \[ F\equiv\theta^iG_\kappa\pmod{p^m}, \qquad K\equiv\kappa+2i\pmod{\varphi(p^m)}. \] The reduction of \(F\) is cuspidal and normalized, since its constant term is zero and its coefficient at \(q\) is \(1\). For every prime \(\ell\ne p\), the Hecke formula and the weight congruence give \[ T_\ell F\equiv\ell^i(1+\ell^{\kappa-1})F\pmod{p^m}. \] Since \(i\geq m\), the coefficients of \(F\) at indices divisible by \(p\) vanish modulo \(p^m\). Taking \(K-1\geq m\), the Hecke formula also gives \(T_pF\equiv0\pmod{p^m}\). Thus the reduction of \(F\), which equals that of \(\theta^iG_\kappa\), is a normalized weak cuspidal eigenform.
\end{proof}

\begin{lemma}[Modularity of the correction term] \label{lem:modular-correction} Let \(p\in\{2,3,5,7\}\), let \(m\geq2\), assuming \(m\geq3\) if \(p=2\), and put \(e=\varphi(p^m)\). Let \(f\in M_k(\mathcal O_f)\), and suppose that \(i\geq m\) and \(\kappa\geq4\) is even, with \[ \kappa+2i\equiv k\pmod e. \] If \[ \theta^e f\equiv\theta^iG_\kappa\pmod{p^{m-1}} \] literally, then the reduction of \[ \frac{\theta^e f-\theta^iG_\kappa}{p^{m-1}} \] modulo the maximal ideal of \(\mathcal O_f\) is the \(q\)-expansion of a modular form over its residue field. 
\end{lemma}

\begin{proof} By the classical approximation used in the preceding proof, there are classical modular forms \(F_1,F_2\), with coefficients in \(\mathcal O_f\), satisfying \[ F_1\equiv\theta^e f\pmod{p^m}, \qquad F_2\equiv\theta^iG_\kappa\pmod{p^m}, \]whose weights are both congruent to \(k\) modulo \(e\). 

Let \(E_j\) denote the normalized Eisenstein series of weight \(j\). The modular form \[ H= \begin{cases} E_4^{\,2^{m-3}},&p=2,\\ E_6^{\,3^{m-2}},&p=3,\\ E_{p-1}^{\,p^{m-1}},&p=5,7 \end{cases} \] has weight \(e\) and satisfies \(H\equiv1\pmod{p^m}\). Multiplying the form of smaller weight by a suitable power of \(H\), we may therefore assume that \(F_1\) and \(F_2\) have the same weight \(K\), without changing these congruences.

The assumed literal congruence now gives \[ D:=\frac{F_1-F_2}{p^{m-1}}\in M_K(\mathcal O_f). \] Moreover, \[ D\equiv \frac{\theta^e f-\theta^iG_\kappa}{p^{m-1}} \pmod p. \] Reducing \(D\) modulo the maximal ideal proves the claim. 
\end{proof}

\section{Modular symbols of level \texorpdfstring{\(1\)}{1}}

\subsection{Modular symbols with coefficients in \texorpdfstring{\(R\)}{R}} Let \(d\geq0\) be even, and put \[ \Gamma:=\operatorname{SL}_2(\Z). \] For now, let \(R\) be a commutative ring. Let \[ V_d(R):=\operatorname{Sym}^d(R^2) \] be the \(R\)-module of homogeneous polynomials of degree \(d\) in \(X\) and \(Y\). We regard \(V_d(R)\) as a right \(\Z[M_2(\Z)]\)-module via \[ (P\mid\gamma)(X,Y) := P(aX+bY,cX+\delta Y), \qquad \gamma= \begin{pmatrix} a&b\\ c&\delta \end{pmatrix} \in M_2(\Z). \] Let \[ \operatorname{St}_R := \operatorname{Div}^0\!\left(\mathbf P^1(\Q)\right)\otimes_{\Z}R \] be the Steinberg module over \(R\). Equivalently, \(\operatorname{St}_R\) is generated by symbols \[ \{\alpha,\beta\}, \qquad \alpha,\beta\in\mathbf P^1(\Q), \] subject to the relations \[ \{\alpha,\alpha\}=0, \qquad \{\alpha,\beta\}+\{\beta,\eta\}+\{\eta,\alpha\}=0. \] We equip \(\operatorname{St}_R\) with the right \(\Gamma\)-action induced by fractional linear transformations, and \[ \operatorname{St}_R\otimes_RV_d(R) \] with the diagonal right action. Following \cite[Notation~2.1 and Definition~3.1]{WieseModularSymbols}, define the \emph{modular symbols of level \(1\) with coefficients in \(R\)} to be the coinvariants \[ \mathbb M_d(R) := H_0\!\left( \Gamma, \operatorname{St}_R\otimes_RV_d(R) \right). \] Let \[ S:= \begin{pmatrix} 0&-1\\ 1&0 \end{pmatrix}, \qquad U:= \begin{pmatrix} 1&-1\\ 1&0 \end{pmatrix}. \] By \cite[Theorem~3.5(1)]{WieseModularSymbols}, there is an isomorphism \[ \mathbb M_d(R) \simeq \frac{V_d(R)} {V_d(R)(1+S)+V_d(R)(1+U+U^2)}. \] We refer to the right-hand side as the \emph{Manin presentation} of \(\mathbb M_d(R)\). Finally, let \[ \iota:= \begin{pmatrix} -1&0\\ 0&1 \end{pmatrix}. \] Since \(\iota\) normalizes \(\Gamma\), its action on \(\operatorname{St}_R\otimes_RV_d(R)\) induces an involution on \(\mathbb M_d(R)\). For $\epsilon \in \{+1, -1\}$, define \[ \mathbb M_d^\epsilon(R) := \mathbb M_d(R)^{\iota=\epsilon} = \ker( \iota - \epsilon). \]

\subsection{Hecke operators} For \(n\geq1\) with \((n,p)=1\), let \[ \mathcal H_n := \left\{ \begin{pmatrix} a&b\\ c&\delta \end{pmatrix} \in M_2(\Z) : a>b\geq0,\quad \delta>c\geq0,\quad a\delta-bc=n \right\}. \] By \cite[\S3.3, Proposition~20, pp.~29--30]{Merel1994}, the element \[ \sum_{M\in\mathcal H_n}M \in\Z[M_2(\Z)] \] satisfies Merel's condition \((C_n)\). Hence \cite[\S2.1, Theorem~4, pp.~16--17]{Merel1994} gives an operator \[ T_n:\mathbb M_d(R)\longrightarrow\mathbb M_d(R) \] which, in the Manin presentation of the preceding subsection, is given by \[ T_n[P] = \sum_{M\in\mathcal H_n}[P\mid M]. \] Here it is the sum over the Heilbronn--Merel family that descends to the Manin quotient; the individual maps \(P\mapsto P\mid M\) need not do so. Although Merel formulates the Hecke theory over \(\C\), the construction is integral: the modular-symbol construction may be carried out over an arbitrary coefficient ring, and these Hecke operators preserve the integral structure \cite[\S1.1, p.~4 and \S2.1, Remark~1, p.~18]{Merel1994}. Thus the preceding formula defines \(T_n\) on \(\mathbb M_d(R)\) for every commutative ring \(R\). The operators \(T_n\) commute with the involution \(\iota\), and hence preserve \(\mathbb M_d^{\pm}(R)\); see \cite[\S2.1, Remark~2, p.~18]{Merel1994}.

\subsection{Comparison with modular forms} \label{sec:target-lattice} Let \[ \mathbb M_d^{\mathrm{tf}}(\Z) := \mathbb M_d(\Z)/\mathbb M_d(\Z)_{\mathrm{tors}} \] be the torsion-free quotient of the integral level-one modular-symbol module. Put \[ t:= \begin{pmatrix} 1&1\\ 0&1 \end{pmatrix}, \] and let \(B_d^{\mathrm{tf}}(\Z)\) be the torsion-free quotient of \[ \frac{V_d(\Z)}{V_d(\Z)(t-1)}. \] By \cite[Theorem~3.5(2)--(3)]{WieseModularSymbols}, the boundary map is defined integrally and induces a homomorphism \[ \partial_d^{\mathrm{tf}}: \mathbb M_d^{\mathrm{tf}}(\Z) \longrightarrow B_d^{\mathrm{tf}}(\Z). \] For \(\epsilon \in \{+1,-1\}\), define \[ P_d^\epsilon(\Z) := \left(\mathbb M_d^{\mathrm{tf}}(\Z)\right)^{\iota=\epsilon} \] and \[ L_d^\epsilon(\Z) := \ker\!\left( \left. \partial_d^{\mathrm{tf}} \right|_{P_d^\epsilon(\Z)} \right). \] 

The lattice \(P_d^\epsilon(\Z)\) is saturated in \(\mathbb M_d^{\mathrm{tf}}(\Z)\). Indeed, \[\mathbb M_d^{\mathrm{tf}}(\Z)/P_d^\epsilon(\Z)\simeq\operatorname{im}(\iota-\epsilon),\] and the latter is a subgroup of the torsion-free group \(\mathbb M_d^{\mathrm{tf}}(\Z)\). The lattice \(L_d^\epsilon(\Z)\) is saturated in \(P_d^\epsilon(\Z)\), as can be seen from \[ P_d^\epsilon(\Z)/L_d^\epsilon(\Z) \cong \operatorname{im}\!\left( \left. \partial_d^{\mathrm{tf}} \right|_{P_d^\epsilon(\Z)} \right) \subset B_d^{\mathrm{tf}}(\Z). \] In particular, \(L_d^\epsilon(\Z)\) is saturated in \(\mathbb M_d^{\mathrm{tf}}(\Z)\).

The Hecke operators preserve \(P_d^\epsilon(\Z)\) and \(L_d^\epsilon(\Z)\): they preserve the integral modular-symbol structure, commute with the involution \(\iota\), and are compatible with the boundary map; see \cite[\S2.1, Proposition~11 and the following Remark, p.~18]{Merel1994}. After extension of scalars to \(\C\), the Eichler-Shimura period pairing, in the modular-symbol formulation of Merel, restricts to a nondegenerate pairing \[ S_{d+2}(\C) \times \left(L_d^+(\Z)\otimes_{\Z}\C\right) \longrightarrow \C; \] see \cite[\S1.5, Theorem~3, pp.~10--11, and \S1.6, Proposition~8, p.~12]{Merel1994}. Moreover, the Hecke operators on modular forms and modular symbols are adjoint with respect to this pairing \cite[\S2.1, Proposition~10, pp.~17--18]{Merel1994}. Consequently, \(S_{d+2}(\C)\) and \(L_d^+(\Z)\otimes_{\Z}\C\) carry the same systems of Hecke eigenvalues. 

We now pass to local coefficients. For \(m \geq 0\), let \[R_m := \begin{cases} \Z_p, \qquad m = 0, \\ \Z_p / p^m \Z_p, \qquad m \geq 1.\end{cases}.\] Write \[ \mathbb M_d^{\mathrm{tf}}(R_m) := \mathbb M_d^{\mathrm{tf}}(\Z)\otimes_{\Z}R_m, \qquad P_d^\epsilon(R_m) := P_d^\epsilon(\Z)\otimes_{\Z}R_m, \qquad L_d^\epsilon(R_m) := L_d^\epsilon(\Z)\otimes_{\Z}R_m. \] When $R_m = \Z_p$, we will simply write \(\mathbb M_d^{\mathrm{tf}}, P_d^\epsilon\), and \(L_d^\epsilon\). The lattices \(P_d^\epsilon\) and \(L_d^\epsilon\) are saturated in \(\mathbb M_d^{\mathrm{tf}}\) and \(P_d^\epsilon\), respectively. In particular \[P_d^\epsilon/p^mP_d^\epsilon\hookrightarrow\mathbb M_d^{\mathrm{tf}}/p^m\mathbb M_d^{\mathrm{tf}}.\]

\subsection{Dickson multipliers} Fix \(m\geq1\). Define \[ \alpha_p := \frac{X^{p^2}Y-XY^{p^2}}{X^pY-XY^p} = \sum_{j=0}^p X^{(p-j)(p-1)}Y^{j(p-1)} \] and \[ \beta_p:=X^pY-XY^p. \] 

\begin{remark} \label{rem:dickson-motivation} The choice of \(\alpha_p\) and \(\beta_p\) is motivated by Dickson's theorem on invariants. In dimension \(2\), this theorem states that \[ \F_p[X,Y]^{\operatorname{GL}_2(\F_p)} = \F_p[\bar \alpha_p,\bar \beta_p^{\,p-1}], \] where bars denote reduction modulo \(p\). See \cite[Theorem~A, p.~699]{Steinberg1987}. \end{remark} 

Define \[ A_m:=\alpha_p^{p^{m-1}} \qquad\text{and}\qquad B_m:=\beta_p^{p^{m-1}}. \] Their degrees are \[ a_m:=\deg A_m=p^m(p-1), \qquad b_m:=\deg B_m=p^{m-1}(p+1). \] The same prime-power lifts of the Dickson invariants occur in Deng's \cite{Deng2019} study of torsion in the cohomology of \(\operatorname{SL}_2(\mathbf Z)\); his \(f_{2,m}\) and \(f_{1,m}\) correspond to our \(A_m\) and \(B_m\), respectively.

\begin{lemma} \label{lem:coprime-invariants} The reductions \[ \bar A_m,\bar B_m\in\F_p[X,Y] \] are coprime. \end{lemma} 

\begin{proof} Since \[ \bar A_m=\bar \alpha_p^{\,p^{m-1}} \qquad\text{and}\qquad \bar B_m=\bar \beta_p^{p^{m-1}}, \] it is enough to prove that \(\bar \alpha_p\) and \(\bar \beta_p\) are coprime. Now \[ \bar \beta_p = XY(X^{p-1}-Y^{p-1}). \] The polynomial \(\bar \alpha_p\) is divisible by neither \(X\) nor \(Y\). Moreover, we have \[ \bar \alpha_p \equiv (p+1)Y^{p(p-1)} = Y^{p(p-1)} \pmod{X^{p-1} - Y^{p-1}}. \] Hence any common divisor of \(\bar \alpha_p\) and \(X^{p-1}-Y^{p-1}\) must divide \(Y^{p(p-1)}\). Since \[ \gcd(X^{p-1}-Y^{p-1},Y)=1, \] this common divisor is a unit. Thus \[ \gcd(\bar \alpha_p,\bar \beta_p)=1, \] and the result follows. \end{proof} 

Define the coefficient multiplication map \[ \Psi_d: V_{d-a_m}(\Z_p)\oplus V_{d-b_m}(\Z_p) \longrightarrow V_d(\Z_p) \] by \begin{equation} \Psi_d(F,G):=A_mF+B_mG.  \label{eq:mult-map} \end{equation} 

\begin{proposition}[Surjective coefficient transfer]\label{prop:surjective-coefficient-transfer} If \[d\geq a_m+b_m=p^{m-1}(p^2 + 1),\] then $\Psi_d$ is surjective. 
\end{proposition}

The same numerical cutoff appears in \cite[Lemma~5.33]{Deng2019}: in Deng's notation, the corresponding quotient of polynomial coinvariants vanishes for \(d>p^{m+1}+p^{m-1}-2=a_m+b_m-2\).

\begin{proof} By Lemma~\ref{lem:coprime-invariants}, the reductions \(\bar A_m\) and \(\bar B_m\) are coprime in \(\F_p[X,Y]\). Reducing \(\Psi_d\) modulo \(p\), we obtain
\[\bar\Psi_d:V_{d-a_m}(\F_p)\oplus V_{d-b_m}(\F_p)\longrightarrow V_d(\F_p),\qquad (F,G)\longmapsto \bar A_mF+\bar B_mG.\]

Suppose \[\bar A_mF+\bar B_mG=0. \] Since \(\bar A_m\) and \(\bar B_m\) are coprime, we have \[\bar A_m\mid G,\qquad\bar B_m\mid F.\] Thus there is a unique \[H\in V_{d-a_m-b_m}(\F_p) \] such that \[(F,G)=(-\bar B_mH,\bar A_mH). \] Hence \[ \dim_{\F_p}\ker\bar\Psi_d = d-a_m-b_m+1.\]

Since \(d\geq a_m+b_m\), it follows that \[\begin{aligned}\dim_{\F_p}\operatorname{im}\bar\Psi_d &= (d-a_m+1)+(d-b_m+1)-(d-a_m-b_m+1)\\ &= d+1\\ &= \dim_{\F_p}V_d(\F_p).\end{aligned}\] Therefore \(\bar\Psi_d\) is surjective.

The cokernel of \(\Psi_d\) is a finitely generated \(\Z_p\)-module, and its reduction modulo the maximal ideal \((p)\) is the cokernel of \(\bar\Psi_d\), which is zero. Nakayama's lemma therefore gives \[\operatorname{coker}\Psi_d=0.\] Thus \(\Psi_d\) is surjective.
\end{proof}

\subsection{Transfer map for modular symbols}
We will now descend the multiplication map \eqref{eq:mult-map} to a transfer map on modular symbols modulo \(p^m\).

Define the group homomorphism \[\chi_m : R_m^\times \rightarrow R_m^\times, \qquad u \mapsto u^{p^{m-1}}.\]

\begin{lemma} \label{lem:dickson-transform} Let \[\gamma\in M_2(\Z),\qquad p\nmid\det\gamma.\] Then \[A_m\mid\gamma=A_m, \qquad B_m\mid\gamma = \chi_m(\det\gamma)B_m \qquad\text{in }R_m[X,Y]. \]
\end{lemma}

\begin{proof} Modulo \(p\), one has \[\bar \alpha_p\mid\gamma=\bar \alpha_p, \qquad \bar \beta_p\mid\gamma=\det(\gamma)\bar \beta_p. \] These identities follow by a direct calculation. Raising them to the power \(p^{m-1}\) gives the stated congruences modulo \(p^m\), by the binomial theorem.
\end{proof}

\begin{proposition}
\label{prop:transfer-map-surjective}
The coefficient multiplication map of \eqref{eq:mult-map} induces a homomorphism \[\Psi_d: \mathbb M_{d-a_m}(R_m)\oplus
\mathbb M_{d-b_m}(R_m) \longrightarrow \mathbb M_d(R_m).\] Moreover, if \(p\nmid n\), then \begin{equation} T_n\circ\Psi_d =\Psi_d\bigl(T_n\oplus\chi_m(n)T_n\bigr).\label{eq:twisted-hecke-equivariance}\end{equation} If \[d\geq a_m+b_m=p^{m-1}(p^2+1), \] then \(\Psi_d\) is surjective.\end{proposition}

\begin{proof} For \[F\in V_{d-a_m}(R_m), \qquad G\in V_{d-b_m}(R_m), \] Lemma~\ref{lem:dickson-transform} gives \begin{equation}\begin{aligned} \Psi_d(F,G)\mid\gamma &= (A_mF+B_mG)\mid\gamma\\ &= (A_m\mid\gamma)(F\mid\gamma) + (B_m\mid\gamma)(G\mid\gamma)\\ &= A_m(F\mid\gamma) +\chi_m(\det\gamma)B_m(G\mid\gamma)\\ &= \Psi_d\bigl(F\mid\gamma, \chi_m(\det\gamma)(G\mid\gamma) \bigr).\end{aligned}\label{eq:coefficient-equivariance}\end{equation}

Since \[\det S=\det U=1, \] it follows that \[\Psi_d(F,G)\mid S = \Psi_d(F\mid S,G\mid S) \] and \[ \Psi_d(F,G)\mid U = \Psi_d(F\mid U,G\mid U). \] Hence \(\Psi_d\) preserves the Manin relations, and therefore descends to \[ \Psi_d: \mathbb M_{d-a_m}(R_m)\oplus \mathbb M_{d-b_m}(R_m) \longrightarrow \mathbb M_d(R_m). \]

Now let \(p\nmid n\). Every matrix \(M\in\mathcal H_n\) has determinant \(n\), so \eqref{eq:coefficient-equivariance} gives \[ \begin{aligned} T_n\Psi_d(F,G) &= \sum_{M\in\mathcal H_n} \Psi_d(F,G)\mid M\\ &= \sum_{M\in\mathcal H_n} \Psi_d(F\mid M,\chi_m(n)(G\mid M))\\ &= \Psi_d\bigl(T_nF,\chi_m(n)T_nG\bigr). \end{aligned} \] This proves \eqref{eq:twisted-hecke-equivariance}.

Finally, if \[d\geq p^{m-1}(p^2+1),\] then the coefficient multiplication map is surjective by Proposition~\ref{prop:surjective-coefficient-transfer}. Hence the
induced map on the Manin quotients is also surjective.
\end{proof}

\subsection{For odd \texorpdfstring{\(p\)}{p}}

In this subsection, suppose that \(p\) is odd. For \(\epsilon\in\{+1,-1\}\), the idempotents \[e^\epsilon:=\frac{1+\epsilon\iota}{2}\] are defined over \(\Z_p\) and \(R_m\). We have \begin{equation}\begin{aligned} \mathbb M_d^\epsilon(R_m) &=e^\epsilon\mathbb M_d(R_m),\\ P_d^\epsilon &=\left(\mathbb M_d^{\mathrm{tf}}\right)^\epsilon =e^\epsilon\mathbb M_d^{\mathrm{tf}},\\ \left( \mathbb M_d^{\mathrm{tf}}/p^m\mathbb M_d^{\mathrm{tf}} \right)^\epsilon &= e^\epsilon\left(
\mathbb M_d^{\mathrm{tf}}/p^m\mathbb M_d^{\mathrm{tf}} \right) = P_d^\epsilon/p^mP_d^\epsilon. \end{aligned} \end{equation} Moreover, \begin{equation} \mathbb M_d(R_m) = \mathbb M_d^+(R_m)\oplus\mathbb M_d^-(R_m), \label{eq:pm-decomp} \end{equation} with \(e^\epsilon\) acting as the identity on \(\mathbb M_d^\epsilon(R_m)\).

Since \(p^{m-1}\) is odd, Lemma~\ref{lem:dickson-transform} gives \begin{equation}A_m\mid\iota=A_m, \qquad B_m\mid\iota=-B_m\label{eq:am-bm-iota}\end{equation} in \(R_m[X,Y].\) Thus multiplication by \(A_m\) preserves the \(\iota\)-sign, whereas multiplication by \(B_m\) reverses it.

\begin{proposition}\label{prop:plus-transfer} Suppose that \(p\) is odd. For each \[ \epsilon\in\{+1,-1\},\] there is a map \[ \Psi_d^\epsilon: \mathbb M_{d-a_m}^ {\epsilon}(R_m) \oplus \mathbb M_{d-b_m}^{-\epsilon}(R_m) \rightarrow \mathbb M_d^{\epsilon}(R_m), \] satisfying \begin{equation} T_n\circ\Psi_d^\epsilon =  \Psi_d^\epsilon \bigl(T_n\oplus\chi_m(n)T_n\bigr) \qquad ((n,p)=1).  \label{eq:signed-manin-transfer} \end{equation} Moreover, if \[ d\geq p^{m-1}(p^2+1),\] then $\Psi^\epsilon_d$ is surjective. In particular, for \(\epsilon=+1\), composing with the quotient map to \(P^+_d/p^mP^+_d\) gives a surjection \begin{equation} \Phi_d: \mathbb M_{d-a_m}^{+}(R_m) \oplus \mathbb M_{d-b_m}^{-}(R_m) \twoheadrightarrow P_d^+/p^mP_d^+.  \label{eq:plus-target-transfer} \end{equation}
\end{proposition}

\begin{proof} For \(F \in \mathbb M_{d-a_m}(R_m)\) and \(G \in \mathbb M_{d-b_m}(R_m)\), we have, by \eqref{eq:am-bm-iota}, \[ e^\epsilon\Psi_d(F,G) = \Psi_d(e^\epsilon F,e^{-\epsilon}G). \] Thus, defining $\Psi^\epsilon_d$ as the restriction of $\Psi_d$ to $\mathbb M^\epsilon_{d-a_m}(R_m) \oplus \mathbb M^{-\epsilon}_{d-b_m}(R_m)$, we see that the image of $\Psi_d^\epsilon$ lands in $\mathbb M^\epsilon_d(R_m)$. By Proposition~\ref{prop:transfer-map-surjective}, it is twisted Hecke-equivariant, and surjective if $d \geq p^{m-1}(p^2 + 1).$

Composing with the natural surjection
\[ \mathbb M_d(R_m) \twoheadrightarrow \mathbb M_d^{\mathrm{tf}}/p^m\mathbb M_d^{\mathrm{tf}} \] and applying $e^+$ gives \eqref{eq:plus-target-transfer}. 
\end{proof}

\subsection{Propagation of Hecke operator relations}

The transfer maps above allow relations on Hecke operators in lower coefficient degrees to be propagated to higher degrees.

By abuse of notation, we extend the character \(\chi_m\) to an endomorphism of the polynomial algebra in the prime-to-\(p\) Hecke operators by \[\chi_m(T_n):=\chi_m(n)T_n=n^{p^{m-1}}T_n \qquad ((n,p)=1). \]

Let \(F\) be a polynomial, with coefficients in \(R_m\), in finitely many prime-to-\(p\) Hecke operators. For \[\epsilon\in\{+1,-1\},\] the twisted Hecke equivariance of Proposition~\ref{prop:plus-transfer} gives \begin{equation} F\circ\Psi_d^\epsilon = \Psi_d^\epsilon \bigl(F\oplus\chi_m(F)\bigr).  \label{eq:two-branch-relation-equivariance} \end{equation}

\begin{corollary}[One-step propagation] \label{cor:one-step-propagation} Suppose that \(p\) is odd and \[ d\geq p^{m-1}(p^2+1). \] Let \(\epsilon\in\{+1,-1\}\). If \[F\mathbb M_{d-a_m}^{\epsilon}(R_m) = 0 \qquad \mbox{ and } \qquad \chi_m(F)\mathbb M_{d-b_m}^{-\epsilon}(R_m)= 0, \] then \[F\mathbb M_d^\epsilon(R_m)=0. \]

In particular, for \(\epsilon=+1\), \[F \left(P_d^+/p^mP_d^+\right)=0 \qquad \mbox { and } \qquad F\left(L_d^+(R_m)\right)=0. \]
\end{corollary}

\begin{proof}
The first assertion follows from \eqref{eq:two-branch-relation-equivariance} and the surjectivity of \(\Psi_d^\epsilon\) granted by Proposition~\ref{prop:transfer-map-surjective}. The final assertion follows by passing to the quotient \[ \mathbb M_d^+(R_m) \twoheadrightarrow P_d^+/p^mP_d^+ \] and then restricting to \( L_d^+(R_m). \)
\end{proof}

\begin{theorem}[Weight propagation]\label{thm:weight-propagation} Suppose that \(p\) is odd, and let \(r\) be an even integer with \(0\leq r<2p^{m-1}. \) Let \(F\) be a polynomial, with coefficients in \(R_m\), in finitely many prime-to-\(p\) Hecke operators. For \( q \geq 0 \) put \[ \epsilon_q:=(-1)^q. \] 

Suppose that, for every \(q\), \[\chi_m^q(F)\mathbb M_{r+2jp^{m-1}}^{\epsilon_q}(R_m) = 0 \qquad \left( \frac{p+1}{2}\leq j<\frac{p^2+1}{2} \right). \] Then, for every \(j\geq \frac{p + 1}{2}\) and every \(q\), \[ \chi_m^q(F) \mathbb M_{r+2jp^{m-1}}^{\epsilon_q}(R_m) = 0. \]

Consequently, for every even \(d\geq p^{m-1}(p + 1)\) satisfying \( d\equiv r\pmod{2p^{m-1}}, \) \[F \left(P_d^+/p^m P_d^+\right) = 0 \qquad\text{and}\qquad F \left( L_d^+ / p^m L_d^+\right) = 0. \]
\end{theorem}

\begin{proof} For ease of reading, let us introduce \[u := \frac{p(p-1)}{2}, \qquad v := \frac{p+1}{2}, \qquad J := u + v = \frac{p^2 + 1}{2}.\] Then \[a_m = 2p^{m-1}u, \qquad b_m = 2p^{m-1}v.\] Write \[d_j := r + 2p^{m-1}j.\] Then the transfer map can be written as \[\Psi^\epsilon_{d_j}: \mathbb M_{d_{j-u}}^\epsilon(R_m) \oplus \mathbb M_{d_{j-v}}^{-\epsilon}(R_m) \rightarrow \mathbb M_{d_j}^\epsilon(R_m). \] In this notation, the surjectivity condition of Proposition~\ref{prop:plus-transfer} can be written as \(j \geq J.\)

We argue by strong induction on \(j\). There is nothing to prove for \(v\leq j< J. \) Suppose that \( j\geq J. \) We quickly check that \[j - u \geq J - u=v, \qquad j- v\geq J-v=u\geq v.\] By Proposition~\ref{prop:plus-transfer}, \[\Psi_{d_j}^{\epsilon_q}:\mathbb M_{d_{j-u}}^{\epsilon_q}(R_m) \oplus \mathbb M_{d_{j-v}}^{-\epsilon_q}(R_m)
\twoheadrightarrow \mathbb M_{d_j}^{\epsilon_q}(R_m) \] is surjective. Recall that the second summand in the domain carries the twist \(\chi_m\).

By the induction hypothesis,
\[ \chi_m^q(F) \mathbb M_{d_{j-u}}^{\epsilon_q}(R_m)=0 \] and \[ \chi_m^{q+1}(F) \mathbb M_{d_{j-v}}^{\epsilon_{q+1}}(R_m)=0. \] By Corollary~\ref{cor:one-step-propagation}, \[ \chi_m^q(F)\mathbb M_{d_j}^{\epsilon_q}(R_m)=0, \] finishing the induction.

Taking \(q=0\) gives the assertion on \(\mathbb M_d^+(R_m)\), and hence on \(P_d^+/p^mP_d^+\) and \(L_d^+ / p^m L_d^+\). 
\end{proof}

\begin{remark}\label{rem:periodic-weight-propagation} More generally, the Hecke relation may depend periodically on the coefficient degree. The same induction applies to a finite family \(F_d\), provided that the family is closed under the two transitions \[ (d,F_d,\epsilon) \longmapsto (d-a_m,F_d,\epsilon) \] and \[ (d,F_d,\epsilon) \longmapsto (d-b_m,\chi_m(F_d),-\epsilon). \] 
\end{remark}

For \(p=2\), the determinant twist is trivial. We state the simpler propagation statement.

\begin{theorem}[Weight propagation for \(p=2\)] \label{thm:weight-propagation-p2} Suppose that \(p=2\) and \(m\geq2\). Let \(r\) be an even integer with \[0\leq r<2^{m-1}, \] and let \(F\) be a polynomial, with coefficients in \(R_m\), in finitely many odd Hecke operators. Suppose that \[F\mathbb M_{r+j2^{m-1}}(R_m) = 0 \qquad (2\leq j\leq4). \] Then \[F\mathbb M_{r+j2^{m-1}}(R_m) = 0 \] for every \(j\geq 2\).

Consequently, for every even \(d\geq 2^{m}\) satisfying \(d\equiv r\pmod{2^{m-1}},\) \[F \left(P_d^+ / 2^m P_d^+\right) = 0 \qquad\text{and}\qquad F\left( L_d^+ / 2^m L_d^+\right) = 0. \]
\end{theorem}

\begin{proof} By strong induction on \(j\), using the degree shifts \[ a_m=2^m, \qquad b_m=3\cdot2^{m-1}. \]
\end{proof}

\subsection{Calculus of linear relations}

To obtain higher congruences, we need polynomial relations in divided Hecke operators such as \(T/p^\alpha\). Division need not define an endomorphism on the finite modular-symbol modules used in the transfer argument. We therefore encode it by the linear relation \(Tx=p^\alpha y\). Such relations pass through the transfer maps and yield ordinary operator congruences after passage to the free lattice.

We use the basic calculus of linear relations developed by Arens \cite[\S~2]{Arens1961}, who also observes that its elementary constructions extend to modules over a ring. We introduce presentations of relations to make their compatibility with module homomorphisms explicit. Let $R$ be a commutative ring. If $M$ is an $R$-module, we define an \emph{$R$-linear relation on $M$} to be an $R$-submodule of $M \oplus M$. We denote by $\operatorname{LinRel}_R(M)$ the set of all $R$-linear relations on $M$. Any $R$-linear map $T \in \operatorname{End}_R(M)$ gives rise to such a relation via its graph \[\Gamma_T := \{(x,Tx): x \in M\}.\]

If $\mathcal A \in \operatorname{LinRel}_R(M)$ and $x,y \in M$, write $x \mathcal A y$ to mean that $(x,y) \in \mathcal A$. For a sequence $x_0, x_1, \ldots x_n \in M$, write \[x_0 \mathcal A x_1 \mathcal A \ldots \mathcal A x_n\] to mean that \[x_0 \mathcal A x_1, \qquad x_1 \mathcal A x_2, \qquad \ldots \qquad x_{n-1} \mathcal A x_n.\] Let \[\mathcal A(x) := \{y \in M: (x, y) \in \mathcal{A}\}.\] Define the domain of $\mathcal A$ as \[\operatorname{dom}\mathcal A := \{x \in M: \mathcal A(x) \not = \emptyset\}.\]

If $\mathcal A, \mathcal B \in \operatorname{LinRel}_R(M)$ and $r \in R$, define \[\mathcal A + \mathcal B := \{(x,y+z): (x,y) \in \mathcal A, (x,z) \in \mathcal B\},\] \[r \mathcal A = \{(x,ry): (x,y) \in \mathcal A\},\] \[\mathcal A \mathcal B = \{(x,z): \exists y \in M, (x,y) \in \mathcal B, (y,z) \in \mathcal A\}.\] 

If $f: M \rightarrow N$ is $R$-linear, we get a map $f_\ast: \operatorname{LinRel}_R(M) \rightarrow \operatorname{LinRel}_R(N)$ defined as follows: we have an $R$-linear map \[f \oplus f: M \oplus M \rightarrow N \oplus N,\] and for every $\mathcal A \in \operatorname{LinRel}_R(M)$, we can define \[f_\ast(\mathcal A) := (f \oplus f )(\mathcal A).\]

Now we make a more universal definition. Define a \emph{presentation of a linear relation over $R$} to be a triple $\mathfrak P := (P, p_0, p_1)$ of an $R$-module $P$ with two distinguished elements $p_0, p_1 \in P$. For the connection with pp-definable relations when \(P\) is finitely presented, see \cite{Prest2009}. If $\mathfrak P$ is a presentation of a linear relation over $R$ and $M$ is an $R$-module, let \[\mathfrak P(M) := \{(\phi(p_0), \phi(p_1)): \phi \in \Hom_R(P, M)\}.\] Thus $\mathfrak P(M) \in \operatorname{LinRel}_R(M)$. Moreover, if $f: M \rightarrow N$ is $R$-linear, then \[f_\ast \mathfrak P(M) \subset \mathfrak P(N).\] Consequently, \[ f\bigl(\operatorname{dom}\mathfrak P(M)\bigr) \subseteq \operatorname{dom}\mathfrak P(N). \] In particular, if \(f\) is surjective and \(\operatorname{dom}\mathfrak P(M)=M\), then \(\operatorname{dom}\mathfrak P(N)=N\). Presentations also commute with finite direct sums: \[ \mathfrak P(M_1\oplus M_2) = \mathfrak P(M_1)\oplus\mathfrak P(M_2), \] \[\operatorname{dom}\mathfrak P(M_1 \oplus M_2) = \operatorname{dom}\mathfrak{P}(M_1) \oplus \operatorname{dom} \mathfrak P(M_2).\] 

We now define the sum, product, and scalar action on presentations of linear relations over $R$. Let $\mathfrak P = (P, p_0, p_1)$ and $\mathfrak Q = (Q, q_0, q_1)$ be two such relations. Define \[\mathfrak P + \mathfrak Q := \left(\frac{P \oplus Q}{R(p_0, - q_0)}, \bar p_0 = \bar q_0, \overline{p_1 + q_1} \right)\] where bars denote images in the quotient. Also define \[\mathfrak P \mathfrak Q := \left(\frac{Q\oplus P}{R(q_1, - p_0)}, \bar q_0, \bar p_1 \right).\] For $r \in R$, define $r\mathfrak P := (P, p_0, rp_1)$.  Then, for every $M$, \[(\mathfrak P + \mathfrak Q)(M) = \mathfrak P(M) + \mathfrak Q(M),\] \[(\mathfrak P \mathfrak Q)(M) = \mathfrak P(M) \mathfrak Q(M), \] \[(r\mathfrak P)(M) = r \mathfrak P(M). \] Finally, let \[\mathfrak 1_R := (R, 1, 1).\] Then there are canonical isomorphisms \[\mathfrak P \mathfrak 1_R \cong \mathfrak 1_R \mathfrak P \cong \mathfrak P.\] Let us remark that we can make these definitions functorially. Consider the functor \[\mathscr U: \operatorname{R-Mod} \rightarrow \operatorname{R-Mod}, \qquad \mathscr U(M) := M \oplus M.\] Then a \emph{linear relation $\mathscr{P}$ over $R$} is a subfunctor $\mathscr P \subset \mathscr U$. If $\mathfrak P = (P, p_0, p_1)$ is a presentation of a linear relation over $R$, then the image subfunctor \[\mathscr P_{\mathfrak P} := \im \left(\Hom_R(P, -{}) \rightarrow \mathscr U\right)\] is a linear relation over $R$. Note however that not every linear relation will have a presentation. 

We now specify how to evaluate a polynomial at linear relations. For a linear relation \(\mathcal A\) on \(M\), put \[ \mathcal A^0:=\Gamma_{\mathrm{id}_M}, \qquad \mathcal A^{n+1}:=\mathcal A\mathcal A^n. \] Thus, for \(n\geq1\), \[ \mathcal A^n(x_0) = \{x_n\in M: \exists x_1,\ldots,x_{n-1}\in M,\  x_0\mathcal A x_1\mathcal A\cdots\mathcal A x_n\}. \] 

For a nonzero polynomial \[ F(X)=\sum_{j=0}^{n}c_jX^j\in R[X], \] define \[ F(\mathcal A):=\sum_{j:c_j\neq0}c_j\mathcal A^j. \] For the zero polynomial, put \(F(\mathcal A):=\Gamma_0\). More generally, if \[ F(X_1,\ldots,X_s) =\sum_{\boldsymbol j}c_{\boldsymbol j} X_1^{j_1}\cdots X_s^{j_s}, \] we define \[ F(\mathcal A_1,\ldots,\mathcal A_s) := \sum_{\boldsymbol j:c_{\boldsymbol j}\neq0} c_{\boldsymbol j} \mathcal A_1^{j_1}\cdots\mathcal A_s^{j_s}. \] The order of composition is the displayed order, with the rightmost relation applied first. Different monomials may use different intermediate elements. 

These definitions use the expanded polynomial. Evaluation at multivalued relations need not respect polynomial addition or multiplication; in particular, algebraic simplifications of expressions in relations are not automatically valid. For commuting endomorphisms \(T_1,\ldots,T_s\), however, \[ F(\Gamma_{T_1},\ldots,\Gamma_{T_s}) = \Gamma_{F(T_1,\ldots,T_s)}. \] 

For \(F(X)=\sum_{j=0}^n c_jX^j\), any common chain \[ x_0\mathcal A x_1\mathcal A\cdots\mathcal A x_n \] gives \[ \left(x_0,\sum_{j=0}^n c_jx_j\right)\in F(\mathcal A). \] Thus a computation using a common chain supplies an element of the polynomial relation defined above.

We will now define the \emph{divided operator} linear relation. Let $T,S \in R$. Define a presentation of a linear relation over $R$ by \[\mathfrak D_{T,S} := \left(\frac{R \oplus R}{R(Te_0 -Se_1)}, \bar e_0, \bar e_1\right), \qquad e_0 := (1,0), \qquad e_1 := (0,1).\] Then on an $R$-module $M$, \[\mathfrak D_{T,S}(M) = \{(x,y): Tx = Sy\}.\] If $\operatorname{dom} \mathfrak D_{T,S}(M) = M$, we will say that \emph{$M$ admits a division relation of $T$ by $S$} and symbolically write \[''\frac{T}{S}'' := \mathfrak D_{T,S}(M).\] 

We now specialize a bit more towards our intended applications. Let \(A\) be a commutative \(\mathbf Z_p\)-algebra, and let \(L\) be an \(A\)-module that is finite free over \(\mathbf Z_p\). Fix \(m\geq1\), and let \(M\) be an \(A/p^mA\)-module equipped with an \(A/p^mA\)-linear surjection \[ q:M\twoheadrightarrow L/p^mL. \] We apply the preceding definitions over the ring \(A/p^mA\).
Elements of \(A\) acting on \(M\) are understood through their images in \(A/p^mA\).

Let \(\mathcal B\) be a linear relation on \(M\), and let \(0\leq b\leq m\). We write \[ \mathcal B\equiv0\pmod{p^b}\quad\text{on }M \] if \[ \mathcal B(x)\cap p^bM\neq\varnothing \qquad\text{for every }x\in M. \] Equivalently, for every \(x\in M\), there exists \(y\in p^bM\) such that \((x,y)\in\mathcal B\). 

For an endomorphism \(T\) of \(M\), this definition gives \[ \Gamma_T\equiv0\pmod{p^b} \quad\Longleftrightarrow\quad TM\subseteq p^bM. \] More generally, for \(T\in A\), \[ \mathcal B\Gamma_T\equiv0\pmod{p^b} \] means that \[ \mathcal B(Tu)\cap p^bM\neq\varnothing \qquad\text{for every }u\in M. \]

For a presentation \(\mathfrak P\) over \(A/p^mA\), put \[ \mathfrak I:=\mathfrak D_{1,p^b}\mathfrak P. \] Then \[ \mathfrak I(M) = \{(x,\rho):\exists y\in\mathfrak P(M)(x),\ y=p^b\rho\}. \] Consequently, \[ \mathfrak P(M)\equiv0\pmod{p^b} \quad\Longleftrightarrow\quad \operatorname{dom}\mathfrak I(M)=M. \] Thus congruence conditions for presented relations pass through surjective module homomorphisms and hold on a finite direct sum whenever they hold on each summand.

\begin{proposition}[Division and retained precision] \label{prop:division-retained-precision} In the preceding setup, let \(T_1,\ldots,T_s\in A\), and let \(1\leq\alpha_i<m\). Put \[ \mathcal Z_i:=\mathfrak D_{T_i,p^{\alpha_i}}(M). \] For each \(i\), if \(\operatorname{dom}\mathcal Z_i=M\), then \[ T_iL\subseteq p^{\alpha_i}L. \] Consequently, \[ Z_i:=\frac{T_i}{p^{\alpha_i}} \] defines an endomorphism of \(L\).  
     
Suppose now that all the \(Z_i\) preserve \(L\), whether established this way or separately, and put \[ h:=m-\max_i\alpha_i. \] Let \(F\in\mathbf Z_p[X_1,\ldots,X_s]\), let \(S\in A\), and let \(0\leq b\leq h\). If \[ F(\mathcal Z_1,\ldots,\mathcal Z_s)\Gamma_S \equiv0\pmod{p^b}\quad\text{on }M, \] then \[ S F(Z_1,\ldots,Z_s)L\subseteq p^bL. \] Here the coefficients of \(F\) on the source are understood through their images in \(A/p^mA\).\end{proposition}

\begin{proof} Suppose first that \(\operatorname{dom}\mathcal Z_i=M\). For \(u\in L\), choose \(x\in M\) with \(q(x)=u\bmod p^mL\), and choose \(y\in M\) satisfying \[ T_i x=p^{\alpha_i}y. \] Applying \(q\) gives \[ T_i u\in p^{\alpha_i}L+p^mL=p^{\alpha_i}L. \] This proves the first assertion.

Now assume that all the \(Z_i\) preserve \(L\). Let \[ q_h:M\twoheadrightarrow L/p^hL \] be the composite of \(q\) with reduction modulo \(p^h\). If \((x,y)\in\mathcal Z_i\), applying \(q\) and cancelling \(p^{\alpha_i}\) on the free lattice gives \[ q_h(y)=Z_iq_h(x), \] since \(h\leq m-\alpha_i\). Thus every choice of output of \(\mathcal Z_i\) has, after applying \(q_h\), the same value as the integral operator \(Z_i\). Following the intermediate elements in each monomial, we therefore obtain \[ q_h(y)=F(Z_1,\ldots,Z_s)q_h(x) \] whenever \[ (x,y)\in F(\mathcal Z_1,\ldots,\mathcal Z_s). \] For any \(u\in L\), choose \(x\in M\) lifting its class modulo \(p^mL\). By hypothesis, there is \(y\in p^bM\) such that \[ (Sx,y)\in F(\mathcal Z_1,\ldots,\mathcal Z_s). \] Consequently, \[ F(Z_1,\ldots,Z_s)Su\in p^bL+p^hL=p^bL. \] Since \(A\) is commutative, \(S,Z_1,\ldots,Z_s\) commute on \(L\otimes_{\mathbf Z_p}\mathbf Q_p\), and hence on \(L\). This gives the stated inclusion.
\end{proof}

\subsection{Propagation of divided Hecke relations}
\label{sec:divided-relation-propagation}

Fix \(m\geq1\), and assume \(m\geq2\) if \(p=2\). Recall that \(R_m=\mathbf Z/p^m\mathbf Z\), and put \[ \mathbb A:=\mathbf Z_p[t_n:(n,p)=1], \qquad \mathbb A_m:=\mathbb A/p^m\mathbb A =R_m[t_n:(n,p)=1]. \] For odd \(p\), put \(\epsilon_q=(-1)^q\). We regard \[ \mathbb M_d^{\epsilon_q}(R_m) \] as an \(\mathbb A_m\)-module by letting \(t_n\) act as \(\chi_m(n)^qT_n\). With these actions, the transfer map \[ \Psi_d^{\epsilon_q}: \mathbb M_{d-a_m}^{\epsilon_q}(R_m) \oplus \mathbb M_{d-b_m}^{\epsilon_{q+1}}(R_m) \longrightarrow \mathbb M_d^{\epsilon_q}(R_m) \] is \(\mathbb A_m\)-linear and is surjective when \(d\geq a_m+b_m\). Here the two source modules carry the actions associated with orientations \(q\) and \(q+1\), respectively. 

For \(p=2\), let \(t_n\) act as \(T_n\) on \(\mathbb M_d(R_m)\). The transfer map \[ \Psi_d: \mathbb M_{d-a_m}(R_m)\oplus\mathbb M_{d-b_m}(R_m) \longrightarrow\mathbb M_d(R_m) \] is \(\mathbb A_m\)-linear and surjective in the same range. 

\begin{proposition}[Propagation of presented relations] \label{prop:presented-relation-propagation} Let \(\mathfrak P\) be a presentation of a linear relation over \(\mathbb A_m\).
     
Suppose first that \(p\) is odd, and let \(r\) be an even integer with \(0\leq r<2p^{m-1}\). Assume that \[ \operatorname{dom}\mathfrak P \bigl(\mathbb M_{r+2jp^{m-1}}^{\epsilon_q}(R_m)\bigr) = \mathbb M_{r+2jp^{m-1}}^{\epsilon_q}(R_m) \] for every \(q\geq0\) and every integer \(j\) satisfying \[ \frac{p+1}{2}\leq j<\frac{p^2+1}{2}, \] with the \(\mathbb A_m\)-actions specified above. Then the same equality holds for every \(j\geq(p+1)/2\) and every \(q\geq0\).

For \(p=2\), let \(r\) be an even integer with \(0\leq r<2^{m-1}\). If \[ \operatorname{dom}\mathfrak P \bigl(\mathbb M_{r+j2^{m-1}}(R_m)\bigr) = \mathbb M_{r+j2^{m-1}}(R_m) \qquad(2\leq j\leq4), \] then the same equality holds for every \(j\geq2\). 

In both cases, the same conclusion holds with the full-domain condition replaced by the congruence \[ \mathfrak P \bigl(\mathbb M_{r+2jp^{m-1}}^{\epsilon_q}(R_m)\bigr) \equiv0\pmod{p^b}, \] or its unsplit analogue for \(p=2\), where \(0\leq b\leq m\).
\end{proposition}

\begin{proof} Use the same strong induction on \(j\) as in Theorems~\ref{thm:weight-propagation} and~\ref{thm:weight-propagation-p2}. At each step, the induction hypothesis gives full domain on both source modules, hence on their direct sum. Full domain then passes through the surjective transfer map. 

For congruences, apply the same argument to the presentation \[ \mathfrak D_{1,p^b}\mathfrak P. \] 
\end{proof}  

As \(r\) varies, the initial degrees are exactly the even degrees in the ranges \[ b_m\leq d<a_m+b_m \qquad(p\text{ odd}), \] and \[ 2^m\leq d<5\cdot2^{m-1} \qquad(p=2). \] For odd \(p\), the specified module structures are periodic in \(q\) with period \(p-1\), so only finitely many orientations need to be checked.

As in Remark~\ref{rem:periodic-weight-propagation}, the presentation may depend on degree residues and orientation. The same induction applies provided that, at each transfer step, the source conditions are those of the target presentation evaluated on the two source modules with their respective \(\mathbb A_m\)-actions. Thus any family closed under the two degree transitions propagates in the same way. One may likewise restrict to degree residues closed under these transitions. 

To obtain operator congruences, take orientation \(q=0\) and use \[ \mathbb M_d^+(R_m) \twoheadrightarrow P_d^+/p^mP_d^+ \qquad(p\text{ odd}), \] or \[\mathbb M_d(R_m) \twoheadrightarrow \mathbb M_d^{\mathrm{tf}}/2^m\mathbb M_d^{\mathrm{tf}} \qquad(p=2). \] Give the free lattices their \(\mathbb A\)-module structures by letting \(t_n\) act as \(T_n\). Proposition~\ref{prop:division-retained-precision}, applied with \(A=\mathbb A\), then establishes integrality of the divided operators and the resulting polynomial congruences on the free lattice, under its stated hypotheses.

These conclusions restrict to \(L_d^+\). Indeed, the numerator Hecke operators preserve \(L_d^+\), and saturation implies that their integral divided operators preserve it as well. For \(p=2\), these operators also commute with \(\iota\), by cancellation on the torsion-free module, and hence preserve \(P_d^+\). Finally, saturation gives \[ L_d^+\cap p^bP_d^+=p^bL_d^+ \] for odd \(p\), and \[ L_d^+\cap 2^b\mathbb M_d^{\mathrm{tf}}=2^bL_d^+ \] for \(p=2\).

Relations with prescribed inputs are included in the same argument. For example, \[ F(\mathcal Z_1,\ldots,\mathcal Z_s)\Gamma_S \equiv0\pmod{p^b} \] is a presented relation, with the action of \(S\) twisted along with the numerator operators.

\subsection{Recursive computation of Manin quotients and ideal images}
The transfer map allows a recursive and fast computation of the Manin quotients. We describe the procedure here. Fix \(m\geq1\), assuming \(m\geq2\) if \(p=2\). Let $d \geq 0$ be even and $R_m = \Z/p^m Z$. We use the convention that \(V_d(R_m) = 0\) if \(d < 0\). 

Suppose $d < a_m + b_m$. By the proof of Proposition~\ref{prop:surjective-coefficient-transfer}, the transfer map \[\Psi_d: V_{d-a_m}(R_m)\oplus V_{d-b_m}(R_m) \longrightarrow V_d(R_m)\] is injective. In fact, its reduction over $\F_p$ is injective, hence it is split injective. Thus there exists a free \(R_m\)-submodule \(W_d\subseteq V_d(R_m)\) such that \begin{equation}\label{eq:recursive-coefficient-decomposition}V_d(R_m) = A_mV_{d-a_m}(R_m) \oplus B_mV_{d-b_m}(R_m) \oplus W_d. \end{equation} The complement \(W_d\) need not be stable under \(\operatorname{SL}_2(\Z)\). Since \(W_d\) is a direct summand of a finite free module over the local ring \(R_m\), it is itself finite free.

Put \[ H_d:= \mathbb M_{d-a_m}(R_m)\oplus \mathbb M_{d-b_m}(R_m). \] Define an \(R_m\)-linear surjection \[ \rho_d:V_d(R_m)\longrightarrow H_d\oplus W_d, \qquad A_mF+B_mG+w\longmapsto([F],[G],w). \] Let \(\mathcal R_d\) be the \(R_m\)-submodule generated by \[ \rho_d(w+w\mid S), \qquad \rho_d(w+w\mid U+w\mid U^2), \] as \(w\) ranges over a basis of \(W_d\).  

\begin{proposition}\label{prop:recursive-manin-presentation} The map \[ ([F],[G],w)\longmapsto[A_mF+B_mG+w] \] induces an isomorphism of \(R_m\)-modules \[ (H_d\oplus W_d)/\mathcal R_d \simeq \mathbb M_d(R_m). \]
\end{proposition}

\begin{proof}
For each nonnegative degree \(e\), put \[ N_e:=V_e(R_m)(1+S)+V_e(R_m)(1+U+U^2). \] Terms involving negative degrees below are omitted. By \eqref{eq:recursive-coefficient-decomposition} and
the definition of \(\rho_d\), we have \[ \ker\rho_d=A_mN_{d-a_m}+B_mN_{d-b_m}. \] Since \(A_m\) and \(B_m\) are invariant under \(S\) and \(U\), \[ N_d= A_mN_{d-a_m}+B_mN_{d-b_m} +W_d(1+S)+W_d(1+U+U^2). \] Consequently, \[ \ker\rho_d\subseteq N_d, \qquad \rho_d(N_d)=\mathcal R_d. \] Therefore \[ \mathbb M_d(R_m) =\frac{V_d(R_m)}{N_d} \simeq \frac{V_d(R_m)/\ker\rho_d}{N_d/\ker\rho_d} \simeq \frac{H_d\oplus W_d}{\mathcal R_d}.\]
\end{proof}

The Hecke actions can also be computed recursively. For \((n,p)=1\), define an endomorphism of \(H_d\oplus W_d\) by \[ \widetilde T_n(x,y,w) = \bigl(T_nx,\chi_m(n)T_ny,0\bigr) + \rho_d\left(\sum_{\gamma\in\mathcal H_n}w\mid\gamma\right). \] By twisted Hecke equivariance of the transfer map, this endomorphism preserves \(\mathcal R_d\) and induces \(T_n\) on the quotient of Proposition~\ref{prop:recursive-manin-presentation}. Thus only the Hecke images of a basis of \(W_d\) need to be computed directly. 

For odd \(p\), the construction is compatible with the plus/minus decomposition. Applying the projections \(e^\pm\), we may choose complements separately in the two eigenspaces: \[ V_d(R_m)^\epsilon = A_mV_{d-a_m}(R_m)^\epsilon \oplus B_mV_{d-b_m}(R_m)^{-\epsilon} \oplus W_d^\epsilon. \] Taking \(W_d=W_d^+\oplus W_d^-\) gives an \(\iota\)-stable complement. Since \(A_m\mid\iota=A_m\) and \(B_m\mid\iota=-B_m\), the map \(\rho_d\) is equivariant for the involution \[ (x,y,w)\longmapsto(\iota x,-\iota y,\iota w) \] on \(H_d\oplus W_d\). In particular, \(\mathcal R_d\) is stable under this involution. Applying \(e^\epsilon=(1+\epsilon\iota)/2\) therefore gives \[ \mathbb M_d^\epsilon(R_m) \simeq \frac{ \mathbb M_{d-a_m}^\epsilon(R_m) \oplus \mathbb M_{d-b_m}^{-\epsilon}(R_m) \oplus W_d^\epsilon }{ e^\epsilon\mathcal R_d }. \] To compute \(e^\epsilon\mathcal R_d\), form the relations from a basis of all of \(W_d\) and then project. In particular, the \(U\)-relations must include starting vectors of both signs. The induced Hecke operators preserve these signed quotients. For \(p=2\), we use the unsplit construction. 

The same construction permits recursive verification of presented linear relations. Give the first lower quotient the usual \(\mathbb A_m\)-action and the second the action \(t_n\mapsto\chi_m(n)T_n\). The maps induced by multiplication by \(A_m\) and \(B_m\) are then \(\mathbb A_m\)-linear.  Let \(\mathfrak P\) be a presentation over \(\mathbb A_m\). If \(\mathfrak P\) has full domain on both lower quotients, functoriality shows that their images lie in \[ \operatorname{dom} \mathfrak P\bigl(\mathbb M_d(R_m)\bigr). \] It therefore suffices to check \[ [w]\in\operatorname{dom} \mathfrak P\bigl(\mathbb M_d(R_m)\bigr) \] for each vector \(w\) in a basis of \(W_d\). Indeed, this domain is a submodule, and these classes together with the transferred generators generate \(\mathbb M_d(R_m)\). Likewise, a congruence \[ \mathfrak P\bigl(\mathbb M_d(R_m)\bigr) \equiv0\pmod{p^b} \] can be verified recursively by applying this argument to the presentation \(\mathfrak D_{1,p^b}\mathfrak P\).

The preceding constructions also apply to ideal images. Let \(I\subseteq\mathbb A_m\) be an ideal. For odd \(p\), give \(\mathbb M_d^\epsilon(R_m)\) and \(\mathbb M_{d-a_m}^\epsilon (R_m)\) the action \(t_n\mapsto\chi_m(n)^qT_n\), and \(\mathbb M_{d-b_m}^{-\epsilon}\) the action \(t_n\mapsto\chi_m(n)^{q+1}T_n\). The transfer map restricts to \[ \Psi_d^\epsilon: I\mathbb M_{d-a_m}^\epsilon(R_m) \oplus I\mathbb M_{d-b_m}^{-\epsilon}(R_m) \longrightarrow I\mathbb M_d^\epsilon(R_m). \] This restriction is surjective when \(d\geq a_m+b_m\). Indeed, the transfer is \(\mathbb A_m\)-linear with the specified actions, and a surjective \(\mathbb A_m\)-linear map \(f:M\to N\) satisfies \(f(IM)=IN\). For \(p=2\), the same statement holds for the unsplit modules, with no twist.

Since taking ideal images also commutes with finite direct sums, the same induction proves propagation of ordinary operator relations and full-domain conditions for presented relations on these ideal images, with the initial degree ranges of Proposition~\ref{prop:presented-relation-propagation}. Families depending on degree and orientation must satisfy the same transition compatibility conditions. Congruences of presented relations are interpreted within the ideal images, so the terminal submodule is \(p^bI\mathbb M_d^\epsilon(R_m)\). 

For odd \(p\) and \(d<a_m+b_m\), the recursive presentation gives \[ \begin{aligned} I\mathbb M_d^\epsilon(R_m) ={}& \Psi_d^\epsilon\left( I\mathbb M_{d-a_m}^\epsilon(R_m) \oplus I\mathbb M_{d-b_m}^{-\epsilon}(R_m) \right)\\ &+\sum_{f\in I}f[W_d^\epsilon], \end{aligned} \] where \([W_d^\epsilon]\) denotes the image of \(W_d^\epsilon\) in the Manin quotient, and \(f\) acts through the target orientation. If \(I=(f_1,\ldots,f_s)\), this equality becomes \[ \begin{aligned} I\mathbb M_d^\epsilon(R_m) ={}& \Psi_d^\epsilon\left( I\mathbb M_{d-a_m}^\epsilon(R_m) \oplus I\mathbb M_{d-b_m}^{-\epsilon}(R_m) \right)\\ &+\sum_{j=1}^s f_j[W_d^\epsilon]. \end{aligned} \] For \(p=2\), use the corresponding unsplit construction.

For recursive verification of a presented relation \(\mathfrak P\), suppose it has full domain on both lower ideal images. It then suffices to check \[ f_j[w]\in \operatorname{dom}\mathfrak P \bigl(I\mathbb M_d^\epsilon(R_m)\bigr) \] for each ideal generator \(f_j\) and each vector \(w\) in a basis of \(W_d^\epsilon\). Congruences are treated by applying the same argument to \(\mathfrak D_{1,p^b}\mathfrak P\).

\part{Computations and classifications}

Throughout this part, let \(f\) be a normalized cuspidal level-one eigenform of weight \(k\), and let \(\cO_f\) be the ring of integers of a finite extension of $\Q_p$ containing the coefficients of \(f\). We put
\[d:=k-2, \qquad r\equiv d\pmod{p^{m-1}(p-1)}, \qquad 0 \leq r < p^{m-1}(p-1).\]

\section{The prime \texorpdfstring{\(2\)}{2}: classification modulo \texorpdfstring{\(256\)}{256}}
\label{sec:p2}

For every even residue \(r\bmod128\), put \[ c_{3,r}:=1+3^{r+1}\pmod{256}, \qquad c_{5,r}:=1+5^{r+1}\pmod{256}, \] and define \[ X_r:=T_3-c_{3,r}. \] Also put \[ \varepsilon_r:= \begin{cases} 1,&r\equiv0,6,8,14\pmod{16},\\ 0,&r\equiv2,4,10,12\pmod{16}. \end{cases} \] Although \(r\) is taken modulo \(128\), these quantities depend only on \(r\bmod64\).

\subsection{Certified Hecke operator identities}

Put \[ F_r(Z):=Z^2-\varepsilon_rZ. \] The computation is performed at precision \(256\).  For \(m=8\), the Dickson multipliers have degrees \[ a_8=256, \qquad b_8=384, \qquad \mbox{ with } a_8 + b_8 = 640. \] Hence the exact induction base in Proposition~\ref{prop:presented-relation-propagation} consists, for every even \(0\leq r<128\), of the three degrees \[ d=r+j\,128, \qquad 2\leq j\leq4. \] Equivalently, these are the \(192\) even degrees \[ 256\leq d\leq638. \] 

For every such degree \(d\), put \[ \mathcal Z_r:= ''\frac{t_3 - c_{3,r}}{2^7}'' = \mathfrak D_{t_3-c_{3,r},\,2^7} \bigl(\mathbb M_d(\mathbf Z/256\mathbf Z)\bigr), \] where \(r\equiv d\pmod{128}\) and \(t_3\) acts as \(T_3\). The computation \cite[\href{https://github.com/nrustom/hecke-congruences/blob/main/playground_mod_256.ipynb}{\nolinkurl{playground_mod_256.ipynb}}]{RustomHeckeCongruences} verifies \[ \operatorname{dom}\mathcal Z_r =\mathbb M_d(\mathbf Z/256\mathbf Z), \qquad F_r(\mathcal Z_r)\equiv0\pmod2, \] and also verifies the ordinary relation \[ \bigl(T_5-c_{5,r}\bigr) \mathbb M_d(\mathbf Z/256\mathbf Z)=0. \] Since both \(a_8\) and \(b_8\) are divisible by \(128\), the quantities \(c_{3,r}\), \(c_{5,r}\), and \(\varepsilon_r\) are unchanged under the two transitions of Remark~\ref{rem:periodic-weight-propagation}.  Thus these families satisfy the compatibility required for propagation.

Proposition~\ref{prop:presented-relation-propagation} propagates these conditions to every even degree \[ d\geq256. \] The remaining even degrees \[ 0\leq d<256 \] are verified directly. By Proposition~\ref{prop:division-retained-precision}, with \(m=8\), \(\alpha_1=7\), \(b=1\), and \(S=1\), for \[ r\equiv d\pmod{128}, \] the operator \[ Z_r:=\frac{T_3-c_{3,r}}{2^7} \] belongs to \[ \End_{\mathbf Z_2}\bigl(\mathbb M_d^{\mathrm{tf}}\bigr), \] preserves \(P_d^+\) and \(L_d^+\), and satisfies \begin{equation} \bigl(Z_r^2-\varepsilon_rZ_r\bigr) \mathbb M_d^{\mathrm{tf}} \subseteq2\mathbb M_d^{\mathrm{tf}}, \qquad \bigl(Z_r^2-\varepsilon_rZ_r\bigr)L_d^+ \subseteq2L_d^+. \label{eq:p2-divided-relation}\end{equation}

Likewise, Theorem~\ref{thm:weight-propagation-p2}, together with the direct verification in the lower degrees, gives \begin{equation} \bigl(T_5-c_{5,r}\bigr)L_d^+ \subseteq256L_d^+. \label{eq:p2-T5-relation} \end{equation} Thus both relations hold in every even coefficient degree \(d\).

\subsection{Congruences for \texorpdfstring{\(a_3\)}{a3} and
\texorpdfstring{\(a_5\)}{a5}}

By the Eichler--Shimura comparison recalled in Section~\ref{sec:target-lattice}, choose a primitive simultaneous eigenvector \[ v_f\in L_d^+\otimes_{\mathbf Z_2}\mathcal O_f \] for the prime-to-\(2\) Hecke eigensystem of \(f\). 

Since \(Z_r\) preserves \(L_d^+\), its eigenvalue on \(v_f\) is integral.  Thus \[ z_3:=\frac{a_3(f)-c_{3,r}}{2^7}\in\mathcal O_f. \] Evaluating \eqref{eq:p2-divided-relation} on \(v_f\) gives \[ z_3^2-\varepsilon_rz_3\in2\mathcal O_f. \]

If \(\varepsilon_r=0\), then \[ z_3^2\in2\mathcal O_f, \] so \(v_2(z_3)>0\).  Hence \[ v_2\bigl(a_3(f)-c_{3,r}\bigr)>7, \] which is precisely \[ a_3(f)\equiv_{\mathrm{val}}c_{3,r}\pmod{256}. \]

If \(\varepsilon_r=1\), then \[ z_3(z_3-1)\in2\mathcal O_f. \] The two factors differ by a unit, so exactly one of them belongs to the maximal ideal of \(\mathcal O_f\).  Therefore \[z_3\equiv0\quad\text{or}\quad1 \pmod{\mathfrak m_{\mathcal O_f}}, \] and consequently \[ a_3(f)\equiv_{\mathrm{val}} c_{3,r}\quad\text{or}\quad c_{3,r}+2^7 \pmod{256}. \]

Finally, evaluating \eqref{eq:p2-T5-relation} on \(v_f\) gives the literal, and hence valuative, congruence \[ a_5(f)\equiv_{\mathrm{lit}}c_{5,r}\pmod{256}. \] Thus \begin{equation} a_3(f)\equiv_{\mathrm{val}} \begin{cases} c_{3,r}\pmod{256},&\varepsilon_r=0,\\ c_{3,r}\ \text{or}\ c_{3,r}+128\pmod{256},  &\varepsilon_r=1, \end{cases} \qquad a_5(f)\equiv_{\mathrm{lit}}c_{5,r}\pmod{256}. \label{eq:p2-coordinate-classification} \end{equation}

\subsection{Final classification}

\begin{proof}[Proof of Theorem~\ref{thm:intro-p2}] For \(r\equiv k-2\pmod{128}\), direct calculation gives
\[
\begin{array}{c|cc}
 & a_3 & a_5\\ \hline
\theta^{64}G_k
 & c_{3,r} & c_{5,r}\\
\theta^{16}G_{k+32}
 & c_{3,r}+128 & c_{5,r}
\end{array}
\qquad\pmod{256}.
\]
Thus \eqref{eq:p2-coordinate-classification} selects \(\theta^{64}G_k\) when \(\varepsilon_r=0\), and exactly one of the two twists when \(\varepsilon_r=1\). 

Write the selected twist as \(\theta^iG_\kappa\). By Lemma~\ref{lem:weak-cyclotomic}, its reduction modulo \(256\) is a weak eigenform of weight congruent to \(\kappa+2i\equiv k\pmod{64}\). It therefore agrees with \(f\) at the weight operator \([5]\), as well as at \(T_3\) and \(T_5\), modulo \(256\) in the valuative sense. Since \[ \mathbb T_{\mathrm{dc}}^{(2)} =\overline{\mathbf Z_2[T_3,T_5,[5]]}, \] the two prime-to-\(2\) Hecke eigensystems agree. As \(n^{64}\equiv1\pmod{256}\) for odd \(n\), and both twists vanish modulo \(256\) at even indices, this gives \[ \theta^{64}f \equiv_{\mathrm{val}}\theta^iG_\kappa\pmod{256}. \] 

Conversely, for every even weight residue modulo \(64\), the computation supplies a normalized cuspidal level-one eigenform realizing each permitted pair of \(T_3\)- and \(T_5\)-values in that weight residue. The same completed-Hecke generation argument proves agreement of the entire prime-to-\(2\) eigensystems.

Different weight residues are distinguished by \[ a_3(f)^2-a_9(f)=3^{k-1}, \] since \(3\) has order \(64\) modulo \(256\). Within a weight residue, the two permitted systems, when both occur, are distinguished by their \(T_3\)-values. There are \(32+16=48\) permitted systems altogether, and every one is realized.
\end{proof}

\section{The prime \texorpdfstring{\(3\)}{3}: classification modulo \texorpdfstring{\(81\)}{81}} \label{sec:p3} 

Throughout this section, put \[ r\equiv d=k-2\pmod{54},\qquad 0\leq r<54. \] We call \(r\equiv2\pmod6\), equivalently \(k\equiv4\pmod6\), the \emph{zero branch}, and \(r\equiv0,4\pmod6\), equivalently \(k\equiv0,2\pmod6\), the \emph{nonzero branch}. As a consequence of the calculations below, these branches are also characterized by \[\begin{aligned} \text{zero branch} &\Longleftrightarrow a_2(f)\in9\mathcal O_f,\\ \text{nonzero branch} &\Longleftrightarrow a_2(f)\notin9\mathcal O_f. \end{aligned} \]

\subsection{Certified Hecke operator identities}

\subsubsection{The common \(T_7\) relation}

Let \(t_r\in\{0,\ldots,80\}\) represent \(1+7^{r+1}\pmod{81}\). We work at precision \(81\), where \[ a_4=162,\qquad b_4=108,\qquad a_4+b_4=270. \] For every even degree \(0\leq d<270\) and both orientations \(q\in\{0,1\}\), put \[ \mathcal N_{r,q}:= ''\frac{t_7-t_r}{27}'' = \mathfrak D_{t_7-t_r,\,27} \bigl(\mathbb M_d^{\epsilon_q}(\mathbf Z/81\mathbf Z)\bigr), \qquad \epsilon_q=(-1)^q. \] Here \(t_7\) acts as \(T_7\), since \(\chi_4(7)=7^{27}\equiv1\pmod{81}\). The computation \cite[\href{https://github.com/nrustom/hecke-congruences/blob/main/playground_mod_81.ipynb}{\nolinkurl{playground_mod_81.ipynb}}] {RustomHeckeCongruences} verifies \[ \operatorname{dom}\mathcal N_{r,q} =\mathbb M_d^{\epsilon_q}(\mathbf Z/81\mathbf Z), \qquad \mathcal N_{r,q}^{\,2}\equiv0\pmod3. \]

Both degree shifts preserve \(r\pmod{54}\). Proposition~\ref{prop:presented-relation-propagation}, together with the direct verification in degrees \(d<108\), therefore propagates these conditions to every even degree. Proposition~\ref{prop:division-retained-precision}, with \(m=4\), \(\alpha_1=3\), \(b=1\), and \(S=1\), gives \[ N_r:=\frac{T_7-t_r}{27}\in\End_{\mathbf Z_3}(P_d^+), \qquad N_r^2P_d^+\subseteq3P_d^+. \] These conclusions restrict to \(L_d^+\) by saturation.

\subsubsection{The nonzero branch}

Suppose \(r\equiv0,4\pmod6\). Define \(c_r\) by \begin{equation}
\begin{array}{c|rrrrrrrrrrrrrrrrrr}
r
&0&4&6&10&12&16&18&22&24&28&30&34&36&40&42&46&48&52\\
\hline
c_r
&1&4&4&1&7&-2&10&-5&13&-8&16&16&-8&13&-5&10&-2&7
\end{array}
\label{eq:p3-nonzero-constants}
\end{equation}
Put \(F(X):=X^4-2X^3+X^2\). We work at precision \(243\), where \[ a_5=486,\qquad b_5=324,\qquad a_5+b_5=810. \] For every even degree \[ 0\leq d<810,\qquad d\equiv r\pmod{54}, \] and both orientations \(q\in\{0,1\}\), put \[ \mathcal Q_{r,q}:= ''\frac{t_2^2-9c_r}{81}'' = \mathfrak D_{t_2^2-9c_r,\,81} \bigl(\mathbb M_d^{\epsilon_q}(\mathbf Z/243\mathbf Z)\bigr), \qquad \epsilon_q=(-1)^q, \] where \(t_2\) acts as \(\chi_5(2)^qT_2=(-1)^qT_2\). The computation \cite[\href{https://github.com/nrustom/hecke-congruences/blob/main/playground_mod_81.ipynb}{\nolinkurl{playground_mod_81.ipynb}}]{RustomHeckeCongruences} verifies \[ \operatorname{dom}\mathcal Q_{r,q} =\mathbb M_d^{\epsilon_q}(\mathbf Z/243\mathbf Z), \qquad F(\mathcal Q_{r,q})\equiv0\pmod3. \]

Both degree shifts preserve \(r\pmod{54}\), and the twist leaves \(t_2^2\) unchanged. Hence Proposition~\ref{prop:presented-relation-propagation}, together with the direct verification in degrees \(d<324\), propagates these conditions to every even degree \(d\equiv r\pmod{54}\). Proposition~\ref{prop:division-retained-precision}, with \(m=5\), \(\alpha_1=4\), \(b=1\), and \(S=1\), gives \[ Q_r:=\frac{T_2^2-9c_r}{81} \in\End_{\mathbf Z_3}(P_d^+), \qquad F(Q_r)P_d^+\subseteq3P_d^+. \] By saturation, these conclusions restrict to \(L_d^+\): \begin{equation} Q_rL_d^+\subseteq L_d^+, \qquad \bigl(Q_r(Q_r-1)\bigr)^2L_d^+\subseteq3L_d^+. \label{eq:p3-nonzero-relations} \end{equation}

\subsubsection{The zero branch}

Suppose \(r\equiv2\pmod6\). We work at precision \(2187\), where \[ a_7=4374,\qquad b_7=2916,\qquad a_7+b_7=7290. \] Put \(I=(9,t_2)\subseteq\mathbb A\). For every degree \[ 0\leq d<7290,\qquad d\equiv r\pmod{54}, \] and both orientations \(q\in\{0,1\}\), we work on \[ I\mathbb M_d^{\epsilon_q}(\mathbf Z/2187\mathbf Z), \qquad \epsilon_q=(-1)^q, \] where \(t_2\) acts as \(\chi_7(2)^qT_2=(-1)^qT_2\).

Define the polynomials \(h_r(X)\in\mathbf Z[X]\) by \[ \begin{array}{c|c} r & h_r(X)\\ \hline 2,8,20,26,38,44 & 1-X^2\\ 14 & 2+5X\\ 32,50 & 2+X^2 \end{array} \] and put \[ G_r(X):=(X^3-X)^2+3h_r(X)(X^3-X). \] On the above ideal image, define the linear relations \[ \mathcal A_q:= ''\frac{t_2}{9}'' = \mathfrak D_{t_2,\,9} \bigl(I\mathbb M_d^{\epsilon_q} (\mathbf Z/2187\mathbf Z)\bigr), \] and \[ \mathcal B_{r,q}:= ''\frac{G_r(\mathcal A_q)}{9}'' = \mathfrak D_{1,\,9}\circ G_r(\mathcal A_q). \] Here \(G_r(\mathcal A_q)\) is evaluated using the expanded polynomial convention. All these relations are formed on \(I\mathbb M_d^{\epsilon_q}(\mathbf Z/2187\mathbf Z)\); in particular, every intermediate element introduced in a composition or division belongs to this ideal image. 

Define the sets \(E_{r,s}\) by
\[
\begin{array}{c|ccc}
r & E_{r,0} & E_{r,1} & E_{r,2}\\ \hline
2,20,38 & \{0,2\} & \{0,1\} & \{0,1\}\\
8,26,44 & \{0,2\} & \{1,2\} & \{1,2\}\\
14      & \{0,2\} & \{0,2\} & \{0,1\}\\
32,50   & \{0,2\} & \{0,1\} & \{0,1\}
\end{array}
\]
and put \[ F_0(Y):=(Y^3-Y)^2, \qquad F_{r,s}(X,Y):= \bigl(1-(X-s)^2\bigr)^6 \prod_{\delta\in E_{r,s}}(Y-\delta)^2 \quad(s=0,1,2). \] These polynomials are expanded before evaluation at the linear relations. In each monomial \(X^iY^j\), the relation \(\mathcal B_{r,q}\) is applied first, followed by \(\mathcal A_q\).  For every degree and orientation in the above range, the computation \cite[\href{https://github.com/nrustom/hecke-congruences/blob/main/playground_mod_81.ipynb}{\nolinkurl{playground_mod_81.ipynb}}] {RustomHeckeCongruences} verifies \begin{equation} F_0(\mathcal B_{r,q})\equiv0\pmod3, \qquad F_{r,s}(\mathcal A_q,\mathcal B_{r,q}) \equiv0\pmod3 \quad(s=0,1,2) \label{eq:p3-mod81-zero-source-relations} \end{equation} on \(I\mathbb M_d^{\epsilon_q}(\mathbf Z/2187\mathbf Z)\).

Both degree shifts preserve \(r\pmod{54}\), and the ideal \(I=(9,t_2)\) is unchanged by the twist \(t_2\mapsto-t_2\). The presentations are therefore compatible with both transfer transitions. By the ideal-image version of Proposition~\ref{prop:presented-relation-propagation} discussed in Section~3.10, the verification in \[ 2916\leq d<7290,\qquad d\equiv r\pmod{54}, \] propagates \eqref{eq:p3-mod81-zero-source-relations} to every degree \(d\geq2916\) in this residue class. Together with the direct verification in degrees \(d<2916\), this establishes these relations in every zero-branch degree, for both orientations.

We now take \(q=0\) and put \[ J_d:=9P_d^++T_2P_d^+\subseteq P_d^+. \] This is a Hecke-stable lattice containing \(9P_d^+\). The natural surjection \[ \mathbb M_d^+(\mathbf Z/2187\mathbf Z) \twoheadrightarrow P_d^+/3^7P_d^+ \] restricts to a surjection \[ I\mathbb M_d^+(\mathbf Z/2187\mathbf Z) \twoheadrightarrow J_d/3^7P_d^+. \] Since \(9P_d^+\subseteq J_d\), we have \(3^7P_d^+\subseteq3^5J_d\), and hence a surjection \[ I\mathbb M_d^+(\mathbf Z/2187\mathbf Z) \twoheadrightarrow J_d/3^5J_d. \] Thus \eqref{eq:p3-mod81-zero-source-relations} also holds on \(J_d/3^5J_d\).

The first relation implies that \(\mathcal B_{r,0}\) has full domain, since \(F_0(Y)\) is monic of degree \(6\). Since \(G_r(X)\) is also monic of degree \(6\), the definition of \(\mathcal B_{r,0}\) then implies that \(\mathcal A_0\) has full domain. Consequently, \[ T_2J_d\subseteq9J_d, \qquad A:=\frac{T_2}{9}\in\End_{\mathbf Z_3}(J_d). \] Cancelling \(9\) in the defining equation for \(\mathcal A_0\) shows that every value of \(\mathcal A_0(x)\) agrees with \(Ax\) modulo \(3^3J_d\). The same is therefore true for \(G_r(\mathcal A_0)\) and \(G_r(A)\).  

Full domain of \(\mathcal B_{r,0}\) now gives \[ G_r(A)J_d\subseteq9J_d, \qquad B_r:=\frac{G_r(A)}9 \in\End_{\mathbf Z_3}(J_d). \] Cancelling \(9\) once more shows that every value of \(\mathcal B_{r,0}(x)\) agrees with \(B_rx\) modulo \(3J_d\). Thus \eqref{eq:p3-mod81-zero-source-relations} gives \begin{equation} F_0(B_r)J_d\subseteq3J_d, \qquad F_{r,s}(A,B_r)J_d\subseteq3J_d \quad(s=0,1,2). \label{eq:p3-mod81-zero-target-relations} \end{equation}

\subsection{Congruences for \texorpdfstring{\(a_2\)}{a2} and \texorpdfstring{\(a_7\)}{a7}}  
\subsubsection{Congruence for \texorpdfstring{\(a_7\)}{a7}} By the Eichler--Shimura comparison recalled in Section~\ref{sec:target-lattice}, choose a primitive simultaneous eigenvector \[ v_f\in L_d^+\otimes_{\mathbf Z_3}\mathcal O_f. \] Since \(N_r=(T_7-t_r)/27\) preserves \(L_d^+\) and \(N_r^2L_d^+\subseteq3L_d^+\), its eigenvalue \[ n_7:=\frac{a_7(f)-t_r}{27} \] satisfies \(n_7\in\mathcal O_f\) and \(n_7^2\in3\mathcal O_f\). Hence \(v_3(n_7)\geq\tfrac12\), giving \begin{equation} a_7(f)\equiv_{\mathrm{val}}t_r \equiv_{\mathrm{val}}1+7^{k-1}\pmod{81}. \label{eq:p3-mod81-T7} \end{equation}

\subsubsection{The nonzero branch} Suppose \(r\equiv0,4\pmod6\). Since \(Q_r=(T_2^2-9c_r)/81\) preserves \(L_d^+\), its eigenvalue \[q_2:=\frac{a_2(f)^2-9c_r}{81} \] is integral. Consequently, \[ a_2(f)^2=9(c_r+9q_2). \] Since \(c_r\equiv1\pmod3\), the factor \(c_r+9q_2\) is a unit. Thus \(v_3(a_2(f))=1\), and \[ z_2:=\frac{a_2(f)}3\in\mathcal O_f^\times, \qquad z_2^2=c_r+9q_2. \] By \eqref{eq:p3-nonzero-relations}, \[ \bigl(q_2(q_2-1)\bigr)^2\in3\mathcal O_f. \] Reducing modulo the maximal ideal of \(\mathcal O_f\) therefore gives \(\bar q_2\in\{0,1\}\). Hence \[ z_2^2\equiv_{\mathrm{val}}c_r \quad\text{or}\quad c_r+9\pmod{27}. \]

Both \(c_r\) and \(c_r+9\) are congruent to \(1\) modulo \(3\), so each has exactly two square roots modulo \(27\). These roots are simple modulo \(3\). It follows that \begin{equation} a_2(f)\equiv_{\mathrm{val}}3u\pmod{81} \quad\text{for some }u\in\mathbf Z/27\mathbf Z \text{ satisfying } u^2\equiv c_r\ \text{or}\ c_r+9\pmod{27}. \label{eq:p3-mod81-nonzero-T2} \end{equation} Thus there are four possible values of \(a_2(f)\) modulo \(81\) on each nonzero weight residue. 

\subsubsection{The zero branch}

Suppose \(r\equiv2\pmod6\). Since \(J_d\) contains \(9P_d^+\), it spans the same \(\mathbf Q_3\)-vector space as \(P_d^+\). We may therefore choose a primitive simultaneous eigenvector for \(f\) in \(J_d\otimes_{\mathbf Z_3}\mathcal O_f\). The integral operators \(A\) and \(B_r\) have eigenvalues \[ x_2:=\frac{a_2(f)}9\in\mathcal O_f, \qquad y_2:=\frac{(x_2^3-x_2)^2+ 3h_r(x_2)(x_2^3-x_2)}9\in\mathcal O_f. \] In particular, \(a_2(f)\in9\mathcal O_f\). 

Put \(c=x_2^3-x_2\). Since \[ c^2+3h_r(x_2)c = c(c + 3h_r(x_2))\in9\mathcal O_f, \] we have \(c\in3\mathcal O_f\): otherwise \(v_3(c)<1\), so \(v_3(c+3h_r(x_2))=v_3(c)\), giving \[ v_3\bigl(c(c+3h_r(x_2))\bigr)=2v_3(c)<2, \] a contradiction. Thus \[ u_2:=\frac{x_2^3-x_2}{3}\in\mathcal O_f, \qquad y_2=u_2^2+h_r(x_2)u_2. \] The factorization \(x_2^3-x_2=x_2(x_2-1)(x_2+1)\) shows that there is a unique \(s\in\{0,1,2\}\) such that \[ x_2\equiv_{\mathrm{lit}}s\pmod3. \]

Evaluating \eqref{eq:p3-mod81-zero-target-relations} and reducing modulo the maximal ideal of \(\mathcal O_f\) gives \[ \bar y_2^3=\bar y_2, \qquad \prod_{\delta\in E_{r,s}}(\bar y_2-\delta)^2=0. \] Indeed, the selector \(\bigl(1-(x_2-s)^2\bigr)^6\) reduces to \(1\). Consequently, \[ \bar y_2\in E_{r,s}, \qquad \bar u_2^{\,2}+\overline{h_r(s)}\,\bar u_2 \in E_{r,s}. \]

Write \(x_2=s+3v_2\), with \(v_2\in\mathcal O_f\). Then \[ u_2=\frac{x_2^3-x_2}{3} \equiv \frac{s^3-s}{3}-v_2 \pmod{3\mathcal O_f}. \] Thus \(\bar u_2\), together with \(s\), determines \(a_2(f)=9s+27v_2\) modulo \(81\) in the valuative sense. 

Choose \(\omega\in\overline{\mathbf Z}_3\) with \(\omega^2=2\). Solving the preceding quadratic conditions gives \begin{equation} a_2(f)\equiv_{\mathrm{val}}\alpha\pmod{81} \qquad\text{for some }\alpha\in\mathcal S_r, \label{eq:p3-mod81-zero-T2} \end{equation} where \[ 
\begin{array}{c|l}
r & \mathcal S_r\\ \hline
2,14,20,32,38,50
& \{0,9,18,27,36,45,54,63,72\}\\
8,26,44
& \{0,18,27,36,45,54,63,\,
9\pm27\omega,\,72\pm27\omega\}.
\end{array}
\]
Indeed, in the first row the permitted values of \(\bar u_2\) are precisely \(\mathbf F_3\) for every \(s\). In the second row they are \(\mathbf F_3\) when \(s=0\), and \(\{\pm1,\pm\bar\omega\}\) when \(s=1,2\).

\subsection{Final classification}

\begin{proof}[Proof of Theorem~\ref{thm:intro-p3}] By the preceding subsections, every normalized cuspidal level-one eigenform \(f\) determines a unique signature \[ \sigma=(r,\alpha), \qquad r\equiv k-2\pmod{54},\qquad 0\leq r<54, \] such that \[ a_2(f)\equiv_{\mathrm{val}}\alpha\pmod{81}, \qquad a_7(f)\equiv_{\mathrm{val}}t_r\pmod{81}. \] The permitted values are listed by weight residue \(k\equiv r+2\pmod{54}\) in Table~\ref{tab:p3-signatures}. There are \[ 18\cdot4+6\cdot9+3\cdot11=159 \] such signatures altogether.

The finite computation \cite[\url{https://github.com/nrustom/hecke-congruences/tree/main/strong_signatures/p3_m4}] {RustomHeckeCongruences} verifies that every signature is realized by a normalized cuspidal level-one eigenform \(g_\sigma\) of weight \(k_\sigma\leq214\), with \[ k_\sigma-2\equiv r\pmod{54}. \] This includes the \(147\) rational signatures and the \(12\) nonrational signatures, the latter occurring in six conjugate pairs.

If \(f\) has signature \(\sigma\), then \[ a_2(f)\equiv_{\mathrm{val}}a_2(g_\sigma)\pmod{81}, \qquad a_7(f)\equiv_{\mathrm{val}}a_7(g_\sigma)\pmod{81}, \] and \(k_\sigma\equiv k\pmod{54}\). Hence \[ 4^{k_\sigma}\equiv4^k\pmod{81}, \] so \(f\) and \(g_\sigma\) also have the same value of the weight operator \([4]\) modulo \(81\). By \eqref{eq:p3-big-generation}, \[ \mathbb T_{\mathrm{dc}}^{(3)} =\overline{\mathbf Z_3[T_2,T_7,[4]]}. \] Their prime-to-\(3\) Hecke eigensystems are therefore congruent modulo \(81\) in the valuative sense: \[ a_n(f)\equiv_{\mathrm{val}}a_n(g_\sigma)\pmod{81} \qquad ((n,3)=1). \]

Conversely, suppose that eigenforms \(f\) and \(g\), of weights \(k\) and \(K\), have congruent prime-to-\(3\) Hecke eigensystems modulo \(81\). The identities \[ a_2(f)^2-a_4(f)=2^{k-1}, \qquad a_2(g)^2-a_4(g)=2^{K-1} \] give \[ 2^{k-1}\equiv_{\mathrm{val}}2^{K-1}\pmod{81}. \] Since these are rational integers, this congruence is literal. As \(2\) has order \(54\) modulo \(81\), we obtain \(k\equiv K\pmod{54}\). Agreement at \(T_2\) then implies that their corresponding values \(\alpha\) agree. Thus distinct signatures give distinct eigensystems.

Reducing the signatures modulo \(27\), direct calculation shows that their values at \(T_2\) and \(T_7\) agree with those of the Eisenstein twists corresponding to the pairs in Theorem~\ref{thm:intro-p3}. For each signature, the pair is unique. Choose representatives \(i\geq4\) and even \(\kappa\geq4\) such that \[ \{i,i+\kappa-1\}=\{a,b\} \quad\text{in }\mathbf Z/18\mathbf Z, \qquad \kappa+2i\equiv k\pmod{54}. \] These congruences at \(T_2\) and \(T_7\) are literal modulo \(27\). By Lemma~\ref{lem:weak-cyclotomic}, \(\theta^iG_\kappa\) is a weak cuspidal eigenform modulo \(27\), with the same value of \([4]\) as \(f\). Equation~\eqref{eq:p3-big-generation} therefore gives \[ a_n(f)\equiv n^i\sigma_{\kappa-1}(n)\pmod{27} \qquad ((n,3)=1). \] Since \(n^{54}\equiv1\pmod{27}\) for \(3\nmid n\), and both twisted series vanish modulo \(27\) at indices divisible by \(3\), this proves \[ \theta^{54}f\equiv_{\mathrm{lit}} \theta^iG_\kappa\pmod{27}. \]

By Lemma~\ref{lem:modular-correction}, \[ \delta_f:= \overline{\frac{\theta^{54}f-\theta^iG_\kappa}{27}} \] is a modular form satisfying \[ \theta^{54}f\equiv_{\mathrm{val}} \theta^iG_\kappa+27\delta_f\pmod{81}. \] For fixed \(i,\kappa\), equation~\eqref{eq:p3-big-generation} shows that its \(q\)-expansion is determined by the signature and is defined over the residue field of \(\mathbf Q_3(\alpha)\). Thus \(\delta_f\) is defined over \(\mathbf F_3\) for rational signatures and over \(\mathbf F_9\) otherwise.

Consequently, there are exactly \(159\) prime-to-\(3\) Hecke eigensystems modulo \(81\), in the valuative sense, and every such eigensystem is realized in weight at most \(214\). 
\end{proof}

\section{The prime \texorpdfstring{\(5\)}{5}: classification modulo
\texorpdfstring{\(125\)}{125}}
\label{sec:p5-mod125}

Throughout this section, put \[ r\equiv d=k-2\pmod{100},\qquad 0\leq r<100. \]
For every even residue \(r\pmod{100}\), define the constants
\(\alpha_r,\beta_r,\gamma_r,v_r,h_{r,0},h_{r,1},h_{r,2}\)
by Table~\ref{tab:p5-mod125-constants}.

\begin{table}[H]
\centering
\caption{Constants for the classification modulo \(125\).}
\label{tab:p5-mod125-constants}
\small
\setlength{\tabcolsep}{3pt}
\begin{tabular}{@{}c|rrrrrrr@{}}
\toprule
\(r\) & \(\alpha_r\) & \(\beta_r\) & \(\gamma_r\)
& \(v_r\) & \(h_{r,0}\) & \(h_{r,1}\) & \(h_{r,2}\)\\
\midrule
0  & 1 & 3 & 0 & 4  & 104 & 19 & 1\\
2  & 2 & 0 & 3 & 20 & 116 & 4  & 4\\
4  & 4 & 0 & 1 & 12 & 64  & 14 & 1\\
6  & 4 & 3 & 0 & 15 & 106 & 24 & 4\\
8  & 2 & 0 & 0 & 19 & 124 & 9  & 1\\
10 & 0 & 3 & 0 & 6  & 121 & 19 & 4\\
12 & 0 & 4 & 4 & 0  & 34  & 4  & 1\\
14 & 3 & 4 & 3 & 13 & 36  & 14 & 4\\
16 & 4 & 0 & 4 & 15 & 44  & 24 & 1\\
18 & 1 & 3 & 3 & 21 & 101 & 9  & 4\\
20 & 4 & 3 & 0 & 9  & 29  & 19 & 1\\
22 & 3 & 0 & 3 & 5  & 66  & 4  & 4\\
24 & 4 & 0 & 1 & 12 & 114 & 14 & 1\\
26 & 3 & 3 & 0 & 5  & 56  & 24 & 4\\
28 & 4 & 0 & 0 & 14 & 49  & 9  & 1\\
30 & 2 & 3 & 0 & 1  & 71  & 19 & 4\\
32 & 4 & 4 & 4 & 15 & 84  & 4  & 1\\
34 & 3 & 4 & 3 & 13 & 111 & 14 & 4\\
36 & 0 & 0 & 4 & 0  & 94  & 24 & 1\\
38 & 4 & 3 & 3 & 1  & 51  & 9  & 4\\
40 & 2 & 3 & 0 & 14 & 79  & 19 & 1\\
42 & 4 & 0 & 3 & 15 & 16  & 4  & 4\\
44 & 4 & 0 & 1 & 12 & 39  & 14 & 1\\
46 & 2 & 3 & 0 & 20 & 6   & 24 & 4\\
48 & 1 & 0 & 0 & 9  & 99  & 9  & 1\\
\bottomrule
\end{tabular}
\hspace{1.5em}
\begin{tabular}{@{}c|rrrrrrr@{}}
\toprule
\(r\) & \(\alpha_r\) & \(\beta_r\) & \(\gamma_r\)
& \(v_r\) & \(h_{r,0}\) & \(h_{r,1}\) & \(h_{r,2}\)\\
\midrule
50 & 4 & 3 & 0 & 21 & 21  & 19 & 4\\
52 & 3 & 4 & 4 & 5  & 9   & 4  & 1\\
54 & 3 & 4 & 3 & 13 & 61  & 14 & 4\\
56 & 1 & 0 & 4 & 10 & 19  & 24 & 1\\
58 & 2 & 3 & 3 & 6  & 1   & 9  & 4\\
60 & 0 & 3 & 0 & 19 & 4   & 19 & 1\\
62 & 0 & 0 & 3 & 0  & 91  & 4  & 4\\
64 & 4 & 0 & 1 & 12 & 89  & 14 & 1\\
66 & 1 & 3 & 0 & 10 & 81  & 24 & 4\\
68 & 3 & 0 & 0 & 4  & 24  & 9  & 1\\
70 & 1 & 3 & 0 & 16 & 96  & 19 & 4\\
72 & 2 & 4 & 4 & 20 & 59  & 4  & 1\\
74 & 3 & 4 & 3 & 13 & 11  & 14 & 4\\
76 & 2 & 0 & 4 & 20 & 69  & 24 & 1\\
78 & 0 & 3 & 3 & 11 & 76  & 9  & 4\\
80 & 3 & 3 & 0 & 24 & 54  & 19 & 1\\
82 & 1 & 0 & 3 & 10 & 41  & 4  & 4\\
84 & 4 & 0 & 1 & 12 & 14  & 14 & 1\\
86 & 0 & 3 & 0 & 0  & 31  & 24 & 4\\
88 & 0 & 0 & 0 & 24 & 74  & 9  & 1\\
90 & 3 & 3 & 0 & 11 & 46  & 19 & 4\\
92 & 1 & 4 & 4 & 10 & 109 & 4  & 1\\
94 & 3 & 4 & 3 & 13 & 86  & 14 & 4\\
96 & 3 & 0 & 4 & 5  & 119 & 24 & 1\\
98 & 3 & 3 & 3 & 16 & 26  & 9  & 4\\
\bottomrule
\end{tabular}
\end{table}

Put \[ H_r(X):=\alpha_r+\beta_rX+\gamma_rX^2 \] and \begin{equation} N_r(X):=(X^5-X)^2-5H_r(X)(X^5-X). \label{eq:p5-mod125-numerator} \end{equation}

\subsection{Certified Hecke operator identities}

We work at precision \(625\). The Dickson multipliers have degrees \[ a_4=2500,\qquad b_4=750,\qquad a_4+b_4=3250. \] The induction base consists of the even degrees \[ 750\leq d\leq3248, \] with all four orientations \(q\in\{0,1,2,3\}\). On \[ \mathbb M_d^{\epsilon_q}(\mathbf Z/625\mathbf Z), \qquad \epsilon_q=(-1)^q, \] the formal operators \(t_2\) and \(t_{19}\) act as \[ t_2=\chi_4(2)^qT_2=2^{125q}T_2, \qquad t_{19}=\chi_4(19)^qT_{19}=19^{125q}T_{19}. \] In orientation \(q\), we use the constants with subscript \(d+50q\), where the subscript is understood modulo \(100\). The remaining even degrees \(0\leq d<750\) are verified directly in the untwisted plus orientation \(q=0\).

For every degree and orientation in the above verification ranges, the computation \cite[\href{https://github.com/nrustom/hecke-congruences/blob/main/playground_mod_125.ipynb}{\nolinkurl{playground_mod_125.ipynb}}] {RustomHeckeCongruences} verifies \begin{equation} \begin{aligned} &\bigl(t_2^2-h_{d+50q,0} -h_{d+50q,1}t_{19}-h_{d+50q,2}t_{19}^2\bigr) \mathbb M_d^{\epsilon_q}(\mathbf Z/625\mathbf Z)\\ &\qquad\subseteq 125\mathbb M_d^{\epsilon_q}(\mathbf Z/625\mathbf Z). \end{aligned} \label{eq:p5-mod125-joint-source} \end{equation}

On the same module, put \[ \mathcal Z_{19,q}:= ''\frac{t_{19}}5'' = \mathfrak D_{t_{19},\,5} \bigl(\mathbb M_d^{\epsilon_q}(\mathbf Z/625\mathbf Z)\bigr) \] and \[ \mathcal B_{r,q}:= ''\frac{N_{d+50q}(\mathcal Z_{19,q})}{25}'' = \mathfrak D_{1,\,25}\circ N_{d+50q}(\mathcal Z_{19,q}). \] Here \(N_{d+50q}(\mathcal Z_{19,q})\) is evaluated using the expanded polynomial convention. All these relations are formed on \(\mathbb M_d^{\epsilon_q}(\mathbf Z/625\mathbf Z)\); in particular, every intermediate element introduced in a composition or division belongs to this module. The computation verifies \begin{equation} \operatorname{dom}\mathcal Z_{19,q} = \operatorname{dom}\mathcal B_{r,q} = \mathbb M_d^{\epsilon_q}(\mathbf Z/625\mathbf Z). \label{eq:p5-mod125-full-domain}
\end{equation}

Write \(\bar v_r\) for the reduction of \(v_r\) modulo \(5\), and put \[ u_r:=\frac{v_r^5-v_r}{5}\bmod5,\qquad \Lambda_r:=\{\bar v_r,\bar v_r+2,\bar v_r-2\} \subseteq\mathbf F_5. \] The constants in Table~\ref{tab:p5-mod125-constants} satisfy \[ H_r(\bar v_r)=2u_r \qquad\text{in }\mathbf F_5. \] For \(\ell\in\mathbf F_5\), let \(P_{r,\ell}(Y)\in\mathbf Z[Y]\) be a lift of \[ \overline{P}_{r,\ell}(Y):= \begin{cases} 1,&\ell\notin\Lambda_r,\\ (Y+u_r^2)^4,&\ell=\bar v_r,\\ \displaystyle\left( \prod_{t\in\mathbf F_5} \bigl(Y-t^2+H_r(\ell)t\bigr) \right)^4, &\ell\in\Lambda_r\setminus\{\bar v_r\}. \end{cases}\]

For \(\ell\in\{0,1,2,3,4\}\), put \[ E_\ell(X):=\bigl(1-(X-\ell)^4\bigr)^5, \qquad \mathcal E_{\ell,q}:=E_\ell(\mathcal Z_{19,q}), \] and \[ \mathcal J_{r,\ell,q}:= \mathcal E_{\ell,q}\circ \mathcal B_{r,q}\circ \mathcal E_{\ell,q}. \] Polynomial evaluation uses the expanded polynomial convention. For every degree and orientation in the above verification ranges, the computation verifies \begin{equation} P_{d+50q,\ell}(\mathcal J_{r,\ell,q}) \circ\mathcal E_{\ell,q} \equiv0\pmod5 \qquad(\ell=0,1,2,3,4) \label{eq:p5-mod125-selector-source} \end{equation} on \(\mathbb M_d^{\epsilon_q}(\mathbf Z/625\mathbf Z)\). 

These conditions are closed under the two transitions of Remark~\ref{rem:periodic-weight-propagation}. Indeed, the \(A_4\)-transition sends \[ (d,q)\longmapsto(d-2500,q), \] while the \(B_4\)-transition sends \[ (d,q)\longmapsto(d-750,q+1) \] and reverses the sign. Both transitions preserve the subscripts of the constants, since \[ (d-2500)+50q\equiv d+50q\pmod{100}, \qquad (d-750)+50(q+1)\equiv d+50q\pmod{100}. \] The change in orientation accounts for the twist of the Hecke operators. Since \(\chi_4^4=\mathrm{id}\), the four orientations suffice.

Thus the same strong-induction argument as in Proposition~\ref{prop:presented-relation-propagation} propagates \eqref{eq:p5-mod125-joint-source}, \eqref{eq:p5-mod125-full-domain}, and \eqref{eq:p5-mod125-selector-source} to every even degree \(d\geq750\), for all four orientations. Together with the direct verification in even degrees \(d<750\), this establishes these conditions in every even degree in the untwisted plus orientation.

We now take \(q=0\) and pass through the natural surjection \[ \mathbb M_d^+(\mathbf Z/625\mathbf Z) \twoheadrightarrow P_d^+/625P_d^+. \]
Full domain of \(\mathcal Z_{19,0}\) gives \[ T_{19}P_d^+\subseteq5P_d^+, \qquad Z_{19}:=\frac{T_{19}}5 \in\End_{\mathbf Z_5}(P_d^+). \] Cancelling \(5\) in its defining equation shows that every value of \(\mathcal Z_{19,0}\), after passage to the target, agrees with \(Z_{19}\) modulo \(125P_d^+\). The same is therefore true for \(N_r(\mathcal Z_{19,0})\) and \(N_r(Z_{19})\). 

Full domain of \(\mathcal B_{r,0}\) now gives \[ N_r(Z_{19})P_d^+\subseteq25P_d^+, \qquad B_r:=\frac{N_r(Z_{19})}{25} \in\End_{\mathbf Z_5}(P_d^+). \] Cancelling \(25\) once more shows that every value of \(\mathcal B_{r,0}\), after passage to the target, agrees with \(B_r\) modulo \(5P_d^+\). 

Since \[ (Z_{19}^5-Z_{19})^2P_d^+\subseteq5P_d^+, \] the reduction \(\bar Z_{19}\) satisfies \[ (\bar Z_{19}^5-\bar Z_{19})^2=0. \] The factorization \[ (X^5-X)^2=\prod_{j\in\mathbf F_5}(X-j)^2 \] therefore gives \[ P_d^+/5P_d^+ =\bigoplus_{j\in\mathbf F_5}\ker(\bar Z_{19}-j)^2. \] On the \(j\)-th summand, write \(U = \bar Z_{19} - j\), so \(U^2=0\). In characteristic \(5\), \[ E_\ell(j+U) =\bigl(1-(j-\ell+U)^4\bigr)^5 =1-(j-\ell)^4 = \begin{cases} 1,&j=\ell,\\ 0,&j\ne\ell. \end{cases} \] Thus the reductions of \(E_\ell(Z_{19})\) are the corresponding orthogonal idempotents. They commute with \(B_r\), since both operators are polynomials in \(Z_{19}\) over \(\mathbf Q_5\). 

Fix \(\ell\), and write \(E\) and \(B\) for the reductions of \(E_\ell(Z_{19})\) and \(B_r\). By the preceding cancellation argument, \eqref{eq:p5-mod125-selector-source} gives \[ P_{r,\ell}(EBE)E=0. \] Since \(B\) commutes with \(E\), it preserves \(\operatorname{im}E\) and \(\ker E\). On \(\operatorname{im}E\), the operator \(E\) is the identity, so this equation gives \(P_{r,\ell}(B)=0\). On \(\ker E\), the operator \(E P_{r,\ell}(B)\) is zero. Hence \(E P_{r,\ell}(B)=0\) on the whole quotient. Equivalently, \[ E_\ell(Z_{19})P_{r,\ell}(B_r)P_d^+ \subseteq5P_d^+ \qquad(\ell=0,1,2,3,4). \] By saturation, these conclusions restrict to \(L_d^+\). Together with \eqref{eq:p5-mod125-joint-source}, we obtain \begin{equation} \begin{aligned} &Z_{19}L_d^+\subseteq L_d^+, \qquad B_rL_d^+\subseteq L_d^+,\\ &\bigl(T_2^2-h_{r,0}-h_{r,1}T_{19} -h_{r,2}T_{19}^2\bigr)L_d^+ \subseteq125L_d^+,\\ &E_\ell(Z_{19})P_{r,\ell}(B_r)L_d^+ \subseteq5L_d^+ \qquad(\ell=0,1,2,3,4). \end{aligned} \label{eq:p5-mod125-target-relations} \end{equation}

\subsection{Congruences for \texorpdfstring{\(a_2\)}{a2} and \texorpdfstring{\(a_{19}\)}{a19}}

By the Eichler--Shimura comparison recalled in Section~\ref{sec:target-lattice}, choose a primitive simultaneous eigenvector \[ v_f\in L_d^+\otimes_{\mathbf Z_5}\mathcal O_f \] for the prime-to-\(5\) Hecke eigensystem of \(f\). Since \(Z_{19}\) and \(B_r\) preserve \(L_d^+\), their eigenvalues \[ z_{19}:=\frac{a_{19}(f)}5, \qquad b_{19}:=\frac{N_r(z_{19})}{25} \] belong to \(\mathcal O_f\). 

Put \(c=z_{19}^5-z_{19}\). We have  \[ c\bigl(c-5H_r(z_{19})\bigr)=25b_{19}. \] Consequently, \(c\in5\mathcal O_f\): otherwise \(v_5(c)<1\), so \[ v_5\bigl(c-5H_r(z_{19})\bigr)=v_5(c), \] giving \[ v_5\bigl(c(c-5H_r(z_{19}))\bigr)=2v_5(c)<2, \] a contradiction. Thus \[ u_{19}:=\frac{z_{19}^5-z_{19}}5\in\mathcal O_f, \qquad b_{19}=u_{19}^2-H_r(z_{19})u_{19}. \] Reducing \(z_{19}^5-z_{19}\in5\mathcal O_f\) modulo the
maximal ideal shows that \(\bar z_{19}\) equals a unique \(\ell\in\{0,1,2,3,4\}\). Since \[ X^5-X\equiv\prod_{j=0}^4(X-j)\pmod5, \] we have \[ \prod_{j=0}^4(z_{19}-j)\in5\mathcal O_f. \] Every factor with \(j\ne\ell\) is a unit, so \(z_{19}-\ell\in5\mathcal O_f\). Thus \[ z_{19}\equiv_{\mathrm{lit}}\ell\pmod5. \]

Evaluating the selector identity in \eqref{eq:p5-mod125-target-relations} on the primitive eigenvector \(v_f\) gives \[ E_\ell(z_{19})P_{r,\ell}(b_{19})\in5\mathcal O_f. \] Since \(z_{19}\equiv_{\mathrm{lit}}\ell\pmod5\), we have \[ E_\ell(z_{19}) =\bigl(1-(z_{19}-\ell)^4\bigr)^5 \equiv1\pmod{5\mathcal O_f}. \] Thus \(E_\ell(z_{19})\) is a unit, and \(P_{r,\ell}(b_{19})\in5\mathcal O_f\). Reducing modulo the maximal ideal of \(\mathcal O_f\) gives \[ \overline P_{r,\ell}(\bar b_{19})=0. \] Reducing \(b_{19}=u_{19}^2-H_r(z_{19})u_{19}\) also gives \[ \bar b_{19} =\bar u_{19}^{\,2}-H_r(\ell)\bar u_{19}. \] If \(\ell\notin\Lambda_r\), then \(\overline P_{r,\ell}=1\), a contradiction. Hence \(\ell\in\Lambda_r\). 

If \(\ell=\bar v_r\), then \(H_r(\ell)=2u_r\) and \(\overline P_{r,\ell}(Y)=(Y+u_r^2)^4\). Thus \((\bar b_{19}+u_r^2)^4=0\).Since the residue field has no nonzero nilpotents,  \[ 0=\bar b_{19}+u_r^2 =\bar u_{19}^{\,2}-2u_r\bar u_{19}+u_r^2 =(\bar u_{19}-u_r)^2, \] so \(\bar u_{19}=u_r\).

If \(\ell\in\Lambda_r\setminus\{\bar v_r\}\), use the identity \[ \prod_{t\in\mathbf F_5} \bigl(U^2-aU-t^2+at\bigr) = \prod_{t\in\mathbf F_5}(U-t)(U+t-a) =(U^5-U)^2 \qquad(a\in\mathbf F_5). \] Taking \(a=H_r(\ell)\), the equality \(\overline P_{r,\ell}(\bar b_{19})=0\) therefore gives \[ (\bar u_{19}^{\,5}-\bar u_{19})^8=0. \] Hence \(\bar u_{19}\in\mathbf F_5\).

Write \[ z_{19}=\ell+5w_{19}, \qquad w_{19}\in\mathcal O_f. \] Then \[ u_{19} =\frac{z_{19}^5-z_{19}}5 \equiv\frac{\ell^5-\ell}{5}-w_{19} \pmod{5\mathcal O_f}. \] Thus \(\bar u_{19}\) determines \(\bar w_{19}\), and hence \(z_{19}\) modulo \(25\) in the valuative sense. When \(\ell=\bar v_r\), the equality \(\bar u_{19}=u_r\) selects \(z_{19}\equiv_{\mathrm{val}}v_r\pmod{25}\). For either of the other two values of \(\ell\), there are five possibilities, corresponding to \(\bar w_{19}\in\mathbf F_5\).

Consequently, \begin{equation} a_{19}(f)\equiv_{\mathrm{val}}5t\pmod{125} \label{eq:p5-mod125-T19} \end{equation} for a unique \(t\in\mathbf Z/25\mathbf Z\) satisfying \[ t=v_r \qquad\text{or}\qquad \bar t\in\{\bar v_r+2,\bar v_r-2\}. \] There are eleven such values of \(t\) for each even residue \(r\pmod{100}\).

Evaluating the ordinary joint identity in \eqref{eq:p5-mod125-target-relations} on \(v_f\) gives \[ a_2(f)^2-h_{r,0}-h_{r,1}a_{19}(f) -h_{r,2}a_{19}(f)^2\in125\mathcal O_f. \] Together with \eqref{eq:p5-mod125-T19}, this yields \[ a_2(f)^2\equiv_{\mathrm{val}} h_{r,0}+5h_{r,1}t+25h_{r,2}t^2 \pmod{125}. \] The constants satisfy \[ h_{r,0}\equiv \begin{cases} 4,&r\equiv0\pmod4,\\ 1,&r\equiv2\pmod4 \end{cases} \pmod5. \] Thus, for each permitted \(t\), the congruence \[ a^2\equiv h_{r,0}+5h_{r,1}t+25h_{r,2}t^2 \pmod{125} \] has exactly two roots \(a\in\mathbf Z/125\mathbf Z\), both units. Their reductions modulo \(5\) are distinct, and one agrees with the reduction of \(a_2(f)\). For that root, the factor \(a_2(f)+a\) is a unit. Factoring \(a_2(f)^2-a^2\) therefore gives \begin{equation} a_2(f)\equiv_{\mathrm{val}}a\pmod{125}. \label{eq:p5-mod125-T2} \end{equation} Consequently, there are \(22\) permitted pairs \((a,5t)\) for each even residue \(r\pmod{100}\).

\subsection{Final classification}  By the preceding subsection, every normalized cuspidal level-one eigenform \(f\) determines a unique signature \[ \sigma=(r,a,5t), \qquad r\equiv k-2\pmod{100},\qquad 0\leq r<100, \] where \((a,5t)\) is one of the permitted pairs, such that \[ a_2(f)\equiv_{\mathrm{val}}a\pmod{125}, \qquad a_{19}(f)\equiv_{\mathrm{val}}5t\pmod{125}. \] There are \(50\cdot22=1100\) such signatures altogether.

The finite computation \cite[\url{https://github.com/nrustom/hecke-congruences/tree/main/strong_signatures/p5_m3}] {RustomHeckeCongruences} verifies that every signature is realized by a normalized cuspidal level-one eigenform \(g_\sigma\) of weight \(k_\sigma\leq598\), with \[ k_\sigma-2\equiv r\pmod{100}. \] If \(f\) has signature \(\sigma\), then \[ a_2(f)\equiv_{\mathrm{val}}a_2(g_\sigma)\pmod{125}, \qquad a_{19}(f)\equiv_{\mathrm{val}}a_{19}(g_\sigma)\pmod{125}. \] Moreover, \(k_\sigma\equiv k\pmod{100}\), so \[ 6^{k_\sigma}\equiv6^k\pmod{125}. \] Thus \(f\) and \(g_\sigma\) also have the same value of the weight operator \([6]\) modulo \(125\). By \eqref{eq:p5-big-generation}, \[ \mathbb T_{\mathrm{dc}}^{(5)} =\overline{\mathbf Z_5[T_2,T_{19},[6]]}. \] Their prime-to-\(5\) Hecke eigensystems are therefore congruent modulo \(125\) in the valuative sense: \[ a_n(f)\equiv_{\mathrm{val}}a_n(g_\sigma)\pmod{125} \qquad((n,5)=1). \]

Conversely, suppose that eigenforms \(f\) and \(g\), of weights \(k\) and \(K\), have congruent prime-to-\(5\) Hecke eigensystems modulo \(125\). The identities \[ a_2(f)^2-a_4(f)=2^{k-1}, \qquad a_2(g)^2-a_4(g)=2^{K-1} \] give \[ 2^{k-1}\equiv_{\mathrm{val}}2^{K-1}\pmod{125}. \] Since these are rational integers, this congruence is literal. As \(2\) has order \(100\) modulo \(125\), we obtain \(k\equiv K\pmod{100}\). Agreement at \(T_2\) and \(T_{19}\) then implies that their permitted pairs \((a,5t)\) agree. Thus distinct signatures give distinct eigensystems.

Reducing the signatures modulo \(25\), direct calculation identifies their values at \(T_2,T_{19}\) with those of the Eisenstein twists corresponding to the pairs in Theorem~\ref{thm:intro-p5}. For each signature, the pair is unique. Choose \(i\geq3\) and even \(\kappa\geq4\) with \[ \kappa+2i\equiv k\pmod{100}, \] and with \(\{i,i+\kappa-1\}\) reducing modulo \(20\) to the selected pair. By Lemma~\ref{lem:weak-cyclotomic} and \eqref{eq:p5-big-generation}, agreement at these generators and at \([6]\) gives \[ \theta^{100}f\equiv_{\mathrm{lit}} \theta^iG_\kappa\pmod{25}. \] By Lemma~\ref{lem:modular-correction}, \[ \delta_f:= \overline{\frac{\theta^{100}f-\theta^iG_\kappa}{25}} \] is a modular form satisfying \[ \theta^{100}f\equiv_{\mathrm{val}} \theta^iG_\kappa+25\delta_f\pmod{125}. \] For fixed \(i,\kappa\), equation~\eqref{eq:p5-big-generation} shows that its \(q\)-expansion is determined by the signature and is defined over \(\mathbf F_5\), since all signature values are rational integers.

Consequently, there are exactly \(1100\) prime-to-\(5\) Hecke eigensystems modulo \(125\), in the valuative sense, and every such eigensystem is realized in weight at most \(598\).

\section{The prime \texorpdfstring{\(7\)}{7}: classification modulo
\texorpdfstring{\(49\)}{49}}
\label{sec:p7}

\subsection{Certified Hecke operator identities}

Define \begin{equation}\begin{aligned}F_0(X)&:=X^6+7X^5+45X^4+7X^3+4X^2+7X,\\F_2(X)&:=X^6+42X^5+6X^4+9X^2+14X,\\F_4(X)&:=X^6+7X^5+47X^4+14X^3+X^2+28X.\end{aligned}\label{eq:p7-tangent-polynomials}\end{equation} Modulo \(7\), these factor as \begin{equation} \begin{aligned} F_0(X)&\equiv\bigl(X(X-3)(X-4)\bigr)^2,\\F_2(X)&\equiv\bigl(X(X-2)(X-5)\bigr)^2,\\F_4(X)&\equiv\bigl(X(X-1)(X-6)\bigr)^2 \end{aligned}\pmod7.\label{eq:p7-tangent-factorizations}\end{equation}

Also put \[\{\alpha_j,\beta_j\}=\begin{cases}\{3,4\},&j=0,\\ \{2,5\},&j=2,\\ \{1,6\},&j=4, \end{cases}\qquad \mathcal C_j:=\{0,\alpha_j,\beta_j\}. \]

Define \begin{equation} Q_j(X):=\frac{F_j(X)}{49}, \qquad G_j(X,Y):= \frac{Y-2-X(X-\alpha_j)(X-\beta_j)}{7}, \qquad j\in\{0,2,4\}.\label{eq:Q-G-definition} \end{equation} Finally, \[H(X) := (X^7 - X)^3.\]

We first treat \(Q_{r\bmod6}(T_3)\). The computation is performed at precision \(343=7^3\). The Dickson multipliers have degrees \[ a_3=2058,\qquad b_3=392, \qquad\mbox{with}\qquad a_3+b_3=2450,\] so the exact induction base consists of the even degrees \[ 392\leq d<2450. \] For every such \(d\) and every \(0\leq q<6\), put \[ \mathcal Q_q:= ''\frac{F_{(d+2q)\bmod 6}(t_3)}{49}'' = \mathfrak D_{F_{(d+2q)\bmod6}(t_3),\,49} \bigl(\mathbb M_d^{\epsilon_q}(\mathbf Z/343\mathbf Z)\bigr), \qquad \epsilon_q=(-1)^q, \] where \(t_3\) acts as \(\chi_3(3)^qT_3\). The computation \cite[\href{https://github.com/nrustom/hecke-congruences/blob/main/playground_mod_49.ipynb}{\nolinkurl{playground_mod_49.ipynb}}]{RustomHeckeCongruences} verifies \[ \operatorname{dom}\mathcal Q_q =\mathbb M_d^{\epsilon_q}(\mathbf Z/343\mathbf Z), \qquad H(\mathcal Q_q)\equiv0\pmod7. \] These six conditions are closed under both transitions of Remark~\ref{rem:periodic-weight-propagation}: the \(A_3\)-transition preserves \(d+2q\pmod6\), while the \(B_3\)-transition sends \[ (d,q)\longmapsto(d-392,q+1) \] and again preserves \(d+2q\pmod6\).

Now fix an even residue \(r\bmod42\). The preceding transition compatibility and the same strong-induction argument as in Proposition~\ref{prop:presented-relation-propagation} propagate these conditions to every degree \[ d\geq392, \qquad d\equiv r\pmod{42}. \] Applying Proposition~\ref{prop:division-retained-precision} in orientation \(q=0\), with \(m=3\), \(\alpha_1=2\), \(b=1\), and \(S=1\), gives \[ Q_{r\bmod6}(T_3)\in\End_{\mathbf Z_7}(P_d^+), \qquad Q_{r\bmod6}(T_3)L_d^+\subseteq L_d^+. \] Moreover, \begin{equation} H\bigl(Q_{r\bmod6}(T_3)\bigr)P_d^+ \subseteq7P_d^+, \qquad H\bigl(Q_{r\bmod6}(T_3)\bigr)L_d^+ \subseteq7L_d^+. \label{eq:p7-Q-relations} \end{equation} The remaining degrees \[ 0\leq d<392, \qquad d\equiv r\pmod{42}, \] are verified directly. Consequently, these conclusions hold in every even degree \[ d\equiv r\pmod{42}. \]

We next treat \(G_{r\bmod6}(T_3,T_{29})\).  This computation is performed at precision \(49\).  The Dickson multipliers have degrees \[ a_2=294,\qquad b_2=56, \qquad\mbox{with}\qquad a_2+b_2=350, \] so the exact induction base consists of the even degrees \[ 56\leq d<350. \] For every such \(d\) and every \(0\leq q<6\), put \[ \mathcal G_q:= ''\frac{7G_{(d+2q)\bmod 6}(t_3, t_{29})}{7}'' = \mathfrak D_{7G_{(d+2q)\bmod6}(t_3,t_{29}),\,7} \bigl(\mathbb M_d^{\epsilon_q}(\mathbf Z/49\mathbf Z)\bigr), \qquad \epsilon_q=(-1)^q, \] where \(t_n\) acts as \(\chi_2(n)^qT_n\). Here \(7G_j\) denotes the integral polynomial \[ 7G_j(X,Y)=Y-2-X(X-\alpha_j)(X-\beta_j). \] The computation verifies \[ \operatorname{dom}\mathcal G_q =\mathbb M_d^{\epsilon_q}(\mathbf Z/49\mathbf Z), \qquad H(\mathcal G_q)\equiv0\pmod7. \] These six conditions are closed under both transitions of Remark~\ref{rem:periodic-weight-propagation}. Hence the same strong-induction argument as in Proposition~\ref{prop:presented-relation-propagation} propagates these conditions to every degree \[ d\geq56, \qquad d\equiv r\pmod{42}. \] Applying Proposition~\ref{prop:division-retained-precision}
in orientation \(q=0\), with \(m=2\), \(\alpha_1=1\), \(b=1\), and \(S=1\), gives \[ G_{r\bmod6}(T_3,T_{29}) \in\End_{\mathbf Z_7}(P_d^+), \qquad G_{r\bmod6}(T_3,T_{29})L_d^+\subseteq L_d^+. \] Moreover, \begin{equation} H\bigl(G_{r\bmod6}(T_3,T_{29})\bigr)P_d^+ \subseteq7P_d^+, \qquad H\bigl(G_{r\bmod6}(T_3,T_{29})\bigr)L_d^+ \subseteq7L_d^+. \label{eq:p7-G-relations} \end{equation} The remaining degrees \[ 0\leq d<56, \qquad d\equiv r\pmod{42}, \] are verified directly. Consequently, these conclusions hold in every even degree \[ d\equiv r\pmod{42}. \]

\subsection{Congruences for \texorpdfstring{\(a_3\)}{a3} and
\texorpdfstring{\(a_{29}\)}{a29}}

By the Eichler--Shimura comparison recalled in Section~\ref{sec:target-lattice}, choose a primitive simultaneous eigenvector \[v_f\in L_d^+\otimes_{\mathbf Z_7}\mathcal O_f\] for the prime-to-\(7\) Hecke eigensystem of \(f\).

We first determine \(a_3(f)\) modulo \(7\).  Since \[Q_{r\bmod6}(T_3)=\frac{F_{r\bmod6}(T_3)}{49}\] is integral on \(L_d^+\), evaluation on \(v_f\) gives \[ F_{r\bmod6}(a_3(f))\in49\mathcal O_f. \] Reducing modulo the maximal ideal and using \eqref{eq:p7-tangent-factorizations}, there is a unique \[ c\in\mathcal C_{r\bmod6} \] such that \[ a_3(f)\equiv c\pmod{\mathfrak m_{\mathcal O_f}}. \]

Write \(a_3(f)=c+y\). Direct calculation gives \[F_{r\bmod6}(c)\in49\mathbf Z_7,\qquad F'_{r\bmod6}(c)\in7\mathbf Z_7, \qquad \frac{F''_{r\bmod6}(c)}2\in\mathbf Z_7^\times. \] If \(0<v_7(y)<1\), then in the Taylor expansion of \(F_{r\bmod6}(c+y)\) the quadratic term has strictly smaller valuation than every other term, contradicting \[ F_{r\bmod6}(a_3(f))\in49\mathcal O_f. \] It follows that \begin{equation} a_3(f)-c\in7\mathcal O_f. \label{eq:p7-T3-residual} \end{equation} Since \(G_{r\bmod6}(T_3,T_{29})\) is also integral on \(L_d^+\), evaluation on \(v_f\) gives \[ a_{29}(f)-2 - a_3(f) \bigl(a_3(f)-\alpha_{r\bmod6}\bigr) \bigl(a_3(f)-\beta_{r\bmod6}\bigr) \in7\mathcal O_f. \] By \eqref{eq:p7-T3-residual}, one of the three factors in the cubic product lies in \(7\mathcal O_f\).  Hence \begin{equation} a_{29}(f)-2\in7\mathcal O_f. \label{eq:p7-T29-residual} \end{equation} 

We may therefore write \begin{equation} a_3(f)=c+7x_f, \qquad a_{29}(f)=2+7y_f, \qquad x_f, y_f\in\mathcal O_f. \label{eq:p7-raw-digits} \end{equation}

Put \[q_f:=Q_{r\bmod6}(a_3(f)), \qquad g_f:=G_{r\bmod6}(a_3(f),a_{29}(f)). \] By the preceding subsection, \[ q_f,g_f\in\mathcal O_f \] and \[ H(q_f),H(g_f)\in7\mathcal O_f. \] Reducing modulo the maximal ideal gives \[ \bar q_f^7=\bar q_f, \qquad \bar g_f^7=\bar g_f, \] so \[ \bar q_f,\bar g_f\in\mathbf F_7. \] 

Plugging \[ a_3 = c + 7 x_f, \qquad a_{29} = 2 + 7y_f, \qquad x_f,y_f \in \cO_f, \] into \eqref{eq:Q-G-definition} and expanding, we find that
\begin{equation}
\begin{array}{c|c|c|c}
r\bmod6 & c & \bar q_f & \bar g_f \\ \hline
0 & 0
  & 4\bar x_f^{\,2}+\bar x_f
  & \bar y_f-5\bar x_f \\[1mm]
0 & 3
  & 2\bar x_f^{\,2}+2\bar x_f+3
  & \bar y_f-4\bar x_f \\[1mm]
0 & 4
  & 2\bar x_f^{\,2}
  & \bar y_f-4\bar x_f \\ \hline
2 & 0
  & 2\bar x_f^{\,2}+2\bar x_f
  & \bar y_f-3\bar x_f \\[1mm]
2 & 2
  & \bar x_f^{\,2}+3\bar x_f+4
  & \bar y_f-\bar x_f \\[1mm]
2 & 5
  & \bar x_f^{\,2}+4\bar x_f
  & \bar y_f-\bar x_f \\ \hline
4 & 0
  & \bar x_f^{\,2}+4\bar x_f
  & \bar y_f-6\bar x_f \\[1mm]
4 & 1
  & 4\bar x_f^{\,2}+\bar x_f+2
  & \bar y_f-2\bar x_f \\[1mm]
4 & 6
  & 4\bar x_f^{\,2}+2\bar x_f+5
  & \bar y_f-2\bar x_f.
\end{array}\label{eq:qf-gf-xf-yf}
\end{equation}

Since $\bar q_f, \bar g_f\in \F_7$, we see that $\bar x_f, \bar y_f \in \F_{49}$. Thus these relations are already enough to prove the finiteness of the number of away-from-$7$ Hecke eigensystems in level $1$. 

In order to obtain classification, we have to correlate $\bar q_f$ and $\bar g_f$. For \(c\in\mathcal C_j\), let \(E_{j,c}(X)\in\mathbf Z[X]\) be the coefficientwise lift, with coefficients in \(\{0,\ldots,6\}\), of the Lagrange polynomial \[ \prod_{c'\in\mathcal C_j\setminus\{c\}} \frac{X-c'}{c-c'}\in\mathbf F_7[X]. \] The denominators are nonzero in \(\mathbf F_7\), since
the elements of \(\mathcal C_j\) are distinct modulo \(7\). Thus \[ E_{j,c}(a_3(f)) \equiv \begin{cases} 1\pmod{\mathfrak m_{\mathcal O_f}}, & a_3 (f)\equiv c\pmod{\mathfrak m_{\mathcal O_f}},\\ 0\pmod{\mathfrak m_{\mathcal O_f}}, & a_3(f) \equiv c'\pmod{\mathfrak m_{\mathcal O_f}} \text{ for some }c'\in\mathcal C_j\setminus\{c\}. \end{cases} \] 

The classification computation is performed at precision \(343\). For every even degree \[ 392\leq d<2450 \] and every orientation \(0\leq q<6\), put \[ r\equiv d+14q\pmod{42}. \] On \[ \mathbb M_d^{\epsilon_q}(\mathbf Z/343\mathbf Z), \qquad \epsilon_q=(-1)^q, \] the computation supplies polynomials \[ R_{r,c}^{(1)},R_{r,c}^{(2)}\in\mathbf Z_7[X,Y] \] for every \(c\in\mathcal C_{r\bmod6}\). We use the same notation for their reductions in \(\mathbf F_7[X,Y]\). On this source module, put
\[
\begin{aligned}
\mathcal Q_q &= ''\frac{F_{r\bmod6}(t_3)}{49}'' &:=\mathfrak D_{F_{r\bmod6}(t_3),\,49} \bigl(\mathbb M_d^{\epsilon_q}(\mathbf Z/343\mathbf Z)\bigr),\\
\mathcal G_q &= ''\frac{7G_{r\bmod6}(t_3,t_{29})}{7}'' &:=\mathfrak D_{7G_{r\bmod6}(t_3,t_{29}),\,7} \bigl(\mathbb M_d^{\epsilon_q}(\mathbf Z/343\mathbf Z)\bigr),
\end{aligned}
\]
where \(t_n\) acts as \(\chi_3(n)^qT_n\).
Thus both relations are now taken at precision \(343\).
Let
\[
S_{c,q}:=E_{r\bmod6,c}\bigl(\chi_3(3)^qT_3\bigr)^6.
\]
For \(\nu\in\{1,2\}\), the computation verifies \[ \bigl[(R_{r,c}^{(\nu)})^3\bigr] (\mathcal Q_q,\mathcal G_q)\Gamma_{S_{c,q}} \equiv0\pmod7.\] Here the polynomial is cubed and expanded before evaluating it at the linear relations.

These conditions form a periodic family closed under both transitions of Remark~\ref{rem:periodic-weight-propagation}. Indeed, the \(A_3\)-transition preserves \(r\), while the \(B_3\)-transition sends \[ (d,q)\longmapsto(d-392,q+1) \] and preserves \(r\) modulo \(42\), since \[ -392+14=-378\equiv0\pmod{42}. \] Consequently, Proposition~\ref{prop:presented-relation-propagation}, followed in orientation \(q=0\) by Proposition~\ref{prop:division-retained-precision} with \(m=3\), division exponents \(2,1\), \(b=1\), and \(S=E_{r\bmod6,c}(T_3)^6\), gives, for every even degree \(d\geq392\) with \(d\equiv r\pmod{42}\), \[ E_{r\bmod6,c}(T_3)^6 \left[ R_{r,c}^{(\nu)}\!\left( Q_{r\bmod6}(T_3), G_{r\bmod6}(T_3,T_{29}) \right) \right]^3L_d^+ \subseteq7L_d^+. \] The remaining even degrees \(0\leq d<392\) are verified directly.

We illustrate the argument on the zero branch \(c = 0\) in coefficient-degree residue \(r=0\). Since \(E_{0,0}(a_3(f))\) is a unit, it may be cancelled after evaluating the certified operator relations on \(v_f\).  After reduction modulo \(\mathfrak m_{\mathcal O_f}\), the resulting joint relations give \[ \bar q_f (\bar q_f-3)(\bar q_f-4)(\bar q_f-5)=0, \qquad \bar g_f=2\bar q_f. \] Now using the expressions for \(\bar q_f\) and \(\bar g_f\) in terms of \(\bar x_f\) and \(\bar y_f\) given in
\eqref{eq:qf-gf-xf-yf}, we obtain \[ (\bar x_f,\bar y_f) \in \{ (0,0),(1,1),(2,4),(3,2), (4,2),(5,4),(6,1) \}. \] In particular, all possible values of \(\bar x_f\) and \(\bar y_f\) belong to \(\mathbf F_7\), even when the residue field of \(\mathcal O_f\) is larger.  Therefore, on this branch, \[ \bigl(a_3(f),a_{29}(f)\bigr) \equiv_{\mathrm{val}} \bigl(7x,\,2+7y\bigr)\pmod{49} \] for exactly the seven pairs \((x,y)\) listed above.

The computation verifies the analogous calculation for every even residue \(r\bmod42\) with \(c=0\).  In each case, the common zero set of the two coordinate relations pulls back under \eqref{eq:qf-gf-xf-yf} to exactly seven pairs \[ (\bar x_f,\bar y_f)\in\mathbf F_7^2. \] These pairs are enough for exact classification on the zero branch.

For the nonzero branches (i.e. \(a_3 \equiv_{\mathrm{lit}}c \pmod{7}\), \(c \neq 0\)), the relations \[ E_{r\bmod6,c}(T_3)^6 \left[ R_{r,c}^{(\nu)}\!\left( Q_{r\bmod6}(T_3), G_{r\bmod6}(T_3,T_{29}) \right) \right]^3L_d^+ \subseteq7L_d^+ \] still determine the possible pairs \((\bar q_f,\bar g_f)\), but they do not by themselves determine the correct \((\bar x_f,\bar y_f)\) corresponding to strong eigenforms. They define a set \[\mathcal T_{r,c}:= \left\{ (q,g)\in\mathbf F_7^2: R_{r,c}^{(1)}(q,g)=R_{r,c}^{(2)}(q,g)=0 \right\}. \]

For every nonzero \(c\in\mathcal C_{r\bmod6}\), put formally \[ U_c:=\frac{T_3-c}{7}, \qquad V:=\frac{T_{29}-2}{7}. \] Here division is interpreted through the presented linear relations \(\mathfrak D_{T,S}\); we do not require \(U_c\) or \(V\) separately to preserve \(L_d^+\). 

On the same exact induction base and in the same six orientations as above, the computation verifies compatible congruences of presented linear relations, which propagate by Proposition~\ref{prop:presented-relation-propagation} to every even degree \[ d\geq392, \qquad d\equiv r\pmod{42}, \] and give \[ E_{r\bmod6,c}(T_3)^6 \left[E_{r\bmod6,c}(T_3)^6 U_c -\phi_{r,c}\!\left( Q_{r\bmod6}(T_3), G_{r\bmod6}(T_3,T_{29}) \right) \right]^3 L_d^+ \subseteq 7L_d^+ \] and \[E_{r\bmod6,c}(T_3)^6 \left[ E_{r\bmod6,c}(T_3)^6 V -\psi_{r,c}\!\left( Q_{r\bmod6}(T_3), G_{r\bmod6}(T_3,T_{29}) \right) \right]^3 L_d^+ \subseteq 7L_d^+,\] for some affine polynomials \[ \phi_{r,c},\psi_{r,c}\in\mathbf Z_7[X,Y]. \] We use the same notation for their reductions in \(\mathbf F_7[X,Y]\). The remaining even degrees \(0\leq d<392\) are verified directly.

To justify these lattice inclusions, put \(S_c:=E_{r\bmod6,c}(T_3)^6\). For \(N=T_3-c\) or \(T_{29}-2\), the first division equation for the monomial \((S_cN/7)^3\), starting at \(S_cu\), gives \[ S_c^2N P_d^+\subseteq7P_d^+. \] By the integrality of \(Q_{r\bmod6}(T_3)\) and \eqref{eq:p7-tangent-factorizations}, \(S_c\) acts on \(P_d^+/7P_d^+\) invertibly on the generalized \(c\)-eigenspace and as zero on the other generalized eigenspaces. Hence \(\ker S_c^2=\ker S_c\), so \(S_cN P_d^+\subseteq7P_d^+\). Thus \(S_cU_c\) and \(S_cV\) preserve \(P_d^+\), and also \(L_d^+\) by saturation. Proposition~\ref{prop:division-retained-precision}, applied with \(m=3\), division exponents \(2,1,1\), \(b=1\), and \(S=S_c\), now gives the stated inclusions.

After evaluation on \(v_f\), put \(e_f:=E_{r\bmod6,c}(a_3(f))\). Since \(e_f\equiv1\pmod{\mathfrak m_{\mathcal O_f}}\), we may cancel the outer unit factor \(e_f^6\). Reducing modulo \(\mathfrak m_{\mathcal O_f}\), the inner factors \(e_f^6\) also reduce to \(1\), giving \[ \bar x_f=\phi_{r,c}(\bar q_f,\bar g_f), \qquad \bar y_f=\psi_{r,c}(\bar q_f,\bar g_f). \] Thus every point of \(\mathcal T_{r,c}\) determines exactly one permitted pair \((\bar x_f,\bar y_f)\). 

For example, take \(r=40\) and \(c=1\). The coordinate relations give \[ \mathcal T_{40,1} = \{(0,1),(2,0),(2,3),(6,6)\}. \] The corresponding graph relations reduce to \[ \bar x_f=3\bar g_f+5, \qquad \bar y_f=3. \] Evaluating these formulas on the four points of \(\mathcal T_{40,1}\) gives \[(\bar x_f,\bar y_f) \in \{(0,3),(1,3),(2,3),(5,3)\}. \] Consequently, on this branch, \[ \bigl(a_3(f),a_{29}(f)\bigr) \equiv_{\mathrm{val}} \bigl(1+7x,\,2+7y\bigr) \pmod{49} \] for precisely these four pairs \((x,y)\).

The other nonzero branch is treated in the same way. In every coefficient-degree residue, the centre-zero branch contributes seven signatures and each of the two nonzero branches contributes four, giving fifteen signatures in total. All of these signatures have coordinates in \(\mathbf Z/49\mathbf Z\).

\subsection{Final classification}

\begin{proof}[Proof of Theorem~\ref{thm:intro-p7}]

By the preceding subsections, every normalized cuspidal level-one eigenform \(f\) determines a unique signature \[ \sigma=(r,c,x,y), \qquad r\equiv k-2\pmod{42}, \qquad c\in\mathcal C_{r\bmod6}, \qquad x,y\in\mathbf Z/7\mathbf Z, \] such that \[ a_3(f)\equiv_{\mathrm{val}}c+7x\pmod{49}, \qquad a_{29}(f)\equiv_{\mathrm{val}}2+7y\pmod{49}. \] For each residue \(r\bmod42\), there are exactly \(15\) such signatures, and hence \(315\) signatures altogether.

The finite computation verifies that every one of these \(315\) signatures is realized by a normalized cuspidal level-one eigenform \(g_\sigma\) of weight \(k_\sigma\leq380\), with \[ k_\sigma-2\equiv r\pmod{42}. \]

If \(f\) has signature \(\sigma\), then \[ a_3(f)\equiv_{\mathrm{val}}a_3(g_\sigma)\pmod{49}, \qquad a_{29}(f)\equiv_{\mathrm{val}}a_{29}(g_\sigma)\pmod{49}, \] and \[ k_\sigma\equiv k\pmod{42}. \] Since \(3\) has order \(42\) modulo \(49\), the latter congruence gives \[ 3^{k_\sigma}\equiv3^k\pmod{49}, \] so \(f\) and \(g_\sigma\) also have the same value of the weight operator \([3]\) modulo \(49\).

By Section~\ref{sec:big-hecke-p7}, \[\mathbb T_{\mathrm{dc}}^{(7)} = \overline{\mathbf Z_7[T_3,T_{29},[3]]}. \] Hence the prime-to-\(7\) Hecke eigensystems of \(f\) and \(g_\sigma\) are congruent modulo \(49\) in the valuative sense. Therefore \[ a_n(f)\equiv_{\mathrm{val}}a_n(g_\sigma)\pmod{49} \qquad ((n,7)=1). \] Consequently, only finitely many prime-to-\(7\) Hecke eigensystems occur modulo \(49\), and every such eigensystem is realized in weight at most \(380\).  In fact, the argument gives exactly \(315\) eigensystems.

Reducing the signatures modulo \(7\), direct calculation identifies their values at \(T_3,T_{29}\) with those of the Eisenstein twists corresponding to the pairs in Theorem~\ref{thm:intro-p7}. For each signature, the pair \(\{a,b\}\) is unique. Choose \(i\geq2\) and even \(\kappa\geq4\) with \[ \{i,i+\kappa-1\}=\{a,b\} \quad\text{in }\mathbf Z/6\mathbf Z, \qquad \kappa+2i\equiv k\pmod{42}. \] By Lemma~\ref{lem:weak-cyclotomic} and the completed-Hecke generation above, agreement at \(T_3,T_{29}\) and \([3]\) gives \[ \theta^{42}f\equiv_{\mathrm{lit}} \theta^iG_\kappa\pmod7. \] By Lemma~\ref{lem:modular-correction}, \[ \delta_f:= \overline{\frac{\theta^{42}f-\theta^iG_\kappa}{7}} \] is a modular form satisfying \[ \theta^{42}f\equiv_{\mathrm{val}} \theta^iG_\kappa+7\delta_f\pmod{49}. \] For fixed \(i,\kappa\), completed-Hecke generation shows that its \(q\)-expansion is determined by the signature and is defined over \(\mathbf F_7\), since all signature values are rational integers. The reductions of the \(315\) signatures cover all nine listed pairs, so their strong realizations also realize every listed Eisenstein twist modulo \(7\). This proves Theorem~\ref{thm:intro-p7}.
\end{proof}

\part{Applications}

\section{Strong prime-to-\texorpdfstring{\(p\)}{p} weight bounds}
\label{sec:strong-weight-bounds}

Throughout this section, congruences are understood in the valuative sense.

In \cite[\S1.5]{KRW2016}, the finiteness conjecture for strong eigenforms modulo \(p^m\) is reformulated in terms of a strong weight bound: for fixed \(N,p,m\), one asks for a constant \(B(N,p,m)\) such that every strong eigenform of level \(N\) is congruent modulo \(p^m\) to a strong eigenform of the same level and of weight at most \(B(N,p,m)\).

The classifications proved above concern only the prime-to-\(p\) Hecke eigensystem. We therefore introduce the corresponding prime-to-\(p\) notion. A number \(B\) is called a \emph{strong prime-to-\(p\) weight bound} for \((N,p,m)\) if every normalized characteristic-zero cuspidal eigenform \(f\) of level \(N\) is congruent modulo \(p^m\), at all Hecke operators prime to \(p\), to a normalized characteristic-zero cuspidal eigenform \(g\) of level \(N\) and weight at most \(B\). We denote the least such bound, when it exists, by \[ B^{(p)}(N,p,m). \] 

It is also natural to allow cyclotomic twists. A number \(B\) is called a \emph{strong weight bound up to \(\theta\)-twist} for \((N,p,m)\) if, for every normalized characteristic-zero cuspidal eigenform \(f\) of level \(N\), there exist an integer \(i\geq0\) and a normalized characteristic-zero cuspidal eigenform \(g\) of level \(N\) and weight at most \(B\) such that \[ a_n(f)\equiv n^i a_n(g)\pmod{p^m} \qquad ((n,p)=1). \] We denote the least such bound, when it exists, by \[ B_\theta(N,p,m). \] Equivalently, the prime-to-\(p\) eigensystem of \(f\) is congruent modulo \(p^m\) to the \(\theta^i\)-twist of that of \(g\).

\begin{table}[ht]
\centering
\begin{tabular}{c|r|r}
\hline
modulus
&
strong weight bound \(B\)
&
bound up to \(\theta\)-twist \(B_\theta\)
\\
\hline
\(2\)       & \(12\)  & \(12\) \\
\(2^2\)     & \(12\)  & \(12\) \\
\(2^3\)     & \(12\)  & \(12\) \\
\(2^4\)     & \(18\)  & \(12\) \\
\(2^5\)     & \(22\)  & \(16\) \\
\(2^6\)     & \(30\)  & \(24\) \\
\(2^7\)     & \(46\)  & \(32\) \\
\(2^8\)     & \(90\)  & \(62\) \\
\(2^9\)     & \(158^{\dagger}\) & \(128^{\dagger}\) \\
\hline
\(3\)       & \(12\)  & \(12\) \\
\(3^2\)     & \(32\)  & \(16\) \\
\(3^3\)     & \(70\)  & \(24\) \\
\(3^4\)     & \(214\) & \(58\) \\
\(3^5\)     & \(646^{\dagger}\) & \(166^{\dagger}\) \\
\hline
\(5\)       & \(30\)  & \(12\) \\
\(5^2\)     & \(142\) & \(24\) \\
\(5^3\)     & \(598\) & \(106\) \\
\hline
\(7\)       & \(56\)  & \(16\) \\
\(7^2\)     & \(380\) & \(78\) \\
\hline
\end{tabular}
\caption{Strong level-one weight bounds for prime-power away-from-\(p\) congruences. The bound \(B_\theta\) means that every strong eigensystem is congruent away from \(p\) to \(\theta^i g\) for some \(i\geq0\) and some strong level-one eigenform \(g\) of weight at most \(B_\theta\). The entries marked \(\dagger\) are conjectural.}\label{tab:strong-weight-bounds}
\end{table}

The proved and conjectural values in Table~\ref{tab:strong-weight-bounds} display some suggestive patterns. For \(p=3,5,7\), the proved values at \(m=1,2\) satisfy \[ B^{(p)}(1,p,m)=(p+2)\varphi(p^m)+2. \] At the higher exponents, a different uniform pattern appears: the values for \[ (p,m)=(3,3),(3,4),(3,5),(5,3) \] satisfy \[ B^{(p)}(1,p,m)=(p+1)\varphi(p^m)-2. \] The data for \(p=2\) appear less regular, although the bounds remain of the same order of magnitude as \(\varphi(2^m)\). The bounds up to \(\theta\)-twist likewise do not yet suggest a uniform formula valid for all primes and exponents. It seems premature to conjecture an exact general formula from these data, though it seems more reasonble to ask the following.

\paragraph{Question.} Does there exist, for every prime \(p\), a constant \(C_p>0\) such that \[B^{(p)}(1,p,m)\leq C_p\varphi(p^m) \] for every \(m\geq1\)?

\clearpage
\appendix

\section{Signatures modulo \texorpdfstring{\(81\)}{81}}
\label{app:p3-signatures}

For each even weight residue \(k\bmod54\), the permitted signatures are listed below. The entries \(\alpha,\beta\) represent the congruences \[ a_2(f)\equiv_{\mathrm{val}}\alpha\pmod{81}, \qquad a_7(f)\equiv_{\mathrm{val}}\beta\pmod{81}. \] Choose \(\omega\in\overline{\mathbf Z}_3\) with \(\omega^2=2\), and put \[ \mathcal E:= \{0,18,27,36,45,54,63,\, 9\pm27\omega,\,72\pm27\omega\}. \] Both signs are included, so \(\mathcal E\) contains eleven values. In every row, \(\beta\equiv1+7^{k-1}\pmod{81}\).

\begingroup
\small
\renewcommand{\arraystretch}{1.1}
\begin{longtable}{@{}clc@{}}
\caption{The \(159\) permitted signatures modulo \(81\).}
\label{tab:p3-signatures}\\
\toprule
\(k\bmod54\) & Permitted values \(\alpha\) & \(\beta\)\\
\midrule
\endfirsthead
\multicolumn{3}{c}{\tablename\ \thetable\ (continued)}\\
\toprule
\(k\bmod54\) & Permitted values \(\alpha\) & \(\beta\)\\
\midrule
\endhead
\midrule
\multicolumn{3}{r}{Continued on next page}\\
\endfoot
\bottomrule
\endlastfoot
0  & \(12,39,42,69\)                  & \(59\)\\
2  & \(3,24,57,78\)                   & \(8\)\\
4  & \(0,9,18,27,36,45,54,63,72\)     & \(20\)\\
6  & \(6,33,48,75\)                   & \(41\)\\
8  & \(6,33,48,75\)                   & \(17\)\\
10 & \(\mathcal E\)                   & \(56\)\\
12 & \(3,24,57,78\)                   & \(23\)\\
14 & \(12,39,42,69\)                  & \(26\)\\
16 & \(0,9,18,27,36,45,54,63,72\)     & \(11\)\\
18 & \(15,39,42,66\)                  & \(5\)\\
20 & \(24,30,51,57\)                  & \(35\)\\
22 & \(0,9,18,27,36,45,54,63,72\)     & \(47\)\\
24 & \(6,21,60,75\)                   & \(68\)\\
26 & \(21,33,48,60\)                  & \(44\)\\
28 & \(\mathcal E\)                   & \(2\)\\
30 & \(3,30,51,78\)                   & \(50\)\\
32 & \(12,15,66,69\)                  & \(53\)\\
34 & \(0,9,18,27,36,45,54,63,72\)     & \(38\)\\
36 & \(12,15,66,69\)                  & \(32\)\\
38 & \(3,30,51,78\)                   & \(62\)\\
40 & \(0,9,18,27,36,45,54,63,72\)     & \(74\)\\
42 & \(21,33,48,60\)                  & \(14\)\\
44 & \(6,21,60,75\)                   & \(71\)\\
46 & \(\mathcal E\)                   & \(29\)\\
48 & \(24,30,51,57\)                  & \(77\)\\
50 & \(15,39,42,66\)                  & \(80\)\\
52 & \(0,9,18,27,36,45,54,63,72\)     & \(65\)\\
\end{longtable}
\endgroup

\clearpage

\section{Signatures modulo \texorpdfstring{\(49\)}{49}}
\label{app:p7-signatures}

For each even weight residue \(k\bmod42\), the permitted signatures are listed below. The entries \(c,x,y\) are taken in \(\{0,\ldots,6\}\), and represent the congruences \[ a_3(f)\equiv_{\mathrm{val}}c+7x\pmod{49}, \qquad a_{29}(f)\equiv_{\mathrm{val}}2+7y\pmod{49}. \]

\begingroup
\small
\renewcommand{\arraystretch}{1.1}
\begin{longtable}{@{}ccl@{}}
\caption{The \(315\) permitted signatures modulo \(49\).}
\label{tab:p7-signatures}\\
\toprule
\(k\bmod42\) & \(c\) & Permitted pairs \((x,y)\)\\
\midrule
\endfirsthead
\multicolumn{3}{c}{\tablename\ \thetable\ (continued)}\\
\toprule
\(k\bmod42\) & \(c\) & Permitted pairs \((x,y)\)\\
\midrule
\endhead
\midrule
\multicolumn{3}{r}{Continued on the next page}\\
\endfoot
\bottomrule
\endlastfoot
0 & 0 & $\{(0,6), (1,1), (2,0), (3,3), (4,3), (5,0), (6,1)\}$ \\*
 & 1 & $\{(0,3), (1,3), (2,3), (5,3)\}$ \\*
 & 6 & $\{(1,3), (4,3), (5,3), (6,3)\}$ \\
\addlinespace
2 & 0 & $\{(0,0), (1,1), (2,4), (3,2), (4,2), (5,4), (6,1)\}$ \\*
 & 3 & $\{(0,4), (1,4), (3,4), (6,4)\}$ \\*
 & 4 & $\{(0,4), (3,4), (5,4), (6,4)\}$ \\
\addlinespace
4 & 0 & $\{(0,4), (1,1), (2,6), (3,5), (4,5), (5,6), (6,1)\}$ \\*
 & 2 & $\{(0,5), (3,5), (5,5), (6,5)\}$ \\*
 & 5 & $\{(0,5), (1,5), (3,5), (6,5)\}$ \\
\addlinespace
6 & 0 & $\{(0,4), (1,6), (2,5), (3,1), (4,1), (5,5), (6,6)\}$ \\*
 & 1 & $\{(0,6), (1,6), (4,6), (6,6)\}$ \\*
 & 6 & $\{(0,6), (2,6), (5,6), (6,6)\}$ \\
\addlinespace
8 & 0 & $\{(0,0), (1,1), (2,4), (3,2), (4,2), (5,4), (6,1)\}$ \\*
 & 3 & $\{(0,0), (2,0), (5,0), (6,0)\}$ \\*
 & 4 & $\{(0,0), (1,0), (4,0), (6,0)\}$ \\
\addlinespace
10 & 0 & $\{(0,6), (1,3), (2,1), (3,0), (4,0), (5,1), (6,3)\}$ \\*
 & 2 & $\{(1,1), (2,1), (3,1), (6,1)\}$ \\*
 & 5 & $\{(0,1), (3,1), (4,1), (5,1)\}$ \\
\addlinespace
12 & 0 & $\{(0,1), (1,3), (2,2), (3,5), (4,5), (5,2), (6,3)\}$ \\*
 & 1 & $\{(0,2), (1,2), (2,2), (5,2)\}$ \\*
 & 6 & $\{(1,2), (4,2), (5,2), (6,2)\}$ \\
\addlinespace
14 & 0 & $\{(0,6), (1,0), (2,3), (3,1), (4,1), (5,3), (6,0)\}$ \\*
 & 3 & $\{(0,3), (3,3), (4,3), (5,3)\}$ \\*
 & 4 & $\{(1,3), (2,3), (3,3), (6,3)\}$ \\
\addlinespace
16 & 0 & $\{(0,0), (1,4), (2,2), (3,1), (4,1), (5,2), (6,4)\}$ \\*
 & 2 & $\{(1,4), (2,4), (3,4), (6,4)\}$ \\*
 & 5 & $\{(0,4), (3,4), (4,4), (5,4)\}$ \\
\addlinespace
18 & 0 & $\{(0,4), (1,6), (2,5), (3,1), (4,1), (5,5), (6,6)\}$ \\*
 & 1 & $\{(1,5), (3,5), (4,5), (5,5)\}$ \\*
 & 6 & $\{(1,5), (2,5), (3,5), (5,5)\}$ \\
\addlinespace
20 & 0 & $\{(0,4), (1,5), (2,1), (3,6), (4,6), (5,1), (6,5)\}$ \\*
 & 3 & $\{(0,6), (1,6), (2,6), (4,6)\}$ \\*
 & 4 & $\{(2,6), (4,6), (5,6), (6,6)\}$ \\
\addlinespace
22 & 0 & $\{(0,0), (1,4), (2,2), (3,1), (4,1), (5,2), (6,4)\}$ \\*
 & 2 & $\{(0,0), (3,0), (5,0), (6,0)\}$ \\*
 & 5 & $\{(0,0), (1,0), (3,0), (6,0)\}$ \\
\addlinespace
24 & 0 & $\{(0,6), (1,1), (2,0), (3,3), (4,3), (5,0), (6,1)\}$ \\*
 & 1 & $\{(1,1), (2,1), (3,1), (6,1)\}$ \\*
 & 6 & $\{(0,1), (3,1), (4,1), (5,1)\}$ \\
\addlinespace
26 & 0 & $\{(0,1), (1,2), (2,5), (3,3), (4,3), (5,5), (6,2)\}$ \\*
 & 3 & $\{(0,2), (3,2), (4,2), (5,2)\}$ \\*
 & 4 & $\{(1,2), (2,2), (3,2), (6,2)\}$ \\
\addlinespace
28 & 0 & $\{(0,6), (1,3), (2,1), (3,0), (4,0), (5,1), (6,3)\}$ \\*
 & 2 & $\{(0,3), (1,3), (4,3), (6,3)\}$ \\*
 & 5 & $\{(0,3), (2,3), (5,3), (6,3)\}$ \\
\addlinespace
30 & 0 & $\{(0,0), (1,2), (2,1), (3,4), (4,4), (5,1), (6,2)\}$ \\*
 & 1 & $\{(1,4), (2,4), (3,4), (6,4)\}$ \\*
 & 6 & $\{(0,4), (3,4), (4,4), (5,4)\}$ \\
\addlinespace
32 & 0 & $\{(0,4), (1,5), (2,1), (3,6), (4,6), (5,1), (6,5)\}$ \\*
 & 3 & $\{(0,5), (2,5), (5,5), (6,5)\}$ \\*
 & 4 & $\{(0,5), (1,5), (4,5), (6,5)\}$ \\
\addlinespace
34 & 0 & $\{(0,4), (1,1), (2,6), (3,5), (4,5), (5,6), (6,1)\}$ \\*
 & 2 & $\{(2,6), (4,6), (5,6), (6,6)\}$ \\*
 & 5 & $\{(0,6), (1,6), (2,6), (4,6)\}$ \\
\addlinespace
36 & 0 & $\{(0,0), (1,2), (2,1), (3,4), (4,4), (5,1), (6,2)\}$ \\*
 & 1 & $\{(1,0), (3,0), (4,0), (5,0)\}$ \\*
 & 6 & $\{(1,0), (2,0), (3,0), (5,0)\}$ \\
\addlinespace
38 & 0 & $\{(0,6), (1,0), (2,3), (3,1), (4,1), (5,3), (6,0)\}$ \\*
 & 3 & $\{(0,1), (1,1), (3,1), (6,1)\}$ \\*
 & 4 & $\{(0,1), (3,1), (5,1), (6,1)\}$ \\
\addlinespace
40 & 0 & $\{(0,1), (1,5), (2,3), (3,2), (4,2), (5,3), (6,5)\}$ \\*
 & 2 & $\{(0,2), (1,2), (4,2), (6,2)\}$ \\*
 & 5 & $\{(0,2), (2,2), (5,2), (6,2)\}$ \\
\end{longtable}
\endgroup

\end{document}